\documentclass[12pt,a4paper]{amsbook}
\usepackage{fullpage}
\usepackage{amssymb,amsfonts}
\usepackage{graphicx}

\newtheorem{theorem}{Theorem}[chapter]
\newtheorem{lemma}[theorem]{Lemma}
\newtheorem{corollary}[theorem]{Corollary}
\newtheorem{proposition}[theorem]{Proposition}

\theoremstyle{definition}
\newtheorem{definition}[theorem]{Definition}
\newtheorem{remark}[theorem]{Remark}
\newtheorem{example}[theorem]{Example}
\newtheorem{problem}[theorem]{Problem}
\numberwithin{equation}{chapter}

\DeclareMathOperator{\arsh}{arcsinh}
\DeclareMathOperator{\diam}{diam}
\DeclareMathOperator{\card}{card}

\newcommand{\be}{\begin{equation}}
\newcommand{\ee}{\end{equation}}

\newcommand{\abs}[1]{\lvert #1\rvert}
\newcommand{\Rt}{{\mathbb R}^2}
\newcommand{\Rn}{{\mathbb R}^n}
\newcommand{\Rtbar}{\overline{{\mathbb R}^2}}
\newcommand{\Rnbar}{\overline{{\mathbb R}^n}}
\newcommand{\R}{{\mathbb R}}

\newcommand{\C}{{\mathbb C}}
\newcommand{\Cbar}{\overline{\mathbb C}}

\newcommand{\ak}{\tilde{\alpha}}

\newcommand{\psubset}{\varsubsetneq}

\newcommand{\aeq}{\approx}
\newcommand{\ale}{\lesssim}
\newcommand{\age}{\gtrsim}
\newcommand{\mle}{\ll}
\newcommand{\mge}{\gg}
\newcommand{\notcomp}{\lessgtr}
\newcommand{\ip}[2]{<\!\!#1, #2\!\!>}
\newcommand{\MA}{\operatorname{MA}}
\newcommand{\HMA}{\operatorname{HMA}}
\newcommand{\HC}{\operatorname{HC}}
\newcommand{\Cc}{\overline{\mathbb{C}}} %Riemann sphere
\newcommand{\A}{{\mathcal A}}

\newcommand{\es}{{\mathcal S}}
\newcommand{\LU}{{\mathcal{LU}}}

\newcommand{\IR}{{\mathbb R}}
\newcommand{\IN}{{\mathbb N}}
\newcommand{\K}{{\mathcal K}}

\newcommand{\D}{{\mathbb D}}

\newcommand{\real}{{\operatorname{Re}\,}}
\newcommand{\imaginary}{{\operatorname{Im}\,}}

\newcommand{\dist}{{\operatorname{dist}}}

\newcommand{\ds}{\displaystyle}
\newcommand{\Ra}{\Rightarrow}
\newcommand{\ra}{\rightarrow}
\newcommand{\Llra}{\Longleftrightarrow}
\newcommand{\hol}{{\mathcal H}}
\newcommand{\F}{{\mathcal F}}
\def\p{\partial}

\makeindex

\begin{document}
\thispagestyle{empty}
% \begin{figure}
% %\vspace{-10cm}
% \begin{center}
% %\vspace*{-9.5cm}
% \includegraphics{emblem1.ps}
% \end{center}
% \caption{Hi}
% \end{figure}
% \begin{center}
% {\epsfig{file=emblem1.ps, scale=1.0}}
% \end{center}

\vspace*{1cm}
\begin{center}
{\large\bf INEQUALITIES AND GEOMETRY OF HYPERBOLIC-TYPE METRICS,
RADIUS PROBLEMS AND NORM ESTIMATES}% OF UNIVALENT FUNCTIONS}
\end{center}
\vspace*{1cm}
\begin{center}
{\it A THESIS}
\end{center}
%\vspace*{.4cm}
\begin{center}
{\it submitted by}
\end{center}
\vspace*{.6cm}
\centerline{\large\bf SWADESH KUMAR SAHOO}
\vspace*{1.5cm}
\centerline{\it for the award of the degree}
\vspace*{.5cm}
\centerline{\it of}
\vspace*{.5cm}
\centerline{\large\bf DOCTOR OF PHILOSOPHY}
\vspace{2.5cm}
\centerline{\includegraphics{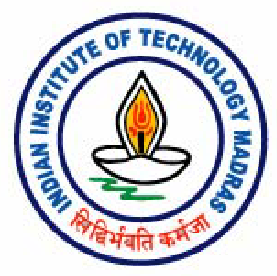}}
%\medskip
\begin{center}\large\bfseries
DEPARTMENT OF MATHEMATICS\\
INDIAN INSTITUTE OF TECHNOLOGY MADRAS\\
CHENNAI -- 600 036
\end{center}
\vspace*{.5cm}
\centerline{\bf DECEMBER 2007}

\bibliographystyle{plain}
\pagestyle{plain}
\pagenumbering{roman}
\newpage
\thispagestyle{empty}

\vspace*{2cm}

\begin{center}
{\large\bf THESIS CERTIFICATE}
\end{center}
\vskip 1cm
This is to certify that the thesis entitled 
{\bf INEQUALITIES AND GEOMETRY OF HYPERBOLIC-TYPE METRICS,
RADIUS PROBLEMS AND NORM ESTIMATES} submitted by 
{\bf Swadesh Kumar Sahoo} to the Indian Institute of Technology
Madras for the award of the degree of Doctor of Philosophy is a 
bonafide record of research work carried out by him under
my supervision. The contents of this thesis, in full or in parts,
have not been submitted to any other Institute or University for
the award of any degree or diploma.

\vskip 3cm\noindent
Madras 600 036
\hfill{Research Guide}\\[5mm]
Date
\hfill{(S. Ponnusamy)}

\newpage
\setcounter{page}{1}
\newpage
\chapter*{\bf ACKNOWLEDGEMENTS}
First and foremost I must thank my thesis supervisor
Professor S. Ponnusamy. His overly enthusiasm, integral view on
research, mission for providing interesting works have
made a deep impression on me. I owe him lots of gratitude for
having me shown this way of research.

I wish to express my warm and sincere thanks to
Professor Peter H\"ast\"o who introduced me to the field of the Apollonian and related
hyperbolic-type metrics; kept an eye on the progress of my work and was
always available when I needed his advices.

It is a great pleasure for me to express my gratitude to Professor R. Parvatham
for inspirations, encouragements and many helps.
I also thank Professors T. Sugawa, M. Vuorinen, and G.P. Youvaraj
for helpful suggestions on my work.
It is my pleasant duty to thank Dr. Zair Ibragimov and Professor D. Minda
for motivating and inspiring discussions on isometry problems and related works.
I must take an opportunity to thank
M. Huang and Professor X. Wang for nice conversations, specially while
preparing results of Chapter 3 of this thesis and related topics.

I would like to thank Professor S.A. Choudum, Head,
Department of Mathematics for
providing desired facilities in the research.
I also thank my Doctoral Committee Members
for their helpful suggestions.

I owe my most sincere gratitude to my university teachers
Professors S. Pattanayak and J. Patel
for guiding and shaping my career as a research fellow.
Special thanks to my university friends Prasant, Binod and Kamal for
their continuous support which helped me a lot in getting interest into
research.

My sincerest thanks to Professors A. Singh and P. Veeramani
for giving me advices whenever I needed.
I must thank Dr. Pravati Sahoo and Dr. Sunanda Naik, who have given their
valuable time, help and advice in the beginning of my career
as a research fellow; and in learning LaTeX as well.
I thank my friends  Bappaditya Bhowmik, Vasudevarao Allu and P. Vasundhra
for helpful discussions within the research group of complex functions theory.

% Without mentioning any name, I would like to thank my friends
% in the Department for all the helpful discussions and support.
I am thankful to my close friends Bhaskar, Jajati, Jitu, Manas, Manoj Raut, Moon,
and Vivek for their better understanding and co-operations.
During my residence life at IIT Madras
many of my friends helped me a lot making
the stay worthwhile; I thank them all.

Last but not the least, all my heartfelt gratitude and love to
my family members for their support, understanding
and cooperation through out my academic career.

\newpage
\chapter*{\bf ABSTRACT}
\noindent{\bf KEYWORDS:}
\parbox[t]{12.9cm}{The Apollonian, the Apollonian inner, the $j$,
the $\lambda$-Apollonian, the Barbilian, Ferrand's,
the inner, the K--P, and the quasihyperbolic metrics;
John, quasi-isotropic, and uniform domains; the comparison property; isometry;
analytic, convex, hypergeometric, starlike, strongly starlike, and univalent functions;
coefficient inequality; pre-Schwarzian norm; integral transforms; and subordinations.}

\vskip1cm
We consider certain inequalities among the Apollonian metric, the Apollonian inner
metric, the $j$ metric and the quasihyperbolic metric. We verify that whether
these inequalities can occur in simply connected planar domains and 
in proper subdomains of $\Rn~(n\ge 2)$. We have seen 
from our verification that most of the cases cannot occur. 
This means that there are many restrictions on domains in which these inequalities 
can occur.
% We observe that some of these inequalities have nice geometric interpretations
% in terms of the domain $G$ in which the metrics are defined and some
% of these inequalities cannot occur in arbitrary domains in $\Rn~(n\ge 2)$.
We also consider two metrics $j$ and $d$, and investigate whether a plane
domain $D\varsubsetneq\C$, for which there exists a constant $c>0$ with
$j(z,w) \le c\, d(z,w)$ for all $z,w \in D$, is a uniform
domain. In particular, we study the case when $d$ is
the $\lambda$-Apollonian metric. We also investigate the question,
whether simply connected quasi-isotropic domains are John disks and conversely.
Isometries of the quasihyperbolic metric, the Ferrand metric and the K--P metric
are also obtained in several specific domains in the complex plane.

In addition to the above, some problems on univalent functions theory
are also solved. We denote by $\es$, the class of normalized univalent
analytic functions defined in the unit disk.
We consider some geometrically motivated 
subclasses, say $\F$, of $\es$. We obtain the largest disk
$|z|<r$ for which $\frac{1}{r}f(rz)\in \F$ whenever $f\in \es$.
We also obtain necessary and sufficient coefficient conditions
for $f$ to be in $\F$. Finally, we present the pre-Schwarzian
norm estimates of functions from $\F$ and that of certain convolution or
integral transforms of functions from $\F$.
{\bf \tableofcontents}
\listoftables
\listoffigures
%\chapter*{\bf GLOSSARY}
% \input{glossary.tex}
\chapter*{\bf ABBREVIATIONS}
\vskip .5cm
\begin{tabular}{lllll}
$\HC$ &&&& Hyperbolic Center\\
$\HMA$ &&&& Hyperbolic Medial Axis\\
K--P &&&& Kulkarni--Pinkall\\
$\LU$ &&&& Locally Univalent\\
$\MA$ &&&& Medial Axis\\
\end{tabular}
\chapter*{\bf NOTATION}
\vskip .5cm
{\bf English Symbols}

\begin{tabular}{lllll}
$\card(A)$ &&&& cardinality of the set $A$\\
$\diam A$ &&&& diameter of the set $A$\\
$\A$ &&&& class of normalized analytic functions in $\D$\\
$A\subset B$ &&&& $A$ is a subset of $B$\\
$A\psubset B$ &&&& $A$ is a proper subset of $B$\\
$B^n(x, r)$ &&&& Euclidean ball of center $x$ with radius $r$\\
$B(x,r)$ &&&& disk with center $x$ and radius $r$\\
$B^n$ &&&& Euclidean unit ball\\
$B_z$ &&&& extremal disk in a domain $G$ centered at $z$\\
$C^n$ &&&& $n$-times continuously differentiable\\
$\C$ &&&& complex plane\\
$\D$ &&&& unit disk\\
$\tilde{d}(x,y)$ &&&& inner metric of the metric $d(x,y)$\\
$\{e_1, e_2, \ldots, e_n\}$ &&&& standard basis of $\Rn$\\
$\|f\|$ &&&& pre-Schwarzian norm of $f$\\
$f*g$ &&&& Hadamard (convolution) product of $f$ and $g$\\
$f\prec g$ &&&& $f$ is subordinate to $g$\\
$\partial G$ &&&& boundary of $G$\\
$\partial_rG$ &&&& set of rectifiably accessible points in $\partial G$\\
$\overline{G}$ &&&& closure of $G$\\
$G^c$ &&&& complement of $G$\\
$\hol$ &&&& class of analytic functions in $\D$\\
$\hol_a$ &&&& class of analytic functions in $\D$ which take
origin into $a\in \C$\\
\end{tabular}

\begin{tabular}{lllll}
$H^n$ &&&& upper half-space\\
$h_D$ &&&& hyperbolic metric defined on $D$\\
$i_z$ &&&& inversion in a circle centered at $z$ with radius $1$\\
$j_G$ &&&& $j$ metric defined on $G$ (it is used in two different meanings)\\
$\K$ &&&& class of convex functions\\
$\K(\alpha)$ &&&& class of convex functions of order
$\alpha,~0\le \alpha \le 1$\\
$k_G$ &&&& quasihyperbolic metric defined on $G$\\
$\hat{K}_z$ &&&& hyperbolic convex hull of the set $\partial B_z\cap \partial G$ in $B_z$\\
$\Rn$ &&&& Euclidean $n$-space\\
$\Rnbar$ &&&& $\Rn\cup\{\infty\}$, the M\"obius space\\
$R_{\zeta}$ &&&& reciprocal of the curvature of $\partial G$
at $\zeta\in \partial G$\\
$S^{n-1}(x, r)$ &&&& Euclidean sphere of center $x$ with radius $r$\\
$S^{n-1}$ &&&& Euclidean unit sphere\\
$\es$ &&&& class of univalent functions\\
$\es^*$ &&&& class of starlike functions\\
$\es^*(\alpha)$ &&&& class of starlike functions of order
$\alpha,~0\le \alpha \le 1$\\
$\es\es^*(\alpha)$ &&&& class of strongly starlike functions of
order $\alpha,~0<\alpha\le 1$.\\
$T_f$ &&&& pre-Schwarzian derivative of $f$\\
$xy$ &&&& line through $x$ and $y$\\
$[x,y]$ &&&& closed line segment between $x$ and $y$\\
$(x,y]$ or $[x,y)$ &&&& half-open (or half-closed) segment between $x$ and $y$\\
$\widehat{xyz}$ &&&& smallest angle between the vectors $x-y$ and $z-y$\\
$x_i$ &&&& $i^{\rm{th}}$ coordinate of $x\in\Rn$\\
$\real z$ &&&& real part of $z$\\
$\imaginary z$ &&&& imaginary part of $z$\\
\end{tabular}

% \newpage
{\bf Greek Symbols}

\begin{tabular}{lllll}
$\alpha_G$ &&&& Apollonian metric defined on the domain $G$\\
$\ak_G$ &&&& Apollonian inner metric defined on the domain $G$\\
$\alpha'_G$ &&&& $\lambda$-Apollonian metric defined on $G$\\
$\bar{\alpha}_G(x;\theta)$ &&&& directed density of the Apollonian metric
at $x$ in the direction $\theta$\\
\end{tabular}

\begin{tabular}{lllll}
$\delta(x)$ &&&& $d(x,\partial G)$, Euclidean distance of $x\in G$
to $\partial G$\\
$d(\gamma)$ &&&& $d$-length of $\gamma$\\
$\ell(\gamma)$ &&&& Euclidean length of $\gamma$\\
$\lambda_D(z_1,z_2)$ &&&&  $\lambda$-length between $z_1$ and $z_2$\\
$\mu_D$ &&&&  K--P metric defined on $D$\\
$\sigma_D$ &&&&  Ferrand's metric defined on $D$\\
\end{tabular}

%{\bf Miscellaneous Symbols}

%\begin{tabular}{lllll}
%$\prec$ &&&& subordination
% \end{tabular}
% \begin{tabular}{lllll}
%\end{tabular}

%{\bf Miscellaneous Symbols}
%
%\begin{tabular}{lllll}
%
%$\{e_1, e_2, \ldots, e_n\}$ &&&& standard basis of $\Rn$\\
%$B^n(x, r)$ &&&& Euclidean ball of center $x$ with radius $r$\\
%$B(x,r)$ &&&& disk with center $x$ and radius $r$\\
%$\tilde{d}(x,y)$ &&&& inner metric of the metric $d(x,y)$\\
%$S^{n-1}(x, r)$ &&&& Euclidean sphere of center $x$ with radius $r$\\
%$\partial_rG$ &&&& set of rectifiably accessible points in $\partial G$\\
%$\|f\|$ &&&& pre-Schwarzian norm of $f$\\
%$\prec$ &&&& subordination\\
%$xy$ &&&& line through $x$ and $y$\\
%$[x,y]$ &&&& closed line segment between $x$ and $y$
%\end{tabular}
%
%\begin{tabular}{lllll}
%$(x,y]$ or $[x,y)$ &&&& half-open (or half-closed) segment between $x$ and $y$\\
%$\widehat{xyz}$ &&&& smallest angle between the vectors $x-y$ and $z-y$\\
%$f*g$ &&&& Hadamard (convolution) product of $f$ and $g$\\
%$A\subset B$ &&&& $A$ is a subset of $B$\\
%$A\psubset B$ &&&& $A$ is a proper subset of $B$\\
%% \end{tabular}
%% \begin{tabular}{lllll}
%$\real z$ &&&& real part of $z$\\
%$\imaginary z$ &&&& imaginary part of $z$\\
%$\es\es^*(\alpha)$ &&&& class of strongly starlike functions of
%order $\alpha,~0<\alpha\le 1$.\\
%\end{tabular}

\newpage
\pagenumbering{arabic}
\setcounter{page}{1}
\chapter{INTRODUCTION}\label{chap1}
The thesis consists of six chapters.
The purpose of this chapter is to give primitive motivations
and backgrounds for the remaining chapters.
In Section \ref{Apo-metric}, we review brief introduction
to the Apollonian metric which is a generalization of
Poincar\'e's model of the hyperbolic metric with some geometric
facts. In Section \ref{Ineq-geom}, we present some
inequalities among certain hyperbolic-type metrics
and their geometric characterizations in terms of domains.
In Section \ref{isometry-sec}, we present isometries of
certain metrics with an aim to investigate the same behavior
for other related metrics in specific domains.
In Section \ref{coeff-subsec}, we give some
motivations to study necessary and sufficient conditions
for functions to be in some subclass of univalent functions
in terms of Taylor's coefficient. In addition,
we introduce the definition of the radius problem
and collect some well-known results with a motivation to
study for some subclasses of univalent analytic functions.
Section \ref{norm-subsec} begins with the pre-Schwarzian norm
of functions from certain well-known classes of locally univalent functions
and ends with some related problems that we solved in last chapter. At last
in Section \ref{summary}, we summarize our investigations
with some conclusion.

The thesis is organized with solutions to a number of problems.
For example, we consider the following problems.

\begin{itemize}
\item Given some sets of inequalities among the Apollonian metric,
the Apollonian inner metric \cite{Ha5}, the $j$ metric and the quasihyperbolic
metric; we ask whether these can occur together in simply connected
planar domains or in general domains of the Euclidean space $\Rn$ ($n\ge 2$)!

\item Can uniform domains be characterized in terms of
inequalities between the $j$ metric and the $\lambda$-Apollonian metric?
What is the relationship between quasi-isotropic domains
and John disks?

\item What are the isometries of the quasihyperbolic metric, the
Ferrand metric and the K--P metric?

\item We identify some subclasses, say $\F$, of the class of
normalized analytic
univalent functions $\es$ and find largest disk
$|z|<r$ for which $\frac{1}{r}f(rz)\in \F$ whenever $f\in \es$.
In addition, we find necessary and sufficient coefficient conditions
for $f$ to be in $\F$.

\item Given some classes of univalent analytic functions, we obtain
the pre-Schwarzian norm estimates of functions from the
given classes as well as that of certain integral or convolution
operators of functions from those classes.
\end{itemize}

% In connections to these five questions, we organize three chapters
% concerning first three questions and four chapters concerning
% the last two questions. However, the main goal of this chapter is to
% give primitive motivations and backgrounds of the problems
% we have considered in the other chapters.
% We observe that the metrics we have
% considered in Chapters \ref{chap2}-\ref{chap4} are of hyperbolic-type
% and theory we have considered in Chapters \ref{chap5} and \ref{chap6}
% is from the classical theory of univalent functions.

% The contents of the thesis is basically from two different
% concepts. Chapters \ref{chap2}-\ref{chap4} are the one
% that consists of isometries of and inequalities among certain
% hyperbolic-type metrics and their geometric descriptions.
% Chapters \ref{chap5}-\ref{chap8} are the another that consists of
% studying certain classes of analytic univalent functions defined in
% the unit disk with an aim to study their radius problems
% and pre-Schwarzian norm estimates.
% But there are several common well-documented results have been established
% between these two concepts and mostly studied by authors like
% Gehring and Osgood \cite{GO}, Hag \cite{Hag01}, Martio and Sarvas \cite{MS},
% Osgood \cite{Osg82} etc.

In the thesis, we say non-empty open connected sets as domains\index{domain}.

\vskip .5cm
\noindent{\bf HYPERBOLIC TYPE METRICS:}
\vskip .2cm
We begin with the definition of a metric as follows.
A metric space is a non-empty set $M$ together with a real valued function
$d:M\times M\to \R$ (called a metric\index{metric}, or
sometimes a distance function) such that for every $x,y,z\in M$
we have the following properties:
\begin{itemize}
\item $d(x,y)\ge 0$, with equality if and only if $x=y$.
\item $d(x,y)=d(y,x)$.
\item $d(x,y)\le d(x,z)+d(z,y)$ (triangle inequality).
\end{itemize}

The Schwarz lemma, named after Hermann Amandus Schwarz,
is a result in complex analysis about holomorphic functions defined on
the unit disk.
% It is clear that Schwarz's lemma and its generalizations became
% widely known in the period of a systematic study of results which
% are  closely connected with the proofs of the Riemann theorem on
% conformal mappings. By the way, the original version of the lemma
% may be found in \cite{Sch}
% where H. A. Schwarz himself discussed the Riemann mapping theorem.
A variant of the Schwarz lemma can be stated that is invariant
under analytic automorphisms on the unit disk, i.e. bijective
holomorphic mappings of the unit disk to itself. This variant
is known as the Schwarz--Pick lemma (after George Pick).
The Schwarz-Pick lemma then essentially gives that a holomorphic
map of the unit disk into itself decreases the distance of points
in the Poincar\'e metric. In early $19^{th}$ century, Poincar\'e
used the unit ball and Lobachevsky used the half space as domains
for their models. By the Riemann mapping theorem we know that
any simply connected proper subdomain of the plane is conformally
equivalent to the unit disk. So it is possible to define the hyperbolic
metric in simply connected subdomains of the complex plane as well.

In contrast to the
situation in the complex plane, the well-known  hyperbolic
metric is defined only in balls and half-spaces in $\Rn$ when $n\ge
3$. Many researchers have proposed metrics that could take the
place of the hyperbolic metric in analysis in higher dimensions.
Probably the most used one is the quasihyperbolic metric
introduced by Gehring and Palka in \cite{GP}. This metric has the slight
disadvantage is that it is not equal to the hyperbolic metric in a
ball.
% , but rather may be off by a multiplicative constant of 2. For
% accurate estimates, for instance asymptotically sharp
% inequalities, this might pose a problem.
Several metrics have also been proposed that are generalizations
of the hyperbolic metric in the sense that they equal the
hyperbolic metric if the domain of definition is a ball or a
half-space. Some examples are the Apollonian metric \cite{Be},
the Ferrand metric \cite{Fe}, the K--P metric
\cite{KP} and Seittenranta's metric \cite{Se}.
Apart from the above
metrics we also consider the $j$ metric and the idea of inner metric in
this thesis. Note that inner metric of the Apollonian metric
is called the {\it Apollonian inner metric}, inner metric of the
$j$ metric is known as the quasihyperbolic metric and
that of Seittenranta's metric is the Ferrand metric.
The common fact for all the above metrics is that they are
defined in some proper subdomain of $\Rn~(n\ge 2)$ and are strongly
affected by the geometry of the boundary of the domain.
Because of this, we sometimes say these  metrics as
hyperbolic-type metrics.
Most of the metrics described here have an invariance
property in the sense of
\begin{equation}\label{isometry}
d_D(x,y)= d_{f(D)}(f(x), f(y)),
\end{equation}
for all $x,y\in D$ and for mappings $f$ belonging to
some fixed class, say the class of conformal maps,
the class of M\"obius maps and the class of similarities.
Here $D\psubset \Rn$ is a metric space with the metric $d$.

% If a metric is of interest, then so are its isometries.
In Chapter \ref{chap4}, we characterize $f$ satisfying
the relation (\ref{isometry}) with respect to some of the
hyperbolic-type path (or conformal path) metrics of the form
\begin{equation}\label{conf-metric}
d_D(x,y)=\inf_{\gamma}\int_{\gamma}p(z)\,|dz|,
\end{equation}
where $p(z)$ is a density function defined on $D$,
$|dz|$ represents integration with respect to path-length, and
the infimum is taken over all rectifiable paths $\gamma$ joining
$x,y\in D$. The {\it hyperbolic metric}\index{hyperbolic metric}\index{metric!hyperbolic}
$h_{\D}$ of the unit disk
$\D$ has the density function $2/(1-|z|^2)$. In this case
the infimum $\gamma$ is attained for the non-euclidean segment
from $x$ to $y$, that is the arc of the circle through
$x$ and $y$ orthogonal to the unit circle. The {\it hyperbolic
metric} $h_D$ of a simply connected plane domain $D$ (other than
$\C$) is obtained by transferring $h_{\D}$ to $h_{D}$ by any
conformal map of $\D$ onto $D$. Indeed, if $f$ maps $\D$ onto
$D$, then the {\it hyperbolic metric}\index{hyperbolic metric}\index{metric!hyperbolic}
of $D$ is defined by
\begin{equation}\label{hyperbolic}
h_D(u,v)=h_{\D}(x,y)\quad \mbox{ for $u=f(x),~v=f(y)$ and $x,y\in \D$.}
\end{equation}
See \cite{BePom,M1,M2} and their references for basic properties
of hyperbolic density.
% If we compare the relations (\ref{isometry}) and (\ref{hyperbolic})
% we see that isometries of the hyperbolic metric of simply connected domains
% are nothing but the conformal maps.

{\em Ferrand's metric} \index{Ferrand's metric} \index{metric!Ferrand's}
\cite{Fe} is defined by replacing the density
function $p(z)$ in (\ref{conf-metric}) with the function
\begin{equation}\label{Ferrand-density}
\sigma_D(z) = \sup_{a,b\in \partial D} \frac{|a-b|}{|a-z| \, |b-z|}.
\end{equation}

The {\em K--P metric} \index{K--P metric}\index{metric!K--P} \cite{KP}
is defined by the density
\begin{equation}\label{KP-metric}
\mu_D(x) = \inf \Big\{ \lambda_B(x)\colon x\in B\subset D,\,
B\ \text{is\ a\ disk\ or\ a\ half-plane}\Big\}.
\end{equation}
Here $\lambda_B$ is the density of the hyperbolic metric in $B$.
Recall that if $B=B(x_0,r)= \{x\in \Rt \colon |x-x_0|<r\}$, then
\[
\lambda_B(x)=\frac{2 r}{r^2-|x-x_0|^2}.
\]

%%%%%%%%%%%%%%%%%%%%%%%%%%%%%%%%%%%%%%%%%%%%
%%%%%%%%%%%%%%%%%%%%%%%%%%%%%%%%%%%%%%%%%%%%
%%%%%%%%%%%%%%%%% SECTION 1 %%%%%%%%%%%%%%%%
%%%%%%%%%%%%%%%%%%%%%%%%%%%%%%%%%%%%%%%%%%%%
%%%%%%%%%%%%%%%%%%%%%%%%%%%%%%%%%%%%%%%%%%%%

\section{The Apollonian Metric}\label{Apo-metric}

The Apollonian metric was first introduced by Barbilian \cite{Ba}
in 1934--35 and then rediscovered by Beardon \cite{Be} in 1998.
This metric has also been considered in \cite{BI,GH,Ro,Se} and
in \cite{Ha2,Ha4,Ha3,Ha5,Ha1,HI,HI2,Ib1,Ib2,Ib3}.
It should also be noted that the same metric has been studied from
a different perspective under the name of the Barbilian metric for
instance in \cite{Ba, Ba2,BR,Bl,Bo,Ke}, cf.\ \cite{BS} for a
historical overview and more references. One interesting historical
point, made in \cite{BS,BS1}, is that  Barbilian himself proposed
the name ``Apollonian metric'' in 1959, which was later independently
coined by Beardon \cite{Be}.
More recently, the Apollonian metric has especially been studied
by H\"ast\"o and Ibragimov in a series of articles, see e.g.
\cite{Ha2}-\cite{Ha1}, \cite{HI,HI2} and \cite{Ib1}-\cite{Ib3}.
An interesting fact
is that the Apollonian metric is also studied with certain group structures \cite{No}.
% and in the same paper Nolder shows that the well-known
% relationship between this metric and the quasihyperbolic metric in domains
% in Iwasawa groups are the same as in Euclidean spaces.

We denote by $\Rnbar = \Rn \cup \{ \infty\}$ the one point
compactification of $\Rn$.
The {\it Apollonian metric}\index{Apollonian metric}\index{metric!Apollonian}
is defined for
$x,y\in G\psubset \Rnbar$ by
$$\alpha_G(x,y) := \sup_{a,b\in\partial G}
\log \frac{|a-y|\,|b-x|}{|a-x|\,|b-y|}
$$
(with the understanding that $|\infty-x|/|\infty-y|=1$).
It is in fact a metric if and only if the complement of $G$ is not
contained in a hyperplane and only a pseudometric otherwise, as was noted
in \cite[Theorem~1.1]{Be}.
Some of the main reasons for the interest in the metric are that
\begin{enumerate}
\item the formula has a very nice geometric interpretation
(see Subsection~\ref{ABA});
\item it is invariant under M\"obius map;
\item it equals the hyperbolic metric in balls and half-spaces;
\item it is monotone: $\alpha_{G_1}(x,y)\le \alpha_{G_2}(x,y)$
whenever $x,y\in G_2\subset G_1$; and
\item it is complete: $\alpha_G(x_n,y)\to \infty$ as
$x_n\to \partial G$ for each $y\in G$.
\end{enumerate}
We next define the Apollonian metric in a different approach called the
{\it Apollonian balls approach}. This is the reason we say the metric
as {\it Apollonian metric}. The name of the balls are called the Apollonian
balls, because they satisfy the definition of Apollonian circles
(see Apollonian circles theorem in \cite{Brannan-geom}).

%%%%%%%%%%%%%%%%%%%%%%%%%%%%%%%%%%%%%%%%%%%%
%%%%%%%%%%%%%%%%%%%%%%%%%%%%%%%%%%%%%%%%%%%%
%%%%%%%%%%%%%%%%% SECTION 2 %%%%%%%%%%%%%%%%
%%%%%%%%%%%%%%%%%%%%%%%%%%%%%%%%%%%%%%%%%%%%
%%%%%%%%%%%%%%%%%%%%%%%%%%%%%%%%%%%%%%%%%%%%

\subsection{The Apollonian balls approach}\label{ABA}

In this subsection we present the Apollonian balls approach which
gives a geometric interpretation of the Apollonian metric.

%Let $G,G'\subset \Rnbar$. A map $f:\,G \to G'$
A map $f:\,\Rnbar \to \Rnbar$ defined by
%\begin{equation}\label{inversion}
$$f(x)=a+\frac{r^2(x-a)}{|x-a|^2},~f(\infty)=a,~f(a)=\infty
$$%\end{equation}
is called an {\it inversion}\index{inversion} in the sphere
%\begin{eqnarray}\label{sphere}
$$S^{n-1}(a,r) := \{z\in\Rn:\,|z-a|=r\}
$$%\nonumber \\
%& = & \{z\in\Rn:\, |z|^2-2\langle z,a \rangle+|a|^2-r^2=0\}
%\end{eqnarray}
for $x,a\in \Rn$ and $r>0$.
%, where $\langle .,. \rangle$ stands usual inner product.
For $x,y\in G\psubset \Rnbar$ we define the following \cite{Se}
$$ q_x:= \sup_{a\in\partial G} \frac{|a-y|}{|a-x|},\
q_y:= \sup_{b\in\partial G}\frac{|b-x|}{|b-y|}.
$$
The numbers $q_x$ and $q_y$ are called the Apollonian parameters of
$x$ and $y$ (with respect to $G$) and
by definition $\alpha_G(x,y)=\log(q_x q_y)$.
This gives an equivalent form of the
Apollonian metric\index{Apollonian metric}\index{metric!Apollonian}.
The balls (in $\Rnbar$!)
$$ B_x:= \Big\{ z\in \Rnbar \colon \frac{|z-x|}{|z-y|}<\frac1{q_x} \Big\}
\quad\text{and}\quad
B_y:= \Big\{ w\in \Rnbar \colon \frac{|w-y|}{|w-x|}<\frac1{q_y} \Big\}, $$
are called the {\em Apollonian balls}\index{Apollonian ball}
about $x$ and $y$, respectively.
Note that these balls are nothing but the Euclidean
balls, see Item 4 below. %with center $x+\frac{x-y}{q_x^2-1}$ and radius $\frac{q_x|x-y|}{q_x^2-1}$.
%We denote the centers and radii of $B_x$ and $B_y$ by
%$x_0$, $y_0$, $r_x$ and $r_y$.
We collect some immediate results regarding these balls; similar results
obviously hold with $x$ and $y$ interchanged.

\begin{enumerate}
\item We have $x\in B_x\subset G$ and $\overline{B_x}\cap \partial G\not= \emptyset$.
\item If $\infty\not\in G$, we have $q_x\ge 1$.
If, moreover, $\infty\not\in \overline{G}$, then $q_x>1$.
\item We have $B_x\cap B_y=\emptyset$. If $G$ is bounded then
$\partial B_x\cap\partial B_y=\emptyset$. In other words, if
$\partial B_x$ intersects $\partial B_y$ then $B_x\cup B_y=G$
(in fact, $\partial B_x=\partial B_y=\partial G$).
\item\label{formula}
If $q_x>1$, $x_0$ denotes the center of $B_x$ and $r_x$ its radius, then
$$x_0=x+\frac{x-y}{q_x^2-1} ~~\mbox{ and }~~ r_x=\frac{q_x|x-y|}{q_x^2-1};
$$
and hence
$$ |x-x_0| = \frac{|x-y|}{q_x^2-1} = \frac{r_x}{q_x}.
$$
\item If $i_x$ and $i_y$ denote the inversions in the spheres $\partial B_x$
and $\partial B_y$ respectively, then $y=i_x(x)=i_y(x)$.
%This implies that if $x_0$ and $y_0$ denote the center of $B_x$ and $B_y$
%respectively, then
%$x$ and $y$ will lie on $[x_0,y_0]$.
%\item If we fix $x,y\in \Rn$ (and forget about $G$) and $B_{x_1}$ and $B_{x_2}$
%are defined as in \itemref{q_xDef} by parameter values $q_{x_1}>q_{x_2}$, then
%$B_{x_1}\subset B_{x_2}$.
\item\label{ABA_delta} We have $q_x-1\le \abs{x-y}/\delta(x) \le q_x+1$.
%and similarly for $\delta(y)$
\end{enumerate}

%%%%%%%%%%%%%%%%%%%%%%%%%%%%%%%%%%%%%%%%%%%%
%%%%%%%%%%%%%%%%%%%%%%%%%%%%%%%%%%%%%%%%%%%%
%%%%%%%%%%%%%%%%% SECTION 3 %%%%%%%%%%%%%%%%
%%%%%%%%%%%%%%%%%%%%%%%%%%%%%%%%%%%%%%%%%%%%
%%%%%%%%%%%%%%%%%%%%%%%%%%%%%%%%%%%%%%%%%%%%

\section{Inequalities and Geometry}\label{Ineq-geom}

As a motivation for the study of inequalities and geometry,
we mention that many inequalities among hyperbolic-type metrics
have been previously studied by well-known authors
and some have geometrical characterizations. For example,
quasidisks and uniform domains characterizations are well
established in terms of inequalities among
hyperbolic-type metrics.

First we bring out the inequalities between the Apollonian metric $\alpha_G$
and the hyperbolic metric $h_G$ in simply connected plane domain $G$.
In this case, the Apollonian metric $\alpha_G$ satisfies the inequality
$\alpha_G\le 2h_G$ (see \cite[Theorem 1.2]{Be}). Furthermore, it is shown in
\cite[Theorem 6.1]{Be} that for bounded convex plane domains the Apollonian metric
satisfies $h_G\le 4\sinh \,[\frac{1}{2}\alpha_G]$, and by considering the example
of the infinite strip $\{x+iy: \,|y|<1\}$, that the best possible constant
in this inequality is at least $\pi$. Later in 1997, Rhodes \cite{Ro}
improved the previous concept to general convex plane domain which says
that $h_G<3.627\sinh\,[\frac{1}{2}\alpha_G]$ and the best possible
constant is at least 3.164. Also due to Beardon \cite{Be},
any bounded plane domain satisfying $\alpha_G\le h_G$ is convex.
On the other hand, Gehring and Hag \cite{GH} established that
any domain $G\psubset \C$ of hyperbolic type is a disk if and only if
$\alpha_G\le h_G$, where disk means the image of unit disk under a
M\"obius map. There are characterization for quasidisks in terms of
the Apollonian metric and the hyperbolic metric as well.
A {\em quasidisk} \index{quasidisk} is the image of a disk under a quasiconformal self
map of $\Rtbar$.
Note that, it is always a non-trivial task to obtain a specific quasiconformal
map when a domain is quasidisk.  For some concrete examples of quasidisks
and corresponding quasiconformal mappings see \cite{GV} (see also \cite{Ge99}).
However, several characterizations \cite{Ge99} of quasidisks have been obtained from
which it became convenient to say whether a domain is quasidisk.
For example, it is well-known that simply connected uniform domains are
quasidisks and it is not always difficult to present a proof for a
simply connected domain to be uniform. There are also
characterizations of quasidisks in terms of hyperbolic-type metrics.
Due to Gehring and Hag \cite{GH}, a simply connected
domain $G\psubset \Rtbar$ is a quasidisk if and only if there is a
constant $c$ such that $h_G(z_1,z_2)\le c\,\alpha_G(z_1,z_2)$ for
$z_1,z_2\in G$. For several other interesting characterizations of quasidisks,
see \cite{Ge99}.

Let $\gamma\colon [0,1]\to G\subset \Rn$ be a path. If $d$ is a metric in $G$,
then the $d$-length
of $\gamma$ is defined by
$$ d(\gamma):= \sup \sum_{i=0}^{k-1} d(\gamma(t_i), \gamma(t_{i+1})), $$
where the supremum is taken over all $k~(<\infty)$ and all sequences $\{t_i\}$ satisfying
$0=t_0<t_1<\cdots<t_k=1$. All paths in the thesis are assumed to be rectifiable,
that is, to have finite Euclidean length. The {\it inner metric}\index{inner metric}\index{metric!inner} of $d$ is defined
by %the formula
$$ \tilde{d}(x,y) := \inf_{\gamma} d(\gamma),$$
where the infimum is taken over all paths $\gamma$ connecting $x$ and $y$ in
$G$. %The term ``intrinsic metric'' is also used for the inner metric.
By repeated use of the triangle inequality it follows that $d\le \tilde{d}$.
We denote the inner metric of the Apollonian metric by $\ak_G$
and call it the {\it Apollonian inner metric}\index{Apollonian inner metric}
\index{metric!Apollonian inner}.
Since $\alpha_G$ is a metric except
in a few domains, it is not reasonable to expect that $\ak_G$
is a metric in every domain. In fact,
$\ak_G$ is a metric if and only if the complement of $G$
is not contained in an $(n-2)$-dimensional hyperplane in $\Rn$
\cite[Theorem~1.2]{Ha5}. In the same paper H\"ast\"o established an
explicit integral formula for the Apollonian inner metric
\cite[Theorem 1.4]{Ha5}, and proved that for most domains there exists
a geodesic connecting two arbitrary points \cite[Theorem 1.5]{Ha5}.

We now define the other two metrics
that we consider in Chapter \ref{chap2}.
Let $G\psubset\Rn$ be a domain and $x,y\in G$.

The metric\index{$j$  metric}\index{metric!$j$} $j_G$
is defined by
\begin{equation}\label{Vu2-metric}
j_G(x,y) := \log\bigg( 1 + \frac{\abs{x-y}}
{\min\{ \dist (x,\partial G), \dist (y,\partial G)\}} \bigg),
\end{equation}
see \cite{Vu2}. In a slightly different form of this metric
was defined in \cite{GO}. Indeed, the $j$ metric from \cite{GO} is defined by
\begin{equation}\label{GO-metric}
j_{G}(x,y)=\log\left(1+\frac{|x-y|} {\dist (x,\partial G)}\right)\left(1+\frac{|x-y|}
{\dist (y,\partial G)}\right) \quad \mbox{ for all $x,y\in G$}
\end{equation}
(see also Chapter \ref{chap3} of the present thesis for this definition in plane domains).
Note that both the metrics defined by (\ref{Vu2-metric}) and (\ref{GO-metric})
are equivalent.
%Therefore, our results will remain unchanged, but up to a constant,
%whichever $j$ is chosen. This is the reason we keep the same notation $j$ for
%the both. However,
In further discussion of the current chapter and in Chapter \ref{chap2}
we use the notation $j$ defined by (\ref{Vu2-metric}); and
in Chapter \ref{chap3} we use the notation $j$ defined by (\ref{GO-metric}).
The metric $j_G$ is complete, monotone and invariant under
similarity maps (see e.g. \cite{Ge3}). It is not M\"obius invariant, but
a M\"obius quasi-invariant \cite[Theorem 4]{GO}.

The {\it quasihyperbolic metric}\index{quasihyperbolic metric}
\index{metric!quasihyperbolic} from \cite{GP} is defined by
$$ k_G(x,y):= \inf_\gamma \int_\gamma \frac{\abs{dz}}{\dist (z,\partial G)},
$$
where the infimum is taken over all paths $\gamma$ joining $x$ and
$y$  in $G$. We recall from \cite{GO} that the quasihyperbolic geodesic
segment exists between each pair of points $x,y\in G$, i.e. the length
minimizing curve $\gamma$ joining $x$ and $y$ for $k_G(x,y)$ exists.
The metric $k_G$ is complete, monotone and changes at most by
the factor 2 under a M\"obius map (i.e. it is
M\"obius quasi-invariant\index{M\"obius quasi-invariant},
see \cite{Ge3,GO,GP}).
Note that the quasihyperbolic metric is the inner
metric of the $j_G$ metric, see for instance \cite[Lemma 5.3]{Ha2}, and hence
we have the fundamental inequality
$$j_G(x,y)\le k_G(x,y) \quad \mbox{ for all $x,y\in G$}
$$
which is used several times in the present thesis.

In order to describe further discussion on  inequalities and geometry
we define some relations on the set of metrics.
\begin{definition}\label{relation-def} Let $d$ and $d'$ be metrics on $G$.
\begin{enumerate}
\item We write $d \ale d'$ if there exists a constant $K>0$ such that
$d \le K d'$. Similarly for the relation $d\age d'$.
\item We write $d \aeq d'$ if $d\ale d'$ and $d\age d'$.
\item We write $d \mle d'$ if $d\ale d'$ and $d\not\age d'$.
\item We write $d \notcomp d'$ if $d\not \ale d' $ and $d \not \age d'$.
\end{enumerate}
\end{definition}

Recall that $\alpha_G\le 2j_G$ in every domain $G\psubset \Rn$ by
\cite[Theorem~3.2]{Be}. Also, it was shown in  \cite[Theorem~4.2]{Se}
that if $G\psubset \Rn$ is convex, then $j_G\le \alpha_G$.
So $\alpha_G\aeq j_G$ in convex domains.
The condition $\alpha_G\aeq j_G$
% was shown in \cite[Theorem~1.3]{Ha1}
% to be equivalent with the
% complement of $G$ being thick in the sense of \cite{VVW}, which means that
% this inequality
is also connected to various interesting properties, see for example
\cite[Theorem~1.3]{Ha1}.
%\cite[Section~1.4]{VVW} for details.
In the paper \cite{Ha1} the term
{\em comparison property}\index{comparison property} was introduced
for the relation $\alpha_G\aeq j_G$.
In \cite{Ha1}, H\"ast\"o has given a geometrical characterization, in
terms of an interior double ball condition, of those domains satisfying
the comparison property.
Additionally, the inequalities
$\ak_G\aeq k_G$, $\alpha_G\aeq \ak_G$ and $\alpha_G\aeq k_G$, which have
been called quasi-isotropy\index{quasi-isotropy},
Apollonian quasiconvexity\index{Apollonian quasiconvexity} and
$A$-uniformity\index{$A$-uniformity},
respectively, have some nice geometric interpretations and
have been considered in \cite{Ha2,Ha4,Ha3}.

We now recall the definition of uniform domains introduced by
Martio and Sarvas in
\cite[2.12]{MS} (see also \cite[(1.1)]{GO} and Definition \ref{unifDef} in
Chapter \ref{chap2} for equivalent formulations).
\begin{definition}\label{unifDef1}
A domain $G$ is called a {\it uniform domain}\index{uniform domain}
\index{domain!uniform} provided there
exists a constant $c$ with the property that each pair of points
$z_1$, $z_2\in G$ can be joined by a path $\gamma \subset
G$ satisfying
$$\ell(\gamma) \le c\, \abs{z_1-z_2} ~\mbox{ and }~
\min_{j=1,2}\ell (\gamma[z_j, z])\leq c\, \dist(z,\partial G)
~\mbox{ for all $z\in \gamma $}.
$$
Here $\gamma[z_j, z]$ denotes the part of $\gamma$ between $z_j$
and $z$.
\end{definition}
Here we remark that the first condition is called the
quasiconvexity condition and the second one is called the
double cone (John) condition. More precisely, if we remove
the first condition from Definition \ref{unifDef1}, then we call
the domain as {\em John domain}\index{John domain}\index{domain!John}.
Simply connected John domains are
called the {\em John disks} \cite{NV}\index{John disk}.
Similarly, if we omit the second
condition from the definition the property would be meant for
the quasiconvex domain. Thus we conclude that every uniform domain
is John domain as well as quasiconvex domain.

Since the quasihyperbolic metric $k_G$ is the inner metric of
the metric $j_G$, we have $j_G\le k_G$ for any domain $G\psubset \Rn$.
On the other hand, due to Gehring and Osgood (see \cite[Corollary 1]{GO}),
a domain $G\psubset \Rn$ is uniform\index{uniform domain}\index{domain!uniform}
if and only if $k_G\ale j_G$ holds.
Although the inequality in \cite{GO} was stated in the form $k_G\le c\,j_G+d$,
it has been proved later that these two forms are equivalent (see for instance
\cite{GH99,Ha2} and \cite[2.50(2)]{Vu2}).
This condition is also equivalent to
$\ak_G\ale j_G$, see \cite[Theorem 1.2]{HPWS}. Thus we have a geometric
characterization of domains satisfying these inequalities as well.

The above observations motivate us to study certain inequalities
among the Apollonian metric $\alpha_G$, the Apollonian inner metric $\ak_G$,
the $j_G$ metric and the quasihyperbolic metric $k_G$;
and their geometric meaning which helps us to form Chapter \ref{chap2}
in this thesis.

The brief idea of Chapter \ref{chap2} is as follows:
Let us first of all note that the following inequalities hold in every
domain $G\psubset \Rn$:
\be\label{genIneqs}
\alpha_G\ale j_G \ale k_G\ \text{ and }\ \alpha_G\ale \ak_G \ale k_G.\ee
The first two are from \cite[Theorem~3.2]{Be} and the
second two from \cite[Remark~5.2 and Corollary~5.4]{Ha2}.
We see that of the four metrics to be considered, the Apollonian
is the smallest and the quasihyperbolic is the largest.
We will undertake a systematic study of which of the inequalities in (\ref{genIneqs})
%$\alpha_G\ale j_G \ale k_G$ and $\alpha_G\ale \ak_G \ale k_G$
can hold
in the strong form with $\mle$ and which of the relations $j_G\mle \ak_G$,
$j_G\aeq \ak_G$ and $j_G\mge \ak_G$ can hold.
% Thus we are led to twelve
% inequalities, which are given along with the results in Table~\ref{table1},
% where we have indicated in column A whether the inequality can hold in
% simply connected planar domains and in column B whether it can hold in
% an arbitrary proper subdomains of $\Rn$. From the table we see
% that most of the cases cannot occur, which means
% that there are many restrictions on which inequalities can occur together.

%%%%%%%%%%%%%%%%%%%%%%%%%%%%%%%%%%%%%%%%%%%%
%%%%%%%%%%%%%%%%%%%%%%%%%%%%%%%%%%%%%%%%%%%%
%%%%%%%%%%%%%%%%% SECTION 4 %%%%%%%%%%%%%%%%
%%%%%%%%%%%%%%%%%%%%%%%%%%%%%%%%%%%%%%%%%%%%
%%%%%%%%%%%%%%%%%%%%%%%%%%%%%%%%%%%%%%%%%%%%

\section{Isometries of Hyperbolic-type Metrics}\label{isometry-sec}

If a metric is of interest, then so are its isometries.
By the {\em isometry problem} \index{isometry problem} for the
metric $d$ we mean
characterizing mappings $f:\, D\to \Rt$ which satisfy (\ref{isometry})
for all $x,y\in D$. Although not a part of our result but as a
motivation; before going to start on path metrics,
we give an overview on isometries of the Apollonian metric.
Isometries of the $j$ metric were studied by H\"ast\"o, Ibragimov
and Lind\'en \cite{Ha11,HIL}.

We recall that the Apollonian metric was introduced by
Beardon in 1998. However, it remained an open question
that what are all its isometries. Beardon first raised this
question and studied whether the Apollonian isometries are
only M\"obius maps. He proved that conformal mappings of plane domains,
whose boundary is a compact subset of the extended negative real axis
that contains at least three points, which are Apollonian isometries
are indeed M\"obius mappings, \cite[Theorem 1.3]{Be}.
In \cite{HI}, H\"ast\"o and Ibragimov
have established that this is true in the case of all open
sets with regular boundary. On the other hand, in \cite{HI2},
the same authors have found Apollonian isometries of plane domains
but without assumption on regularity of the boundary. They proved
that M\"obius mappings in plane domain are the isometries of
the Apollonian metric as long as the domain has at least three
boundary points. We note that the last two observations on isometries
of the Apollonian metric are very recent ones. However, a few years
ago, Gehring and Hag \cite{GH} have considered only the sense
preserving Apollonian isometries in disks and showed that they are
always the restriction of M\"obius maps. Independently, in 2003,
H\"ast\"o has generalized the above idea of Gehring and Hag
(see \cite{Ha2,Ha4}) to $\Rn$ as well.

Next we keep an eye on the study of isometries of hyperbolic-type
conformal metrics. There are three steps in characterizing isometries
of a conformal metric by showing that they are
$$
\text{(1)~ conformal; \quad (2)~ M\"obius; \quad  (3) similarities.}
$$
% \begin{enumerate}
% \item[(1)] to show that they are conformal or anti-conformal;
% \item[(2)] to show that they are M\"obius; and
% \item[(3)] to show that they are similarities.
% \end{enumerate}
The step (1) has been carried out by Martin and Osgood
\cite[Theorem~2.6]{MO} for arbitrary domains assuming only that
the density is continuous, so there is no more work to do there.
Note that step (2) is trivial in dimensions $3$ and higher
(because all conformal maps are nothing but M\"obius maps), and that
step (3) is not relevant for M\"obius invariant metrics like the K--P
metric and Ferrand's metric.

Among the conformal metrics we mainly concentrate on the
quasihyperbolic metric, the K--P metric and the Ferrand metric.
The work of H\"ast\"o \cite{Ha10} on steps (2) and (3) is
very recent one. He proved that, except for the trivial case
of a half-plane, the quasihyperbolic isometries are
similarities with some assumptions of smoothness
on the domain boundary. Indeed, for example, he showed that
a quasihyperbolic isometry which is also a M\"obius transformation is a
similarity provided the domain is a $C^{1}$ domain which is
not a half-plane (see \cite[Proposition 2.2]{Ha10}).
Here {\em $C^{k}$ domain} \index{$C^{k}$ domain} \index{domain!$C^{k}$}
means its boundary is
locally the graph of a $C^{k}$ function.
From \cite[Theorem~2.8]{MO} we know that every quasihyperbolic
isometry is conformal. In dimensions three and higher all conformal
mappings are M\"obius. So certainly, H\"ast\"o could able to
generalize his above result to higher dimensions as well. He has also
generalized the results to $C^3$ domains. Regarding step (2), the reader
is referred to \cite[Section 4]{Ha10}. The work by Herron, Ibragimov and
Minda \cite{HIM} shows that all isometries of the K--P metric
are M\"obius mappings except in simply and doubly connected domains.

The above ideas of isometries encourage us to study isometries in other
specific domains and for the Ferrand metric as well. This leads to form a
survey article along with some new results in Chapter \ref{chap4}.

\vskip .5cm
\noindent{\bf UNIVALENT FUNCTIONS THEORY:}
\vskip .2cm

Univalent function theory is a classical area in the branch of complex analysis.
We know that a {\em function} \index{function} is a
rule of correspondence between two sets such that there is
a unique element in the second set assigned to each element in the first set.
A function on a domain is called
{\it univalent}\index{univalent function}\index{function!univalent} if it is one-to-one.
For example, any function $\phi_a(z)$
of the unit disk to itself, defined by
\begin{equation}\label{disk-to-disk}
\phi_a(z)=\frac{z+a}{1+\bar{a}z}
\end{equation}
is univalent, where $|a|<1$.
Various other terms are used for this concept,
e.g. {\it simple}, {\it schlicht} (the German word for simple).
We are mainly interested in univalent functions that are also
analytic in the given domain.
% One of the basic properties of the above class is that if a
% function is univalent from one domain onto another, then
% the derivative of that function is never zero, the function
% is also invertible and the inverse function is analytic as well.
% The sum of two univalent functions need not be univalent, as can be
% seen from the example that the sum of the two univalent functions
% $z(1-z)^{-1}$ and $z(1+iz)^{-1}$ has a derivative which vanishes at
% $\frac{1}{2}(1+i)$. However, the class of univalent functions
% is preserved under a number of elementary transformations,
% see \cite[page 27]{Duren:univ}.  We will be discussing on
% disk automorphism preservation in Section \ref{norm-subsec}
% which is the mail tool in developing the study of pre-Schwarzian
% and Schwarzian norms.
We refer the standard books by Duren \cite{Duren:univ}, Goodman \cite{Go}
%Lehto \cite{Leh:univ}
and Pommerenke \cite{Pom:univ} for this theory.

The theory of univalent functions is so vast and complicated
so that certain simplifying assumptions are necessary.
If $g(z)$ is analytic in the unit disk $\D$, it has the Taylor series expansion
$$g(z)=b_0+b_1z+b_2z^2+\cdots =\sum_{n=0}^\infty b_nz^n.
$$
We observe that if $g(z)$ is univalent in $\D$ then the function
$f(z)=(g(z)-b_0)/b_1$ is also univalent in $\D$ and conversely.
Setting $b_n/b_1=a_n$ in the above expansion of $g$
we arrive at the normalized form %\index{normalized function}\index{function!normalized}
\begin{equation}\label{series-exp}
f(z)=z+a_2z^2+\cdots =z+\sum_{n=2}^\infty a_nz^n.
\end{equation}
Here we note that, the above normalized form of the function $f$ satisfies
the relation $f(0)=0=f'(0)-1$.
The well-known example in this class is
{\it the Koebe function}\index{Koebe function}\index{function!Koebe},
$k(z)$, defined by
$$k(z)=\frac{z}{(1-z)^2}=z+\sum_{n=2}^\infty nz^n
$$
which is an extremal function for many subclasses of the class of
univalent functions.

It is natural to ask the following two questions about the representation (\ref{series-exp}).
\begin{enumerate}
\item[(1)] Given the sequence of coefficients $\{a_n\}$, how does
it influence some geometric properties of $f(z)$?
\item[(2)] Given some properties of $f(z)$, how does this property
affect the coefficients in (\ref{series-exp})?
\end{enumerate}

We denote by $\A$, the class of analytic functions $f$ in $\D$ of the form
(\ref{series-exp}) and $\es$ denotes the class of all functions $f\in \A$
that are univalent in $\D$.
Functions in the class $\es$ have a nice
geometric property that the range of the function contains a
disk of radius at most 1/4, because an extremal function
$k(z)$ maps the unit disk
onto the whole plane except a slit along the negative real axis
from $-1/4$ to $\infty$. This result is known as Koebe's one-quarter
theorem.

Many authors have studied a number of subclasses of univalent
functions as well. Among those, the class of convex and starlike functions are
the most popular and interesting because of their
simple geometric properties.

A domain $D\subset \C$ is said to be
{\it starlike}\index{starlike domain}\index{domain!starlike}
with respect to
a point $z_0\in D$ if the line segment joining $z_0$ to every other
point $z\in D$ lies entirely in $D$.
A function $f\in \es$ is said to be a
{\it starlike function}\index{starlike function}\index{function!starlike}
if $f(\D)$ is a domain starlike with respect to origin.
The class of all starlike functions is denoted by $\es^*$.
A typical example of a function in this class is the Koebe function
and as an extremal function, the range of every function
$f\in\es^*$ contains the disk $|w|<\frac{1}{4}$.

A domain $D\subset \C$ is said to be
{\it convex}\index{convex domain}\index{domain!convex} if it is starlike
with respect to each of its points; that is, if the line segment
joining any two points of $D$ lies completely in $D$.
Similar to starlike functions
a function $f\in \es$ is said to be convex\index{convex function}
\index{function!convex}
if $f(\D)$ is a convex domain. The class of all convex functions is
denoted by $\K$. The function
%\begin{equation}\label{cvx-extremal}
$$\ell(z)=\frac{z}{1-z}=\sum_{n=1}^\infty z^n
$$
%\end{equation}
belong to the class $\K$ and maps $\D$ onto the half-plane
$\real \{w\}>-\frac12$. This function plays a role of extremal function
for many problems in the class $\K$ as well. The range of every function
$f\in\K$ contains the disk $|w|<\frac{1}{2}$. We note that $z\ell'(z)=k(z)$.

An analytic description of starlike functions is that
$$f\in \es^* ~\mbox{ if and only if }~ \real \left(\frac{zf'(z)}{f(z)}\right)> 0
~~ \mbox{ for $z\in \D$}.
$$
Convex functions have also a similar description:
$$f\in\K ~\mbox{ if and only if }~ \real \left(1+\frac{zf''(z)}{f'(z)}\right)> 0
~~ \mbox{ for $z\in \D$}.
$$
The two preceding descriptions reveal an interesting close analytic
characterization between convex and starlike functions. This says that
$f(z)\in\K$ if and only if $zf'(z)\in \es^*$. This was first observed by
Alexander \cite{Alex} in 1915 and then onwards the result is
known as {\it Alexander's theorem}\index{Alexander's theorem}.
In view of this, the one-to-one correspondence between $\K$ and $\es^*$ is
given by the well-known {\em Alexander transform} \index{Alexander's transform}
\index{transform!Alexander's} defined by
\begin{equation}\label{Alex-trans}
J[f](z)=\int_0^z\frac{f(t)}{t}\,dt.
\end{equation}
That is, $J[f]$ is convex if and only if $f$ is starlike.
Here we remark that the Alexander transform (\ref{Alex-trans}),
in general, does not take an univalent function into another univalent
function (see \cite[Theorem 8.11]{Duren:univ}).
The above properties of Alexander's transform motivate to
study several other generalized transforms in
the theory of univalent functions.

% Because of the independent interest in the classes of convex and starlike
% functions, and as a motivation,
% authors have started looking forward to study several other subclasses
% of univalent functions which have simple geometric properties
% and found many similar and other interesting
% results like in the class of convex and starlike functions.

%%%%%%%%%%%%%%%%%%%%%%%%%%%%%%%%%%%%%%%%%%%%
%%%%%%%%%%%%%%%%%%%%%%%%%%%%%%%%%%%%%%%%%%%%
%%%%%%%%%%%%%%%%% SECTION 5 %%%%%%%%%%%%%%%%
%%%%%%%%%%%%%%%%%%%%%%%%%%%%%%%%%%%%%%%%%%%%
%%%%%%%%%%%%%%%%%%%%%%%%%%%%%%%%%%%%%%%%%%%%

\section{Coefficient Conditions and Radii Problems}\label{coeff-subsec}

We begin with the following conjecture \cite{Bie1,Bie2}:

{\bf Bieberbach's Conjecture.} If $f\in\es$, then for each $n\ge 2$ we have $|a_n|\le n$.

% This conjecture is often referred to in the literature as
% the {\it Bieberbach conjecture}\index{Bieberbach's conjecture},
% as he stated it in a footnote in 1916
% \cite{Bie1,Bie2}.
The conjecture was unsolved for about 70 years
although it had been proved in several special cases $n=2,3,4,5,6$
\cite[page 24]{Pom:univ} and many other subclasses of $\es$.
% Bieberbach himself proved it for $n=2$ and stated
% the conjecture \cite{Bie1,Bie2}. The second result, $|a_3|\le 3$,
% was proved by L\"owner in 1923 \cite{Lo23}. For a short history up to
% the result $|a_6|\le 6$, we refer the standard book by Pommerenke
% \cite[page 24]{Pom:univ}.
But finally, Louis de Branges \cite{de-Bra} settled it in the affirmative in 1985.
% Here one could ask the following question: {\it how about the
% analog of Bieberbach conjecture to the classes $\es^*$ and $\K$?}
% We note that the Bieberbach conjecture for the class $\es^*$ was
% proved by Nevanlinna \cite{Nev20} in 1920. He showed that
% the coefficients of each function $f\in \es^*$ satisfy $|a_n|\le n$
% for all $n\ge 2$ and equality holds for a rotation of the Koebe
% function. As a consequence, due to Alexander's theorem,
% the Bieberbach conjecture for the class $\K$ says that
% if $f\in\K$ then $|a_n|\le 1$ for $n\ge 2$ and equality holds
% for the function $\ell(z)$ defined by (\ref{cvx-extremal}).
%
% Not only in the form (\ref{series-exp}), but also
% there are several other assumptions on functions have
% been studied. For example, the class of functions $f\in\es$
% with $a_2=0$ became also popular. The function $f(z)=z/(1-z^2)$
% is in the later form. In this form the necessary condition
% for $f$ to be in $\es$ is known as $|a_3|\le 1$.
% As a motivation, we have been considering a
% reformulation
% of the functions from the class $\es$
% %special assumptions on functions in Chapter \ref{chap5}
% that recently developed by Obradovi\'c and Ponnusamy
% (see e.g. \cite{OP-05prea,OP-05prea1}).
Several other type of coefficient estimates, such as
sufficient conditions for $f$ to be in several subclasses of $\es$, are also
well-established. For example (see Goodman \cite{Go57})
if $f(z)=z+\sum_{n=2}^\infty a_nz^n$ satisfies
$\sum_{n=2}^\infty n|a_n|\le 1$ then $f\in \es^*$.

%\section{Radius Problems}\label{radius-subsec}
The problem of estimating the radius of various classes of univalent
functions has attracted a certain number of mathematicians
involved in geometric function theory.
For a systematic survey of radius problems, we refer to
\cite[Chapter 13]{Go}.

Let $\mathcal F$ and $\mathcal G$ be two subclasses of $\mathcal A$.
If for every $f\in {\mathcal F}$, $r^{-1} f(rz)\in\mathcal{G}$ for
$r\leq r_0$, and $r_0$ is the largest number for which this holds,
then we say that $r_0$ is the $\mathcal{G}$ radius (or the radius of
the property\index{radius property} connected to $\mathcal G$)
in $\mathcal{F}$. This implies that
if $r>r_0$, there is at least one $f\in {\mathcal F}$ such that
$r^{-1} f(rz)\not\in\mathcal{G}$. Here our main aim
is to obtain $r_0$. There are many results of this type that have been studied
in the theory of univalent functions. For example, the radius of
convexity for the class $\es$ is known to be $2-\sqrt{3}$ and  that
of starlikeness for the same class is $\tanh \pi/4\approx 0.656$.
% the root of the equation
% $$\ln \frac{1+r}{1-r}=\frac{\pi}{2},
% $$
% namely,
%$r_0=\tanh \pi/4\approx 0.656$.

As a motivation of the discussion in Section \ref{coeff-subsec},
we form Chapter \ref{chap5}.

An outline of Chapter \ref{chap5} is as follows:
We define a subclass  $\es_p(\alpha)$,
$-1\le \alpha \le 1$, of starlike functions in the following way \cite{Ro-91}:
$${\es}_p(\alpha) = \left \{f\in {\mathcal S}:\,
\left |\frac{zf'(z)}{f(z)} -1\right |\leq {\rm Re}\,
\frac{zf'(z)}{f(z)}-\alpha, \quad z\in \D  \right \}.
$$
Geometrically, $f\in{\es}_p(\alpha)$ if and only if the domain
values of $zf'(z)/f(z)$, $z\in \D$, is the parabolic region
$({\rm Im}\,w)^2\leq (1-\alpha)[2{\rm Re}\,w -(1+\alpha)]$.
We determine necessary and sufficient
coefficient conditions for certain class of functions to be in $\es_p(\alpha)$.
Also, radius properties are considered for
$\es_p(\alpha)$-class in the class $\mathcal{S}$. We consider
another subclass of the class of univalent functions which has recent interest as follows.
A function $f\in \mathcal{A}$ is
said to be in $\mathcal{U(\lambda,\mu)}$ if
$$ \left| f'(z)\left(\frac{z}{f(z)} \right)^{\mu+1}-1\right| \le
\lambda \quad (|z|<1)
$$
for some $\lambda\ge 0$ and $\mu>-1$.
We find disks $|z|<r:=r(\lambda,\mu)$ for which
$\frac{1}{r}f(rz)\in \mathcal{U(\lambda,\mu)}$ whenever $f\in
\mathcal{S}$.
In addition to a number of new results,
we also present several new sufficient
conditions for $f$ to be in the class $\mathcal{U(\lambda,\mu)}$.

%%%%%%%%%%%%%%%%%%%%%%%%%%%%%%%%%%%%%%%%%%%%
%%%%%%%%%%%%%%%%%%%%%%%%%%%%%%%%%%%%%%%%%%%%
%%%%%%%%%%%%%%%%% SECTION 6 %%%%%%%%%%%%%%%%
%%%%%%%%%%%%%%%%%%%%%%%%%%%%%%%%%%%%%%%%%%%%
%%%%%%%%%%%%%%%%%%%%%%%%%%%%%%%%%%%%%%%%%%%%

\section{Pre-Schwarzian Norm}\label{norm-subsec}

We recall that the class $\es$ is preserved under
disk automorphism\index{disk automorphism}
(also called the {\it Koebe transform})\index{Koebe transform}
\index{transform!Koebe}.
More precisely this means that if $f\in \es$ and
$$g(z)=\frac{f(\phi_a(z))-f(a)}
{(1-|a|^2)f'(a)}
=z+\left(\frac{1}{2}(1-|a|^2)\frac{f''(a)}{f'(a)}-\bar{a}\right)z^2+\cdots,
$$
then $g$ is also in the class $\es$, where $\phi_a(z)$ is defined by
(\ref{disk-to-disk}).
The derivative quantity $T_f:=f''/f'$ is called
the {\em pre-Schwarzian derivative} \index{pre-Schwarzian derivative} of $f$
or {\em logarithmic derivative} \index{logarithmic derivative} of $f'$.
By Bieberbach's theorem and an easy simplification,
we obtain
$$\left|(1-|a|^2)\frac{f''(a)}{f'(a)}-2\bar{a}\right|\le 4,
$$
where equality holds for a suitable rotation of the Koebe function.
Consequently, one has
$$\sup_{|z|<1}(1-|z|^2)\left|\frac{f''(z)}{f'(z)}\right|\le 6
$$
for $f\in \es$.
This can also be seen from the result of Martio and Sarvas
\cite[Item 4.6]{MS} and Osgood \cite[Lemma 1]{Osg82},
which says about an upper bound property of the
pre-Schwarzian derivative in terms of the quasihyperbolic density
in a proper subdomain of the complex plane.
The inequality is sharp, that is, we cannot
replace the constant $6$ by a smaller number and it can be
seen by considering the Koebe function.
This motivates us to study the quantity
\begin{equation}\label{norm}
\|f\|=\sup_{|z|<1}(1-|z|^2)\left|\frac{f''(z)}{f'(z)}\right|,
\end{equation}
in the theory of univalent functions. We usually say this
quantity as {\em pre-Schwarzian norm} \index{pre-Schwarzian norm} of the function $f$.
In general, the assumption on $f$
can be restricted to {\em locally univalent} \index{locally univalent}
\index{function!locally univalent} functions, namely,
$\LU:=\{f\in\A:\, f'(z)\ne0,~ z\in\D\}.$ We may regard $\LU$ as a
vector space over $\C$ not in the usual sense, but in the sense of
Hornich operations (see \cite{Hor69,KPS02,Yam75})
defined by
$$(f\oplus g) (z)=\int_0^z f'(w)g'(w)\,dw ~~\mbox{ and }~~
(\alpha\star f) (z)=\int_0^z \{f'(w)\}^\alpha\, dw
$$
for $f,g\in \LU$ and $\alpha\in \C$, where the branch of
$(f')^{\alpha}=\exp (\alpha \log f')$ is taken so that
$(f')^\alpha (0)=1$.
%and we define the norm of $f\in\LU$ by (\ref{norm}).

The pre-Schwarzian norm has significance in
the theory of Teichm\"uller spaces (see e.g. \cite{AG86}) as well.
We remark that the norm $\|f\|$ is nothing but the
Bloch semi-norm of the function
$\log f'$ (see, for example, \cite{Pom:univ}). We have before already
seen  that
$\|f\|\le 6$ if $f$ is univalent
in $\D$, and it is well-known that if $\|f\|\le 1$ then
$f$ is univalent in
$\D$, and these bounds are sharp (see \cite{Be72,BP84}).
Furthermore,
$\|f\|<\infty$ if and only if $f$ is uniformly locally univalent;
that is, there exists a constant $\rho=\rho(f),~0<\rho\le 1$, such
that $f$ is univalent in each disk of hyperbolic radius $\tanh^{-1}\rho$ in
$\D$, i.e. in each Apollonius (or Apollonian) disk
$$\left\{z\in\C:\,\left|\frac{z-a}{1-\bar az}\right|<\rho\right\},\quad
|a|<1
$$
(see \cite{Yam75,Yam76}). Note that the above disk is called
the Apollonian disk, because it has the same nature as in
the Apollonian balls defined in Subsection \ref{ABA} with $q_a=1/(\rho|a|)$.
Here we observe from Property 5 of Subsection \ref{ABA} that,
the well-known inversion relation,  $a$ and $1/\bar{a}$
are the inverse points with respect to the unit circle.
The set of all $f$ with
$\|f\|<\infty$ is a nonseparable Banach space (see \cite[Theorem
1]{Yam75}). For more geometric and analytic properties of $f$
relating the norm, see \cite{KS98}. Many authors have given norm
estimates for classical subclasses of univalent functions (see for example
\cite{CKPS04, KS2, Oku00, SugawaST, Yam97}).

For $f\in \es$, although its Alexander transform $J[f]$ is not
in $\es$, it is locally univalent and so it is reasonable to obtain
the norm estimates for the Alexander transform of certain
classes of analytic functions. For example, it has been
obtained in \cite{KPS02} that $\|J[f]\|\le 4$ for $f\in \es$ and
the inequality is sharp.

A simple generalization of  $\es^*$
is the so-called class of all {\it starlike
functions of order $\alpha$}\index{order of starlikeness},
$0\leq \alpha \le 1$, denoted by
${\es}^*(\alpha )$. Indeed, $f\in \es^*(\alpha)$ if and
only if $\real (zf'(z)/f(z))\ge \alpha$ in $\D$.
Here we remark that the later inequality is strict
except for $\alpha=1$.
We set ${\es}^*(0) ={\es}^*$.
Similarly, a function $f\in\es$ is said to be
{\it convex of order $\alpha$}\index{order  of convexity} if $\real (1+zf''(z)/f'(z))\ge \alpha$.
This class is denoted by $\K(\alpha)$. Like in the
starlikeness we set $\K(0)=\K$.

In 1999, Yamashita \cite{Yam97} proved that
if $f\in \es^*(\alpha)$ then $\|f\|\le 6-4\alpha$ and
$\|J[f]\|\le 4(1-\alpha)$ (or equivalently, $\|f\|\le 4(1-\alpha)$
for $f$ convex of order alpha) for $0\le \alpha <1$. Both the
inequalities are sharp (see also \cite[Theorem A]{CKPS04}).
There are many classes of functions $f$ for which the norm
$\|f\|$ is finite. We remark that if $f$ is bounded, it may happen that
$\|f\|=\infty$. For instance, the function
$$z\mapsto f(z)=\exp{\frac{z+1}{z-1}}
$$
in the unit disk shows that $T_f(z)=-2z/(1-z)^2$ and hence,
$\|f\|\to \infty$ as $z\to 1^{-}$.

Let us denote $\hol$ for the class of functions $f$ analytic in
the unit disk $\D$ and $\hol_a$ will denote the subclass $\{f\in\hol: \,
f(0)=a\}$, for $a\in\C.$
We say that a function $\varphi\in\hol$ is {\em subordinate}
\index{subordination} to $\psi\in\hol$
and write $\varphi\prec\psi$ or $\varphi(z)\prec\psi(z)$ if there is a
Schwarz function $\omega$ (i.e. a function $\omega\in\hol_0$ with
$|\omega|<1$ in $\D$) satisfying
$\varphi=\psi\circ\omega$ in $\D$. Note that the condition $\varphi\prec\psi$ is
equivalent to the conditions $\varphi(\D)\subset\psi(\D)$ and
$\varphi(0)=\psi(0)$ when $\psi$ is univalent.

If $f,g\in \hol$, with
$$f(z)=\sum_{n=0}^\infty a_nz^n ~~\mbox{ and }~~
g(z)=\sum_{n=0}^\infty b_nz^n,
$$
then the {\em Hadamard product} (or {\em convolution}) of $f$ and $g$
is defined by the function
$$(f*g)(z)=\sum_{n=0}^\infty a_nb_n z^n.
$$

As a motivation in Chapter 6, we consider the class
$$\K(A,B)=\left\{f\in\A:\,1+\frac{zf''(z)}{f'(z)}\prec\frac{1+Az}{1+Bz},
\quad z\in\D\right\},
$$
where $-1\le B<A\le 1$ and $\prec$ denotes the subordination.
For $0<b\le c$, define $B_{b,c}[f]$ by
$$B_{b,c}[f](z)=zF(1,b;c;z)*f(z),
$$
where
%$*$ denotes the usual convolution (or Hadamard product) and
$F(a,b;c;z)$ is the Gauss {\em hypergeometric function}
\index{hypergeometric function}\index{function!hypergeometric} defined by
$$F(a,b;c;z)=\sum_{n=0}^\infty\frac{(a)_n(b)_n}{(c)_n(1)_n}z^n,\quad
z\in\D,
$$
where $(a)_n=a(a+1)\cdots(a+n-1)$ is the Pochhammer symbol (here
$(a)_0=1$) and $c$ is not a non-positive integer.
We have the well-known derivative formula
$$F'(a,b;c;z)=\frac{d}{dz}F(a,b;c;z)= \frac{ab}{c}F(a+1,b+1;c+1;z).
$$
As a special case of the Euler integral
representation for the hypergeometric function, one has
$$F(1,b;c;z)=\frac{\Gamma (c)}{\Gamma (b)\Gamma (c-b)}
\int_0^1 \frac{1}{1-tz} t^{b-1}(1-t)^{c-b-1}\,dt, \quad z\in\D, \quad
\mbox{Re$\,c>$Re$\,b>0$}.
$$
Using this representation we have,  for $f\in \mathcal{A}$, the convolution
transform
$$zF(1,b;c;z)*f(z) = z\left (F(1,b;c;z)*\frac{f(z)}{z}\right ).
$$
%where $*$ denotes the usual Hadamard product (or convolution).
Therefore, we obtain the integral convolution which
defines the (hypergeometric) operator $B_{b,c}[f]$ in the following form
%\begin{equation}\label{conv-op}
$$B_{b,c}[f](z):=zF(1,b;c;z)*f(z)=\frac{\Gamma(c)}{\Gamma(b)\Gamma(c-b)}
\int_0^1 t^{b-1}(1-t)^{c-b-1}\frac{f(tz)}{t}\,dt
$$%\end{equation}
so that
$$(B_{b,c}[f])'(z)  =  F(1,b;c;z)*f'(z).
$$
We obtain sharp pre-Schwarzian norm estimates for
functions in ${\mathcal K}(A,B)$.
In addition, we also present sharp norm estimates
for $B_{b,c}[f](z)$ when $f$ ranges over the class
${\mathcal K}(A,B)$.

Some particular cases need special attention. For example, if
$c=b+1$ and  $b=\gamma +1$, then one has the well-known {\em Bernardi transform}
\index{Bernardi's transform}\index{transform!Bernardi's}
$B_\gamma [f]:=B_{\gamma +1,\gamma +2}[f]$ defined by
\begin{eqnarray}\label{bit}
B_\gamma [f](z)=\frac{\gamma +1}{z^\gamma
}\int_0^zt^{\gamma -1}f(t)\,dt=zF(1,\gamma +1;\gamma +2;z)*f(z),
\end{eqnarray}
for $\gamma >-1$.
We observe that $B_0 [f]=J[f]$ and $B_1 [f]=L[f]$, where
$J[f]$ and $L[f]$ are respectively the Alexander transform of $f$, and the
Libera transform of $f$.
% In terms of convolution, these transforms may also
% be expressed as
% $$J[f](z)=f(z)*(-\log(1-z)) ~\mbox{ and }~L[f](z)=f(z)*\left[-2\left(1+
% \frac{1}{z}\log (1-z)\right)\right].
% $$

Also, similar norm estimates have been established for the class
$${\mathcal F}_\beta=\left\{f\in\A:\,
{\rm Re}\,\left(1+\frac{zf''(z)}{f'(z)}\right)<\frac{3}{2}\beta,
\quad z\in\D\right\},
$$
where $\frac{2}{3}<\beta \le 1$.
% As a consequence, we obtain a number of results
% for some specific integral operators such as the
% Bernardi operator (see \cite{PPS}), the Alexander operator and the
% Libera operator.
The class ${\mathcal F}_\beta$ and its special case ${\mathcal F}_1={\mathcal F}$
have been studied, for example, in  \cite{PR95,PS96,PV07} but for different purposes.
In \cite[Eq. (16)]{PR95} it has been shown that if $f\in {\mathcal F}$, then one has
$$\left |\frac{zf'(z)}{f(z)}- \frac{2}{3} \right |<\frac{2}{3}, \quad z\in \D;
\quad \mbox{i.e.~}~\frac{zf'(z)}{f(z)}\prec \frac{2(1-z)}{2-z}, \quad z\in \D.
$$
Thus,  $\F_{\beta }\subset \F\subset\es^*$ for $\frac{2}{3}<\beta \le 1$.
Note that each $f\in\es^*$
has the well-known analytic characterization:
$$\frac{zf'(z)}{f(z)}\prec\frac{1+z}{1-z} , \quad z\in \D.
$$
In conclusion, we see that the image
domains of the unit disk $\D$ under the functions from ${\mathcal F}_\beta$
and the operators of such functions are quasidisks.
For example, if $f\in\F_1$ then the images $J[f](\D)$ and $L[f](\D)$
under the Alexander and Libera transforms respectively are quasidisks.

%\noindent
%\vskip .2cm \noindent
%{\bf Chapter 7}
As a last result, we obtain %(see \cite{HPS3})
%This chapter deals with
an optimal
but not a sharp pre-Schwarzian norm estimates of functions
$f\in {\mathcal S}^*(\alpha,\beta)$, $0<\alpha\le 1$ and $0\le \beta< 1$,
of
$\A$, where
$$\es^*(\alpha,\beta)=\left\{f\in\A:\, \frac{zf'(z)}{f(z)}\prec
\left(\frac{1+(1-2\beta)z}{1-z}\right)^\alpha\right\}.
$$
Indeed,
$$\|f\|\le L(\alpha,\beta)+2\alpha,
$$
where
%\beq\label{p3eq2}
$$L(\alpha,\beta)=\frac{4(1-\beta)(k-\beta)(k^\alpha
-1)}{(k-1)(k+1-2\beta)}
$$%\eeq
and $k$ is the unique solution of the following equation in $x\in (1,\infty)$:
\begin{eqnarray}
&&\nonumber(1-\alpha)x^{\alpha +2}+\beta(3\alpha-2)x^{\alpha+1}
+[(1-2\beta)(1+\alpha)+2\beta^2(1-\alpha)]x^\alpha \\
~&& \nonumber \hspace{3cm}- \alpha\beta(1-2\beta)x^{\alpha -1}- x^2+2\beta x
=(1-\beta)^2+\beta^2.
\end{eqnarray}

The sharp estimates for this problem remains an open problem.

\section{Summary and Conclusion}\label{summary}

The current chapter is dedicated for the introduction of some basic
concepts and results that we require to present our main
results in the sequel. More precisely, we provide some
history concerning inequalities and isometries of
hyperbolic-type metrics; and
the coefficient estimates, radius problems and
pre-Schwarzian norm estimates of functions from some
subclasses of the class of univalent functions.

In the next three chapters we have studied certain hyperbolic-type
metrics such as the Apollonian metric, its inner metric,
$j$ metric
and its inner metric namely the quasihyperbolic metric.
We look at inequalities among them and their geometric
interpretation in the sense of constructing or
characterizing domains where they hold together.
We also obtain isometries of some hyperbolic-type path metrics such as
the quasihyperbolic metric, the Ferrand metric and the K--P metric in certain
specific domains.

In the remaining two chapters we have considered
some geometrically motivated subclasses $\F$ of $\es$.
We obtain the largest disk
$|z|<r$ for which $\frac{1}{r}f(rz)\in \F$ whenever $f\in \es$.
We also obtain necessary and sufficient coefficient conditions
for $f$ to be in $\F$. In addition, we estimate the pre-Schwarzian
norm of functions from $\F$ and that of certain convolution or
integral transforms of functions from $\F$. Some open questions
concerning certain classes of univalent functions are studied.

We expect that some of the investigations would lead to new results
in different areas of research in function theory.

\chapter{INEQUALITIES AND GEOMETRY OF THE APOLLONIAN AND RELATED METRICS}
\label{chap2}
The structure of this chapter is as follows.
We start by reviewing the definitions, notation and terminology used. The
bulk of the chapter consists of five sections which are organized along
the different methods used to prove the inequalities in Table~\ref{table1}.
Specifically, in Section~\ref{jaSect} we consider the comparison property and
uniformity; and in Section~\ref{QISect} quasi-isotropy. The main problem
in Sections~\ref{qconvSect} and \ref{constructSect} is the inequality
$\alpha_G \age \ak_G$.
% Since we do not have a good geometric understanding
% of this inequality the proofs in these sections are sometimes quite long.
In Section~\ref{noncompSect}
we consider the case when the metrics $j_G$ and $\ak_G$ are not comparable.

Most of the results of this chapter have been published in:
{\bf P. H\"{a}st\"{o}, S. Ponnusamy  and S.K. Sahoo} (2006)
Inequalities and geometry of the Apollonian and related metrics.
{\em Rev. Roumaine Math. Pures  Appl.} \textbf{51}(4), 433--452.

%%%%%%%%%%%%%%%%%%%%%%%%%%%%%%%%%
%%%%%%%%%%% SECTION 1 %%%%%%%%%%%
%%%%%%%%%%%%%%%%%%%%%%%%%%%%%%%%%

\section{Introduction}
\label{intro}

In this chapter we consider the Apollonian metric which was first
introduced by Barbilian \cite{Ba} in 1934--35 and
then rediscovered by Beardon \cite{Be} in 1998.
% The main reasons for dealing with this
% metric are that
% \begin{enumerate}
% \item it has a very nice geometric interpretation (see Subsection~\ref{ABA});
% \item it is invariant under M\"obius map; and
% \item it equals the hyperbolic metric in balls and half-spaces.
% \end{enumerate}
We also consider the inner metric of the Apollonian metric,
the $j_G$ metric and its inner metric, the quasihyperbolic metric.
We are mainly interested in dealing with inequalities among these metrics
(see Table~\ref{table1}) and the geometric meaning of these inequalities.
% We begin our discussion by defining these metrics and stating our
% main results.
The notation
used conforms largely to that of \cite{Be2} and \cite{Vu},
the reader can consult Subsection~\ref{spaceSubsect}, if necessary.

%Unless otherwise stated, we will be considering domains
%(open connected non-empty sets) $G$
%in the M\"obius space $\Rnbar $.
%:=\Rn\cup \{\infty\}$.
Recall that the Apollonian metric\index{Apollonian metric}\index{metric!Apollonian} is defined for
$x,y\in G\psubset \Rnbar$ by
\[ \alpha_G(x,y) := \sup_{a,b\in\partial G}
\log \frac{|a-y|\,|b-x|}{|a-x|\,|b-y|} \]
(with the understanding that $|\infty-x|/|\infty-y|=1$).
% This formula has a very nice geometric interpretation (indeed, this
% is one of the main reasons for the interest in the metric), see Section~\ref{ABA}.
% It is in fact a metric if and only if the complement of $G$
% is not contained in a
% hyperplane and only a pseudometric otherwise, as was noted
% in \cite[Theorem~1.1]{Be}.
This metric was introduced
in \cite{Be} and has also been considered in
\cite{BI,GH,Ro,Se} and \cite{Ha2}--\cite{Ib3}.

% We have already defined the inner metric in Chapter \ref{chap1}.
% Now, for the sake of convenience we recall some basic facts
% about the Apollonian inner metric.
% % We now define the inner metric as follows:
% % Let $\gamma\colon [0,1]\to G\subset \Rn$ be a path, i.e.\  a continuous
% % function. If $d$ is a metric in $G$ then the $d$-length of $\gamma$ is defined by
% % $$ d(\gamma):= \sup \sum_{i=0}^{k-1} d(\gamma(t_i), \gamma(t_{i+1})), $$
% % where the supremum is taken over $k<\infty$ and all sequences $\{t_i\}$ satisfying
% % $0=t_0<t_1<\cdots<t_k=1$. The inner metric of $d$ is defined
% % by the formula
% % $$ \tilde{d}(x,y) := \inf_{\gamma} d(\gamma),$$
% % where the infimum is taken over all paths connecting $x$ and $y$ in $G$.
% We denote the inner metric of the Apollonian metric by $\ak_G$ and
% call it the Apollonian inner metric\index{Apollonian inner metric}
% \index{metric!Apollonian inner}.
% Strictly speaking,
% the Apollonian inner metric is only a pseudometric in a general
% domain $G\psubset\Rn$; it is a metric if and only if
% the complement of $G$ is not contained in
% an $(n-2)$-dimensional plane \cite[Theorem~1.2]{Ha5}.
For definitions and some of the properties of the Apollonian inner metric,
the $j_G$-metric and the quasihypebolic metric we refer to
Section \ref{Ineq-geom}.

\begin{table}
\center{
\begin{tabular}{|rlcc|rlcc|}
\hline
\# & Inequality & A & B &  \# & Inequality & A & B\\
\hline
1. & $\alpha\aeq j \aeq \ak \aeq k$ & + & + &
7. & $\alpha\aeq j \mle \ak \mle k$& -- & -- \\
2. & $\alpha\mle j \aeq \ak \aeq k$ & -- & -- &
8. & $\alpha\mle j \mle \ak \mle k$ & -- & -- \\
3. & $\alpha\aeq j \aeq \ak \mle k$ & -- & -- &
9. & $\alpha\aeq \ak \mle j \aeq k$ & -- & + \\
4. & $\alpha\mle j \aeq \ak \mle k$ & -- & -- &
10. & $\alpha\mle \ak \mle j \aeq k$ & -- & + \\
5. & $\alpha\aeq j \mle \ak \aeq k$ & + & + &
11. & $\alpha\aeq \ak \mle j \mle k$ & -- & -- \\
6. & $\alpha\mle j \mle \ak \aeq k$ & + & + &
12. & $\alpha\mle \ak \mle j \mle k$ & -- & -- \\
\hline
\end{tabular}}
\medskip
\caption{Inequalities between the metrics $\alpha_G$, $j_G$, $\ak_G$ and
$k_G$. The subscripts are omitted for clarity with the understanding that
every metric is defined in the same domain. The A-column refers to whether
the inequality can occur in simply connected planar domains, the
B-column to whether it can occur in proper subdomains of $ \Rn$.}\label{table1}
\end{table}

We will undertake a systematic study of which of the inequalities in (\ref{genIneqs})
%$\alpha_G\ale j_G \ale k_G$ and $\alpha_G\ale \ak_G \ale k_G$
can hold
in the strong form with $\mle$ and which of the relations $j_G\mle \ak_G$,
$j_G\aeq \ak_G$ and $j_G\mge \ak_G$ can hold. Thus we are led to twelve
inequalities, which are given along with the results in Table~\ref{table1},
where we have indicated in column A whether the inequality can hold in
simply connected planar domains and in column B whether it can hold in
an arbitrary proper subdomains of $\Rn$. From the table we see
that most of the cases cannot occur, which means
that there are many restrictions on which inequalities can occur together. For
instance, we deduce from items 1--4 that $j_G\aeq \ak_G$ implies that
$\alpha_G\aeq k_G$ and from items 9--12 that the inequality $\ak_G\mle j_G$
cannot occur in simply connected planar domains.

Since $\ale$ is not a linear order, it is also possible that two metrics are
not comparable. Therefore we consider separately the case
when $j\notcomp \ak$ in Section~\ref{noncompSect}.
Since the table does not list this case, one should be careful with
the interpretations; for instance, it is not true that the inequality
$\ak_G\mle k_G$ cannot occur in simply connected planar domains,
contrary to what might be thought
by considering entries~3, 4, 7, 8, 11 and 12.

% As a motivation for this study we mention that many of the inequalities have
% been previously studied and some have geometrical characterizations.
% A domain in which $j_G\aeq k_G$ holds is known as uniform
% and has found applications in many
% areas of analysis (see e.g.\  \cite{Ge2,Jo80}).
% The relation $\alpha_G\aeq j_G$ is also connected to various
% interesting properties, for example see \cite[Theorem~1.3]{Ha1}.
% % The condition $\alpha_G\aeq j_G$ was shown in \cite[Theorem~1.3]{Ha1}
% % to be equivalent with the
% % complement of $G$ being thick in the sense of \cite{VVW}, which means that
% % this inequality is also connected to various interesting properties, see
% % \cite[Section~1.4]{VVW} for details.
% Additionally, the inequalities
% $\ak_G\aeq k_G$, $\alpha_G\aeq \ak_G$ and $\alpha_G\aeq k_G$, which have
% been called quasi-isotropy, Apollonian quasiconvexity and $A$-uniformity,
% respectively, have some nice geometric interpretations and
% have been considered in \cite{Ha2,Ha4,Ha3}. For concrete relationship
% among these domains, see \cite{HPWS}.

%%%%%%%%%%%%%%%%%%%%%%%%%%%%%%%%%

\subsection{Notation}\label{spaceSubsect}

%The notation used conforms largely to that in \cite{Be2} and \cite{Vu}, as was
%mentioned in the introduction.

We denote by $\{e_1, e_2, \ldots, e_n\}$ the standard basis of
$\Rn$ and by $n$ the dimension of the Euclidean space under
consideration and assume that $n\ge 2$. For $x\in \Rn$ we denote
by $x_i$ its $i^{\rm{th}}$ coordinate. The following notation is
used for Euclidean balls and spheres:
$$ B^n(x, r):=\{y \in \Rn \colon \abs{x-y} <r \},\
S^{n-1}(x, r):=\{y \in \Rn \colon \abs{x-y} =r\},$$
$$ B^n:=B^n(0,1),\ S^{n-1}:=S^{n-1}(0,1).
$$
We denote by $[x,y]$ the closed segment between $x$ and $y$.
%$$ B^n:=B^n(0,1),\ S^{n-1}:=S^{n-1}(0,1),\ H^n:=\{y\in \Rn \colon y_n > 0\}.$$

We use the notation $\Rnbar := \Rn \cup \{ \infty\}$ for the one point
compactification of $\Rn$, equipped with the chordal metric.
Thus an open ball of $\Rnbar$
as an open Euclidean ball, an open half-space or the complement of
a closed Euclidean ball.
We denote by $\partial G$, $G^c$ and $\overline{G}$
the boundary, complement and closure of $G$, all with respect to $\Rnbar$.
%In contrast to topological operations, we always consider metric operations
%with respect to the ordinary Euclidean metric.

We also need some notation for quantities depending on the underlying
Euclidean metric.
For $x\in G\psubset\Rn$ we write $\delta(x):=d(x,\partial G):=
\min \{\abs{x-z}\colon z\in \partial G\}$.
For a path $\gamma$ in $\Rn$ we denote by $\ell(\gamma)$ its
Euclidean length.
For $x,y,z\in\Rn$ we denote by $\widehat{xyz}$ the smallest angle
between the vectors $x-y$ and $z-y$.
% We denote by $xy$ the line through $x$ and $y$ and by $[x,y]$ the closed
%segment between $x$ and $y$.

%%%%%%%%%%%%%%%%%%%%%%%%%%%%%%%%%

%%%%%%%%%%%%%%%%%%%%%%%%%%%%%%%%%%%%%%%%
%%%%%%%%%%%%%%%%%%%%%%%%%%%%%%%%%%%%%%%%
%%%%%%%%%%%%%%%%%%%%%%%%%%%%%%%%%%%%%%%%

\section{Basic Inequalities}\label{jaSect}

In this section we define the comparison property and uniformity which are
the relations from the introduction that have been most thoroughly studied
in the past.

%%%%%%%%%%%%%%%%%%%%%%%%%%%%%%%%%%%%%%%%

\subsection{The comparison property}\label{approxSubsect}

In \cite{Ha1} the term {\em comparison property}\index{comparison
property} was introduced for the relation $\alpha_G\aeq j_G$.
Also, an equivalence formulation of this property has been studied
by H\"ast\"o, see (Theorem 1.3 in \cite{Ha1}).
%in  it was shown that
%this property is equivalent to the complement of $G$ being thick
%in the sense of
% \cite{VVW}.
% There is a discussion of implications of thickness
% in \cite[Section~1.4]{VVW}, as an example we mention that a
% bilipschitz mapping with sufficiently small bilipschitz constant
% from a thick set may be continued to a bilipschitz mapping of the whole space,
% by \cite[Theorem~6.2]{Va5}.

% Recall that the inner metric is defined as a supremum over a sum of the base
% distance.
From the definition of the inner metric (see Section \ref{Ineq-geom})
it directly follows that if $d_1$ and $d_2$ are metrics
in the same domain, then $d_1\aeq d_2$ implies that $\tilde{d}_1\aeq \tilde{d}_2$.
Therefore Inequalities~3 (Table \ref{table1}), $\alpha_G\aeq j_G \aeq \ak_G \mle k_G$,
and 7, $\alpha_G\aeq j_G \mle \ak_G \mle k_G$, cannot occur, since
in both cases we have assumed the comparison property but not the equivalence
of the inner metrics, $\ak_G$ and $k_G$.

A well-known fact from \cite[Theorem~3.2]{Be} is that $\alpha_G\le 2j_G$
in every domain $G\psubset \Rn$. Also, it was shown in  \cite[Theorem~4.2]{Se}
that if $G\psubset \Rn$ is convex, then $j_G\le \alpha_G$.
So $\alpha_G\aeq j_G$ in convex domains.

\begin{lemma} Inequality 5, $\alpha_G\aeq j_G \mle \ak_G \aeq k_G$,
holds in the domain $G:=\{x\in \Rn\colon \abs{x_n}<1\}$.
\end{lemma}

\begin{proof} The domain $G$ is clearly convex, hence it has the
comparison property by  \cite[Theorem~4.2]{Se}. From this
it follows that $\alpha_G\aeq j_G$ and $\ak_G\aeq k_G$.  Consider then the
points $R e_1$ and $-Re_1$, where $R>0$. We have
$j_G(R e_1, -R e_1)=\log(1+2R)$ and $k_G(R e_1, -R e_1)= 2R$, hence
$j_G\mle k_G$, which concludes the proof.
%\endproof
 \end{proof}

%%%%%%%%%%%%%%%%%%%%%%%%%%%%%%%%%%%%%%%%

\subsection{Uniformity}\label{unifSubsect}

Uniform domains were introduced by O. Martio and J. Sarvas in \cite[2.12]{MS},
but the following definition is an equivalent form from \cite[(1.1)]{GO}. In
\cite{Ge2}, there is a survey of characterizations and implications of uniformity.
% as an example we mention that a Sobolev mapping can be extended from
% $G$ to the whole space if $G$ is uniform, see \cite{Jo}.
\begin{definition}\label{unifDef}
A domain $G\psubset \Rn$ is said to be {\em uniform}\index{uniform domain}
\index{domain!uniform} with constant $K$ if
for every $x,y\in G$ there exists a path $\gamma$,
parameterized by arc-length, connecting $x$ and $y$ in $G$, such that
%\begin{enumerate}
%\item
$\ell(\gamma) \le K \abs{x-y}$; and
%\item
$K \delta(\gamma(t)) \ge \min \{t, \ell(\gamma) -t\}$.
%\end{enumerate}
\end{definition}
The relevance of uniformity to our
investigation comes from Corollary~1 of \cite{GO} which states that
a domain $G$ is uniform if and only if $k_G\aeq j_G$.
This condition is also equivalent to
$\ak_G\ale j_G$, see \cite[Theorem 1.2]{HPWS}.
Thus we have a geometric
characterization of domains satisfying this inequality as well.

\begin{example} The unit ball is uniform and has the comparison property.
Hence $\alpha_{B^n} \aeq j_{B^n} \aeq \ak_{B^n}\aeq k_{B^n}$ and so
Inequality~1 can occur.
\end{example}
In fact, Inequality~1 holds in every quasiball, by \cite[Corollary~6.9]{Ha2}.
\begin{lemma}\label{not910}
Inequalities~9 (Table \ref{table1}), $\alpha_G\aeq \ak_G\mle j_G\aeq k_G$,
and 10, $\alpha_G\mle \ak_G\mle j_G\aeq k_G$ cannot occur in simply
connected planar domains.
\end{lemma}

\begin{proof}
We note that in both these inequalities we have
$j_G\aeq k_G$ among the assumptions. But a simply connected planar
domain is uniform if and only if it is a quasidisk, by \cite[Theorem~2.24]{MS},
and we know that quasidisks have the comparison property, by
\cite[Corollary~6.3]{Ha1}. Therefore $j_G\aeq k_G$ implies that
$\alpha_G\aeq j_G$ which contradicts $\alpha_G\mle j_G$ in
both inequalities.
%\endproof
 \end{proof}

%%%%%%%%%%%%%%%%%%%%%%%%%%%%%%%%%%%%%%%%%
%%%%%%%%%%%%%%%%%%%%%%%%%%%%%%%%%%%%%%%%%
%%%%%%%%%%%%%%%%%%%%%%%%%%%%%%%%%%%%%%%%%

\section{Quasi-isotropy}\label{QISect}

%%%%%%%%%%%%%%%%%%%%%%%%%%%%%%%%%

%\subsection{The Apollonian spheres and quasi-isotropy}\label{ASA}

We start by introducing some concepts which allow us to calculate the
Apollonian inner metric. The concept of quasi-isotropy was
introduced in \cite{Ha2} and was studied in depth in \cite{Ha3}.
A very similar notion
used by Zair Ibragimov is conformality, see \cite{Ib1,Ib2,Ib3}.

\begin{definition}\label{qi} We say that a metric space $(G,d)$ with $G\subset\Rn$
is {\em $K$--quasi-isotropic}\index{quasi-isotropy} if
$$ \limsup_{r\to 0} \frac{\sup \{d(x,z)\colon \abs{x-z}=r\}}{
\inf \{d(x,y)\colon \abs{x-y}=r\}} \le K
$$
for every $x\in G$. A $1$--quasi-isotropic metric space is called
{\em isotropic}\index{isotropic}.
\end{definition}

We say that a domain $G\psubset \Rn$ is quasi-isotropic
\index{quasi-isotropic domain}\index{domain!quasi-isotropic}
if $(G,\alpha_G)$ is $K$--quasi-isotropic for some constant $K$;
similarly for isotropic.
We define the function $qi$ on the set of proper subdomains of $\Rn$ so that
$qi(G)$ is the least constant for which $G$ is quasi-isotropic or
$qi(G)=\infty$ if $G$ is not quasi-isotropic for any $K$. The notion
of quasi-isotropy is extended to domains in $\Rnbar$ by M\"obius
invariance.

%Although the Apollonian metric is not isotropic in general, it is possible
Note that the Apollonian metric is not isotropic. It is, nevertheless, possible
to define a directed density\index{directed density} as follows:
\[ \bar\alpha_G(x;r) = \lim_{t\to 0} \tfrac1t \alpha_G(x, x+ t\tfrac{r}{|r|}),
\]
where $r\in \Rn\setminus \{0\}$. Unless otherwise stated, in this chapter,
we will be using the notation $r$ not for a number, but for a vector
whenever we talk about the above notation for the directed density.
If $\bar\alpha_G(x;r)$ is
independent of the vector $r$ at every point of $G$, then the Apollonian
metric is isotropic and we may denote
$\bar\alpha_G(x):=\bar\alpha_G(x;e_1)$ and call this function the
density of $\alpha_G$ at $x$. With this concept we can give the
following alternative characterization of quasi-isotropy.

\begin{lemma}\cite[Lemma~3.5]{Ha3}\label{qiDef2}
For $G\psubset\Rn$ we have
$$ qi(G) = \sup_{x\in G} \frac{\sup_{r\in S^{n-1}} \bar\alpha_G(x;r)}
{\inf_{r\in S^{n-1}} \bar\alpha_G(x;r)}, $$
with the understanding that if $\bar\alpha_G(x;r)=0$ for some $x\in G$ and $r\in S^{n-1}$,
then $qi(G)=\infty$.
\end{lemma}

When we do not need the exact value of the quasi-isotropy constant the
following lemma is often more convenient to use.
\begin{lemma}\cite[Corollary~5.11]{Ha2}\label{qiDef3}
Let $G\psubset\Rn$ be $L-$quasi-isotropic.
Then $\bar\alpha_G(x;r)\delta(x) \ge 1/L$ for every $x\in G$ and $r\in S^{n-1}$.
If conversely $1/L \le \bar\alpha_G(x;r)\delta(x)$ for every $x\in G$
and $r\in S^{n-1}$, then $G$ is $2L$--quasi-isotropic.
\end{lemma}

%%%%%%%%%%%%%%%%%%%%%%%%%%%%%%%%%

%\subsection{The Apollonian inner metric}\label{AIM}

In order to present an integral formula for the Apollonian inner metric we need
to relate the density of the Apollonian metric with the limiting concept of
the Apollonian balls, which we call the Apollonian spheres.

\begin{definition}\label{aSpheres}
Let $G\psubset\Rnbar$, $x\in G$ and $\theta \in S^{n-1}$.
\begin{itemize}
\item If $B^n(x + s \theta, s)\subset G$ for every $s>0$ and $\infty\not\in G$,
then let $r_+=\infty$.
\item If $B^n(x + s \theta, s)\subset G$ for every $s>0$ and $\infty\in G$,
then let $r_+$ be the largest negative real number such that
$G\subset B^n(x + r_+ \theta, \abs{r_+})$.
\item Otherwise let $r_+>0$ be the largest real number such that
$B^n(x + r_+ \theta, r_+)\subset G$.
\end{itemize}
Define $r_-$ in the same way but using the vector $-\theta$ instead
of $\theta$. We define the {\em Apollonian spheres\index{Apollonian sphere}
through $x$ in
direction $\theta$} by $S_+:= S^{n-1}(x + r_+  \theta, r_+)$ and
$S_-:= S^{n-1}(x - r_- \theta, r_-)$ for finite radii and by the limiting half-space for
infinite radii.
\end{definition}

Using these spheres we can present a useful result from \cite{Ha2}.

\begin{lemma}\cite[Lemma~5.8]{Ha2}\label{aDensLem}
Let $G\psubset\Rnbar$ be open, $x\in G\setminus\{\infty\}$ and
$\theta\in S^{n-1}$. Let $r_\pm$ be the radii of the Apollonian spheres $S_\pm$
at $x$ in direction $\theta$. Then
\[ {{\bar{\alpha}}}_G(x;\theta) = \frac{1}{2r_+} + \frac{1}{2r_-},\]
where we understand $1/\infty=0$.
\end{lemma}

\begin{remark} The previous lemma was proved in \cite{Ha2} only for the case
$G\psubset\Rn$. The general case is proved in exactly the same manner.
\end{remark}

The following result shows that we can find the Apollonian inner metric by integrating
over the directed density, as should be expected.
Piecewise continuously differentiable means
continuously differentiable except in a finite number of points.

\begin{lemma}\cite[Theorem~1.4]{Ha5}\label{innerThm}
If $x,y\in G\psubset \Rn$, then
$$ \ak_G(x,y) =
\inf_\gamma \int \bar\alpha_G(\gamma(t); \gamma'(t) ) \abs{\gamma'(t)} dt, $$
where the infimum is taken over all paths connecting $x$ and $y$ in $G$
that are piecewise continuously differentiable (with the understanding
that $\bar\alpha_G(z;0)0 =0$ for all $z\in G$, even though
$\bar\alpha_G(z;0)$ is not defined).
\end{lemma}

%%%%%%%%%%%%%%%%%%%%%%%%%%%%%%%%%%5

%\subsection{Relation to the inequalities}

The importance of quasi-isotropy to the study of inequalities is a consequence
of the following lemma.

\begin{lemma}\cite[Corollary~5.4]{Ha5}\label{qiEquivs}
For $G\psubset\Rn$ the following conditions are equivalent:
\begin{enumerate}
\item\label{item1} $G$ is quasi-isotropic;
\item\label{item2} $\ak_G\aeq k_G$; and
\item\label{item3} $j_G\ale \ak_G$.
\end{enumerate}
\end{lemma}

\begin{corollary} Inequalities~4 (Table \ref{table1}),
$\alpha_G\mle j_G\aeq \ak_G\mle k_G$,
and 8, $\alpha_G\mle j_G\mle \ak_G\mle k_G$, cannot occur.
\end{corollary}

\begin{proof} In both cases the assumption $\ak_G\mle k_G$ implies
that $j_G\not\ale \ak_G$, by the previous lemma. This
contradicts $j_G\aeq \ak_G$ (in 4) and $j_G\mle \ak_G$ (in 8).
\end{proof}

% \begin{proposition}\label{lollipop}
% Inequality~6, $\alpha_G\mle j_G\mle \ak_G\aeq k_G$,
% holds in the domain
% $$ G:=H^2 \setminus\bigg(\overline{B^2}(2e_n,1) \cup [0,e_2] \bigg). $$
% \end{proposition}

% \begin{figure}
% \begin{center}
% \includegraphics[width=6cm]{qiFig.eps}
% \end{center}
% \caption{The lollipop domain.}\label{qiFig}
% \end{figure}

In \cite{Ha2} an exterior ball condition of $G$ was defined as follows:
for every $z\in \partial G$ there exists a ball of radius
$r$ in the set $G^c \cap B^n(z,Lr)$, where $L>1$. This condition was shown to
be sufficient for the comparison property. The following
theorem features a local version of this property.

\begin{theorem}\label{qi-bc}
Let $G\psubset\Rn$ be arbitrary and $L>1$. For every $x\in G$, let
$z\in\partial G$ be such that $|x-z|=\delta(x)$ and suppose there
exists a ball $B$ with radius $r_0=\delta(x)/\sqrt{L^2-1}$ such that
\begin{enumerate}
\item $d:=d(z,\partial B)\le r_0(L-1)$; and
\item for any $y\in B$ the line segment $[x,y]$ connecting $x$ and $y$
intersects $\partial G$.
\end{enumerate}
Then the inequality $\ak_G\aeq k_G$ holds.
\end{theorem}

\begin{proof}
It follows from Lemma~\ref{qiEquivs} that $\ak_G\aeq k_G$ if
and only if $G$ is quasi-isotropic. In order to show that $G$ is quasi-isotropic,
by Lemma~\ref{qiDef3}, it suffices to check that there exists a constant $K$ such that
$\bar\alpha_G(x;r)\delta(x) \ge 1/K$ for every $x\in G$ and $r\in S^{n-1}$.

Let $x\in G$ and $r\in S^{n-1}$, and fix a ball $B$ as in the statement of
the theorem. By (2), we see that the Apollonian
spheres with respect to $G$ are smaller in size than with respect
to $\Rn\setminus B$ and since $\Rn\setminus B$ is isotropic
(as the Apollonian metric equals the hyperbolic metric in a
ball) we get
$$\bar\alpha_G(x;r)\ge\bar\alpha_{\Rn\setminus B}(x;r)=\bar\alpha_{\Rn\setminus B}(x)
=\frac{1}{\delta(x)+d}-\frac{1}{\delta(x)+d+2r_0},
$$
(the second term is negative, as the corresponding ball contains
the point $\infty$). Now if we use (1) and
$r_0=\delta(x)/\sqrt{L^2-1}$, from the hypothesis, then it is easy
to estimate that
$$\bar\alpha_G(x;r)\delta(x)\ge\frac{2r_0\delta(x)}{(\delta(x)+r_0(L-1))
(\delta(x)+r_0(L+1))}=\frac{1}{L+\sqrt{L^2-1}},
$$
and we have a lower bound for $\bar\alpha_G(x;r)\delta(x)$.
\end{proof}

The following result provides us with some concrete examples of
when the conditions
of the previous theorem are satisfied.
Although it is intuitively
obvious that the examples satisfy the conditions of the theorem,
verifying this requires some lengthy calculations and several different
cases. %Details are provided in \cite{Sa}.

\begin{example}\label{g-lollipop}
Let $D\psubset\R^2$ be convex and $D'$ be a subset of $D$ which is
compact and convex. Let $F$ be a line segment connecting $\partial
D$ to $\partial D'$. Then Inequality~6, $\alpha_G\mle j_G\mle
\ak_G\aeq k_G$, holds in the domain $G:=D\setminus(D'\cup F)$.
\end{example}
\begin{proof}
Let $z\in F$ and $\epsilon_0>0$ be such that
$B^2(z,\epsilon)\subset D\setminus D'$ for all
$\epsilon\in(0,\epsilon_0)$. Let $x,y\in S^{1}(z,\epsilon)$ be
diametrically opposite such that $[x,y]$ is perpendicular to $F$.
Then it is easy to see that $\alpha_G(x,y) \to 0$ as $\epsilon\to
0$, but on the other hand $j_G(x,y)=\log 3$. Hence
$\alpha_G(x,y)/j_G(x,y) \to 0$ as $\epsilon\to 0$, which means
that $\alpha_G\mle j_G$ holds. Also we note that $G$ is not
uniform as it is not possible to connect the same $x$ and $y$ with
a path of length comparable to $\epsilon$ as $\epsilon \to 0$,
which violates the first condition in Definition \ref{unifSubsect} of
uniformity. Thus we get $j_G\mle k_G$, because $G$ is uniform if and
only if $k_G\aeq j_G$ and $j_G\le k_G$ always holds. We have thus
proved that $\alpha_G\mle j_G\mle k_G$. So it remains to prove the
last inequality, $\ak_G\aeq k_G$.

Denote $d':=d(D,D')$. Let $B^2(p,r)$ be largest ball contained in
$D'$ and $B^2(p,R)$ be the smallest ball with center $p$
containing $D'$. Since $D'$ is convex and compact, $r$ and $R$ are
finite. We define
$$L=\max\left\{\sqrt{2},\frac{R}{r},1+\frac{d'+\diam D'}{r}\right\}
$$
and check Lemma \ref{qi-bc} for $G$ with this constant $L$. For
$x\in G$ choose $z\in\partial G$ such that $\delta(x)=|x-z|$. Now,
if $z\in\partial D$, take any ball $B \subset D^c$ so that
$z\in\partial B$. Then for any $y\in B$ the line segment $[x,y]$
connecting $x$ and $y$ intersects $\partial D\subset\partial G$.
Since $D$ is convex we can choose any $L>1$ in Lemma \ref{qi-bc}
for this $x$.

Next if $z\in\partial F$, take a line $L'$ perpendicular to $F$
through $x$ and $z$. Consider the balls with radius
$r_0=\delta(x)/\sqrt{L^2-1}$ tangent to both $F$ and $L'$ but on
the other side of $F$ than $x$. Of the two balls satisfying this
condition, denote by $B$ the one closer to $F\cap\partial D$. This
gives $d(z,\partial B)=r_0(\sqrt{2}-1)$. Since $L\ge \sqrt{2}$,
the hypotheses of Lemma \ref{qi-bc} are satisfied for this case.

Finally, suppose $z\in\partial D'$. If $\delta(x)\le
r\sqrt{L^2-1}$, construct rays $L_1$ and $L_2$ starting from $z$
and tangent to $B^2(p,r)$. Choose a ball $B:=B^2(w,r_0)$ centered
at $w$ and radius $r_0=\delta(x)/\sqrt{L^2-1}$ to which $L_1$ and
$L_2$ are tangent. Since $r_0\le r$, $D'$ is convex and $B\subset
D'$, for any $y\in B$ the line segment $[x,y]$ intersects
$\partial D'\subset\partial G$. Let $a$ and $b$ be points where
$L_1$ is tangent to $B^2(p,r)$ and $B$, respectively. Now it is
easy to see that the triangles $\triangle\,apz$ and
$\triangle\,bwz$ are similar, which gives $d(z,\partial B)\le
r_0(R/r-1)$, since $|z-p|\le R$. Because $L\ge R/r$, the
hypotheses of Lemma \ref{qi-bc} are satisfied. If
$\delta(x)>r\sqrt{L^2-1}$, choose a ball $B\subset D^c$ with
radius $r_0=\delta(x)/\sqrt{L^2-1}$ at a distance $d(z,\partial
D)$ from $z$. We see that for any $y\in B$ the line segment
$[x,y]$ intersects $\partial D\subset\partial G$. By the triangle
inequality, it is clear that $d(z,\partial B)=d(z,\partial D)\le
d'+\diam D'$. Since $L\ge 1+(d'+\diam D')/r$ and
$\delta(x)>r\sqrt{L^2-1}$ we get
$\delta(x)>\sqrt{(L+1)/(L-1)}(d'+\diam D')$. This gives
$d<r_0(L-1)$. Thus for any $z\in\partial G$ with
$|x-z|=\delta(x)$, we get all conditions of Lemma \ref{qi-bc},
which gives the conclusion.
\end{proof}

%%%%%%%%%%%%%%%%%%%%%%%%%%%%%%%%%%%%%%%%%
%%%%%%%%%%%%%%%%%%%%%%%%%%%%%%%%%%%%%%%%%
%%%%%%%%%%%%%%%%%%%%%%%%%%%%%%%%%%%%%%%%%

\section{Apollonian Quasiconvexity and
Comparison Property}\label{qconvSect}

In this section we consider Inequalities~2, 11 and 12 (Table \ref{table1}). We prove that
none of them can occur in simply connected planar domains and that
the first one cannot occur in more general domains, either.
Whether the latter two
can occur in this case is unclear, although it seems improbable.

We say that a metric space $(G,d)$ is {\em $K$--quasiconvex}
if for every $x,y\in G$
there exists a path $\gamma$ connecting $x$ and $y$ in $G$ such
that $d(\gamma)\le K d(x,y)$, where $d(\gamma)$ is the $d$-length of $\gamma$
defined in Section \ref{Ineq-geom}.
We note that the metric $d$ is quasiconvex if and only if $d\aeq \tilde{d}$.
In \cite[Proposition~7.3]{Ha2} it was shown that if $\alpha_G$ is quasiconvex
in a simply connected planar domain, then $G$ has the comparison
property. Thus $\alpha_G\aeq \ak_G$\index{Apollonian quasiconvexity} implies
$\alpha_G\aeq j_G$ and so Inequality~11,
$\alpha_G\aeq \ak_G \mle j_G \mle k_G$, cannot occur in this case.
Let us move on to the other two inequalities.

\subsection{The twelfth inequality}
\label{twelveSubsect}

In this subsection we prove that the inequalities $\alpha_G\mle \ak_G \mle
j_G \mle k_G$ cannot occur in simply connected planar domains. We are
not be able to establish whether or not it can occur in domains in general.
Let us first quote two lemmas from \cite{Ha2}.

\begin{lemma} \cite[Lemma~7.1]{Ha2}\label{GBHB}
Let $G\subset\Rn$ be a domain such that $G\cap B^n = H^n \cap B^n$. Then for
every $0<s<1$ and every path $\gamma$
connecting $s e_n$ with $S^{n-1}$ we have
$$ \alpha_G(\gamma) \ge \tfrac12 (\arsh s^{-1} - \arsh 1) . $$
\end{lemma}

\begin{lemma}\cite[Lemma~7.2]{Ha2}\label{a<<jLem}
Let $G\psubset \R^2$ be a simply connected domain and $x,y\in G$ be such that
$N \alpha_G(x,y) <j_G(x,y)$ for some $N> 40$.
Then there exists a disk $B := B^2(b,r)$ and a unit vector $e\in S^1$ such that
\begin{enumerate}
\item for all $z\in G^c \cap B$ we have $\ip{z-b}{e}\le 4 N^{-1/2} r$; and
\item the points $b \pm 0.9 r e$ belong to
different path components of $B\cap G$.
\end{enumerate}
$($Here $\ip{\cdot}{\cdot}$ denotes the usual inner product.$)$
\end{lemma}

The proof of the next result is similar to that of Proposition~7.3 in
\cite{Ha2}.

\begin{proposition}\label{twelveProp}
If $G\psubset\R^2$ is a simply connected domain which does
not have the comparison property, then $\ak_G\not\ale j_G$.
\end{proposition}

\begin{proof} Let us assume that $G$ is simply connected but does not have the
comparison property. Let $x,y\in G$ be such that
$N \alpha_G(x,y)\le j_G(x,y)$ for some $N>300$ and define $\epsilon:=2N^{-1/4}$.

Let $B$ be the disk from Lemma~\ref{a<<jLem} and assume without loss of
generality that $B=B^2$ and $e=e_2$. Let $\gamma$ be a path
connecting $\epsilon e_2$ and $-\epsilon e_2$ in $G$. Every
such path passes through $S^1$, since it is easy to see that
$\epsilon e_2$ and $-\epsilon e_2$ are in different components of
$B^2\cap G$.

Let $\gamma_1$ be the part of $\gamma$ in the component of $G\cap B^2$ which
contains $\epsilon e_2$. In order to derive a lower bound for the density of the
Apollonian metric in $\gamma_1$ it suffices to consider the subset
$B^2\cap \partial G$ of the boundary of $G$. The lower bound gets even smaller if
we assume that $ B^2\cap \partial G = \{ x\in B^2\colon x_2 = - 4/\sqrt{N}\}$.
We can then apply Lemma~\ref{GBHB} to $\gamma_1$ after using
an auxiliary translation ($x\mapsto x+ 4e_2/\sqrt{N}$) and
scaling ($ x\mapsto \sqrt{N}x/\sqrt{N-16}$). Under these operations
the point $\epsilon e_2$ is mapped to $(\epsilon \sqrt{N}+4) e_2
/\sqrt{N-16}$ and so the lemma applies with
$s= (\epsilon \sqrt{N}+4)/\sqrt{N-16}=
(2N^{1/4}+4)/\sqrt{N-16}$. Thus we find that
$$ \ak_G(\epsilon e_2, -\epsilon e_2) \ge
\frac{1}{2} \arsh \bigg(\frac{\sqrt{N-16}}
{2N^{1/4}+4}\bigg) - \frac{1}{2} \arsh 1. $$
On the other hand, we have
$$ j_G(\epsilon e_2, -\epsilon e_2) \le \log(1+2\epsilon/(\epsilon-4/\sqrt{N}))
=\log( 1+ 2N^{1/4}/(N^{1/4}-2)). $$
Hence we see that $\ak_G(\epsilon e_2, -\epsilon e_2) /
j_G(\epsilon e_2, -\epsilon e_2) \to \infty$ as $N \to \infty$ which
means that $\ak_G\not\ale j_G$.
\end{proof}

The following corollary is immediate.

\begin{corollary}
Inequalities~11 and
12 (Table \ref{table1}) cannot occur in simply connected planar domains.
\end{corollary}

Recall that a quasidisk\index{quasidisk} is the image of a
disk under a quasiconformal mapping
$f\colon \R^2 \to \R^2$. Using the previous result we get yet another
characterization of quasidisks (for characterizations in terms of the
Apollonian metric see \cite{Ha2}, for lots of other characterizations see \cite{Ge}).

\begin{corollary} A simply connected plane domain $G$ is a quasidisk if
and only if $\ak_G\ale j_G$.
\end{corollary}

\begin{proof} If a simply connected domain $G$ is a quasidisk,
then $G$ is uniform by \cite[Theorem~2.24]{MS}
and \cite[Corollary~2.33]{MS}, hence $\ak_G\ale k_G \aeq j_G$.
Assume conversely that $\ak_G\ale j_G$. It follows from
Proposition~\ref{twelveProp} that $G$ has the comparison property and
hence also $\ak_G\aeq k_G$ (as in Section~\ref{approxSubsect}).
We thus have $k_G \aeq \ak_G \ale j_G$, which means
that $G$ is uniform and hence a quasidisk by \cite[Theorem~2.24]{MS}.
%\endproof
\end{proof}
The following characterization of uniform domains in terms of
the Apollonian metric is due to \cite{HPWS}.
\begin{lemma}\cite[Theorem 1.2]{HPWS}
A domain $D$ is  uniform if and only if there exists a constant $c$ such that
$\ak_D(z_1,z_2) \leq c\,j_D(z_1,z_2)$ for all $z_1,z_2\in D$.
\end{lemma}

As a consequence, in \cite{HPWS}, the authors have established the negation of
Inequalities 11 and 12 (Table \ref{table1}) for arbitrary domains,
see \cite[Corollary 1.3]{HPWS}.

%%%%%%%%%%%%%%%%%%%%%%%%%%%%%%%%%%%%%%%%%

\subsection{The second inequality}
\label{twoSubsect}

In this subsection we prove that Inequality~2, $\alpha_G\mle j_G\aeq
\ak_G\aeq k_G$, cannot occur in any domain.

Let us quote a lemma from \cite{Ha1} that was used in the proof of
Lemma~\ref{a<<jLem} which we use to derive a variant of that lemma
which is valid in $\Rn$.

\begin{lemma}\cite[Lemma~3.1]{Ha1}\label{a/j<<1}
Let $G\psubset \Rn$ be a domain and $x,y\in G$ be points such that
$\alpha_G(x,y)\le j_G(x,y)/N$, for $N\ge 16$. Then there exist balls $B$, $B_1$
and $B_2$ with radii $r$ and $r_1=r_2\ge (1-3 N^{-1/2})r/2$ such that
$B_1, B_2\subset G\cap B$, $d(B_1, B_2)=2(r-2r_1)$ and that the
segment connecting the centers of $B_1$ and $B_2$ intersects $\partial G$.
\end{lemma}

The following corollary is proved from this lemma by considering a sufficiently
small ball centered at a boundary point on the segment connecting the
centers of $B_1$ and $B_2$.

\begin{corollary}\label{j<<aCor}
If $G\psubset\Rn$ does not have the approximation
property, then for every $\epsilon>0$ there exists a point $z\in \partial G$, a real
number $r>0$ and
a unit vector $\theta\in S^{n-1}$ such that for every $w\in G^c\cap B^n(z,r)$
we have $\ip{w}{\theta}\le \epsilon r$.
\end{corollary}

It follows directly from the next theorem that Inequality~2 cannot occur.

\begin{theorem}\label{qiU=>apx}
If $G\psubset \Rn$ is quasi-isotropic and uniform, then
$G$ has the comparison property.
\end{theorem}

\begin{proof} Assume that $G$ is $L$--quasi-isotropic but does not have
the comparison property. We will show that this implies that $G$
is not uniform.

Let $0<\epsilon<1/(256 L^4)$ and choose $u\in \partial G$, $r>0$ and
$e\in S^{n-1}$ such that $\ip{v}{e}\le \epsilon r$ for all
$v\in G^c\cap B^n(u,r)$ (possible by Corollary~\ref{j<<aCor}).
We assume without
loss of generality that $u=0$, $r=1$ and $e=e_1$. Consider the points
$x:=\sqrt{\epsilon} e_1$ and $y:=-\sqrt{\epsilon} e_1$ and paths
connecting them in $G$. Let us denote
$D:=\{z\in B^n(0,\sqrt[4]{\epsilon}) \colon z_1=0\}$
and define $A$ to be the set of paths joining $x$ and $y$ in $G$ which
intersect $D$, and $B$ to be the set of paths joining $x$ and $y$ in $G$ which
do not intersect $D$.

\begin{figure}
\begin{center}
\includegraphics[width=6cm]{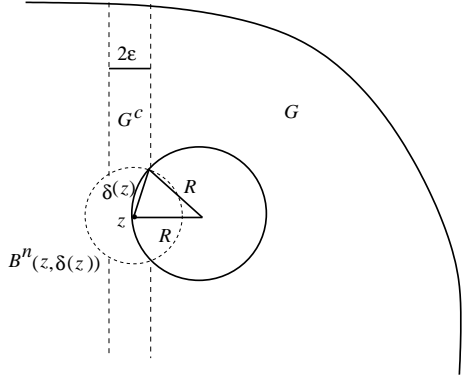}
\end{center}
\caption{The Apollonian sphere at $z$. }\label{qi2Fig}
\end{figure}

Let us consider first a path $\gamma \in A$ parameterized by arc length.
Let $z\in \gamma\cap D$,
$w\in S^{n-1}(z,\delta(z)) \cap \partial G$ and denote $\theta:=(w-z)/\abs{z-w}$.
Then $\bar\alpha_G(z; \theta) \ge 1/\abs{z-w} = 1/\delta(z)$.
Since $B^n(z,\delta(z))\subset G$ and $G^c \cap B^n$ is contained
within $\epsilon$ distance from the plane $P:=\{\xi\in\Rn\colon \xi_1=0\}$, we find
that the Apollonian spheres through $z$ in direction $e_1$ have radii
at least $\min\{\delta(z)^2/(2\epsilon), 1/4\}$, see Figure~\ref{qi2Fig}. Here the first
term comes from spheres limited by the boundary component near the plane $P$
(the case shown in the figure) and the second one comes from spheres
limited by $S^{n-1}$. It follows from this estimate that
$$ \frac{\bar\alpha_G(z;e_1)}{\bar\alpha_G(z; \theta)} \le
\max \{2\epsilon/\delta(z), 4\delta(z) \}. $$
Since $G$ is $L$--quasi-isotropic, this implies by Lemma~\ref{qiDef2} that
$\delta(z)\le 2L \epsilon$ or $4\delta(z) \ge 1/L$.
Since $z\in D$ and $0\in \partial G$, we have
$\delta(z) < \sqrt[4]{\epsilon} < 1/(4L)$  and so we see
that the second condition does not hold. This means that $\delta(z)\le 2L \epsilon$.
If $t_0$ is such that $z:=\gamma(t_0)\in D$, then it is clear that
$\min\{ t_0, \ell(\gamma)- t_0\} \ge d(x,D) = \sqrt{\epsilon}$. Therefore the
inequality $K\delta(\gamma(t_0))\ge \min\{ t_0, \ell(\gamma)- t_0\}$, which is the
second condition from the definition of uniformity, implies that
$K \ge \sqrt{\epsilon} / (2L \epsilon) = 1/(2L \sqrt{\epsilon})$.

On the other hand, for $\gamma \in B$ we have
$\ell(\gamma)/\abs{x-y} \ge 1/\sqrt[4]{\epsilon}$. Recall that
$\ell(\gamma)\le K \abs{x-y}$ is the first condition in the definition of uniformity. Thus
we see that as $\epsilon\to 0$ a path $\gamma$ will violate either the first
(if $\gamma\in B$) or the second condition (if $\gamma\in A$) of uniformity,
which means that $G$ is not uniform, as was to be shown.
%\endproof
 \end{proof}

%\begin{definition}\label{A-unifDef}
%A domain $G\psubset\Rn$ is $A$-{\it uniform} with constant $K$ if
%$k_G\le K\alpha_G$. A domain $G\psubset\Rn$ is said to be $A$-{\it
%uniform} if it is $A$-uniform with some constant $K<\infty$.
%\end{definition}

Using Theorem~\ref{qiU=>apx} we can prove the following
improvement of Proposition~6.6 from \cite{Ha2} which assumed the
comparison property instead of quasi-isotropy in item (2).

\begin{theorem}\label{AuEquivs}
Let $G\psubset\Rn$ be a domain. The following conditions are equivalent:
\begin{enumerate}
\item\label{Au1} $G$ is A-uniform (i.e.\ $k_G\ale \alpha_G$);
\item\label{Au2} $G$ is uniform and quasi-isotropic; and
\item\label{Au3} $G$ is Apollonian quasiconvex and quasi-isotropic.
\end{enumerate}
\end{theorem}

\begin{proof}
The three conditions can be written as
(1) $\alpha_G \aeq k_G$, (2) $j_G\aeq k_G$ and $k_G\aeq \ak_G$, and (3)
$\ak_G\aeq \alpha_G$ and $k_G\aeq \ak_G$, respectively.
If (1) holds, then $\alpha_G \aeq j_G \aeq \ak_G \aeq k_G$ and it is
clear that (2) and (3) hold. Assuming (3) and combining the two inequalities
we again get $\alpha_G \aeq k_G$, i.e.\ (1).
Finally, if (2) holds, then $G$ has the comparison property by
Theorem~\ref{qiU=>apx}, which means that $\alpha_G\aeq j_G$ and so
$\alpha_G\aeq k_G$ and all the metrics are again equivalent.
 \end{proof}

%%%%%%%%%%%%%%%%%%%%%%%%%%%%%%%%%%%%%%%%%
%%%%%%%%%%%%%%%%%%%%%%%%%%%%%%%%%%%%%%%%%
%%%%%%%%%%%%%%%%%%%%%%%%%%%%%%%%%%%%%%%%%

\section{Apollonian Quasiconvexity, other Constructions}
\label{constructSect}

In this section we show that Inequalities~9 (Table \ref{table1}),
$\alpha_G\aeq \ak_G\mle j_G\aeq k_G$, and 10, $\alpha_G\mle \ak_G\mle
j_G\aeq k_G$, can occur in general domains. Recall that we saw in
Lemma~\ref{not910} that these inequalities cannot occur in simply
connected planar domains.

%%%%%%%%%%%%%%%%%%%%%%%%%%%%%%%%%%%%%%%%%

\subsection{The ninth inequality}
\label{nineSubsect}

% We give an example which shows that Inequality~9,
% $\alpha_G\aeq \ak_G \mle j_G \aeq k_G$,
% can occur in a doubly connected planar domain or in a simply connected
% domain in $\Rn$ for $n\ge 3$.
In this subsection we are especially concerned with the relation
$\alpha_G\aeq \ak_G$\index{Apollonian quasiconvexity} (i.e.\ the question whether or not the
Apollonian metric is quasiconvex) to give brief description on
Inequality 9.
It was shown in \cite[Theorem~4.2]{Se} $j_G\le \alpha_G$ for convex $G$;
hence $\alpha_G\aeq j_G$ in convex domains. Recall also the
well-known fact that $j_G\aeq k_G$ if and only if $G$ is uniform.
Thus we conclude that $\alpha_G\aeq \ak_G$ holds for
all convex uniform $G$.
In \cite[Corollary~1.4]{HPWS} it was shown that
$\alpha_G\aeq \ak_G$ implies that $G$ is uniform.
On the other hand, there are also domains in which
$\alpha_G\mle \ak_G$; for example, convex domains that are not uniform.

We will now prove the inequality $\alpha_G\aeq \ak_G$ in
some set of domains. Unfortunately, we do not have a simple geometric
interpretation of this inequality, which means that the proof
is somewhat long. However, the structure is simple:
first we deal with the ``trivial''
cases, where the extra boundary point $p$ has no bearing
on the claim. In the other cases we construct a near-geodesic
path and estimate its length.

%\begin{proposition} In the domain $G:=B^n\setminus \{0\}$ we have
%$\alpha_G\aeq \ak_G \mle j_G \aeq k_G$.
%\end{proposition}
%
%\begin{figure}
%\begin{center}
%\includegraphics[width=6cm]{pathFig.eps}
%\end{center}
%\caption{A short path connecting $x$ and $y$. }\label{pathFig}
%\end{figure}

The general idea with the following theorem and its corollary
is that the inequality $\alpha_G\aeq \ak_G$ is not disturbed
by the addition of some boundary components of co-dimension
at least two, but does not hold for the addition of lower
co-dimension boundary components.

\begin{theorem}\label{1-pt-remove}
Let $D\psubset\Rn$ be a bounded domain. Suppose $p$ is a point in
$D$ and define $G:=D\setminus\{p\}$. If
$\alpha_D\aeq\ak_D$, then $\alpha_G\aeq \ak_G$ as well.
\end{theorem}

\begin{proof}
In this proof we denote by $\delta$ the distance to the
boundary of $D$, not of $G$.
We prove $\alpha_G\age\ak_G$ since $\alpha_G\le\ak_G$ always
holds. Let $x,y\in G$ and denote $B:=B^n(p,\delta(p)/2)$. Let
$\gamma_{xy}$ be a path connecting $x$ and $y$ such that
$\alpha_G(\gamma_{xy})=\ak_G(x,y)$. The existence of $\gamma_{xy}$
is due to \cite[Theorem~1.5]{Ha5}.

First consider the case $x,y\in D\setminus B$. If
$\gamma_{xy}\cap B=\emptyset$, we proceed as follows. Let
$z\in\partial G$ be such that $\delta(p)=|p-z|$. Denote
$R_D:=\diam D/\delta(p)$. For $w\in D\setminus B$ and $r\in S^1$,
we have
$$\bar\alpha_D(w;r)= \frac1{2r_-} + \frac1{2r_+} \ge
\frac{2}{\diam D}\ge\frac{1}{|w-p|R_D},
$$
where the last inequality holds since $|w-p|\ge\delta(p)/2$.
We also see that if the Apollonian spheres are affected by the boundary
point $p$, then
$$\bar\alpha_G(w;r)\le \frac{1}{|w-p|}+\frac{1}{2r_+}\le
\frac{1}{|w-p|}+\bar\alpha_D(w;r)
$$
holds, where $r_+$ denotes the radius of the Apollonian sphere
which touches $\partial D$. Otherwise, we have
$\bar\alpha_G(w;r)=\bar\alpha_D(w;r)$. So for all $w\in D\setminus
B$, the inequalities
\be\label{density-inq}\bar\alpha_G(w;r)\le\bar\alpha_D(w;r)+\frac{1}{|w-p|}\le (1+R_D)\bar\alpha_D(w;r) \ee hold.
By Lemma~\ref{innerThm} we get $\alpha_G(\gamma_{xy})\le
C\alpha_D(\gamma_{xy})$, for some constant $C$, which gives
\be\label{subd-eqv}
\ak_G(x,y)\ale\ak_D(x,y)\aeq\alpha_D(x,y)\le\alpha_G(x,y),
\ee
where the second inequality holds by assumption and the third
holds trivially, as $G$ is a subdomain of $D$.

If $\gamma_{xy}$
intersects $B$, let $\gamma$ be an intersecting part of
$\gamma_{xy}$ from $x_1$ to $x_2$ (if there are more intersecting
parts, we proceed similarly). Let $\gamma'$ be the shortest
circular arc on $\partial B$ from $x_1$ to $x_2$ as shown in the
Figure~\ref{XX}.
\begin{figure}
\begin{center}
\includegraphics[width=6cm]{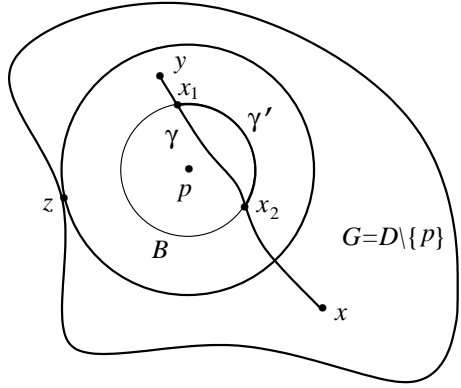}
%\vspace{-.5cm}
\end{center}
\caption{The geodesic path $\gamma_{xy}$ of $\ak_G$ connecting $x$ and $y$
intersects $B$.}\label{XX}
\end{figure}
Using the density bounds $2/\diam D\le \bar\alpha_D(u;r)\le 2/\delta(u)$,
we see that $\alpha_D(\gamma)\ge
2 \ell(\gamma)/\diam D$ and $\alpha_D(\gamma')\le
4\ell(\gamma')/\delta(p)$ hold. But since $\ell(\gamma)\ge
|x_1-x_2|$ and $\ell(\gamma')\le \frac{\pi}{2}|x_1-x_2|$, we have
$\ell(\gamma')\ale\ell(\gamma)$. This shows that
$\alpha_D(\gamma_{xy}')\ale\alpha_D(\gamma_{xy})$ holds. Since
$\gamma_{xy}'\subset G\setminus B$, (\ref{density-inq}) implies that
$\alpha_G(\gamma_{xy}')\ale \alpha_D(\gamma_{xy}')$. So we get
$$\alpha_G(x,y)\ge\alpha_D(x,y)\aeq\ak_D(x,y)=\alpha_D(\gamma_{xy})
\age\alpha_D(\gamma_{xy}')\age\alpha_G(\gamma_{xy}')\ge\ak_G(x,y).
$$
Thus we have shown that $\alpha_G(x,y)\age\ak_G(x,y)$ holds for
all $x,y\in D\setminus B$.

We now consider the case $x,y\in B^n(p,\frac{3}{4}\delta(p))$.
Without loss of generality we assume that $|y-p|\le|x-p|$. Since
$\partial G=\partial D\cup\{p\}$, it is clear that
$$\alpha_G(x,y)\ge
\max \Big\{ \log\frac{|x-p|}{|y-p|},\alpha_D(x,y)\Big\}.
% \ge\frac{1}{2}\left(\log\frac{|x-p|}{|y-p|}+\alpha_D(x,y)\right).
$$
Let $\gamma:=\gamma_1\cup\gamma_2$, where $\gamma_1$ is the path
which is circular about the point $p$ from $y$ to
$|y-p|\frac{x-p}{|x-p|} + p$ and
$\gamma_2$ is the radial part from
$|y-p|\frac{x-p}{|x-p|} + p$ to $x$, as shown
in the Figure~\ref{XXX}.
\begin{figure}
\begin{center}
\includegraphics[width=5cm]{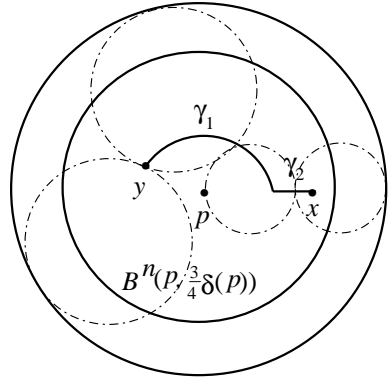}
%\vspace{1cm}
\end{center}
\caption{A short path connecting $x$ and $y$ in
$B^n(p,\frac{3}{4}\delta(p))$.}\label{XXX}
\end{figure}
Since the Apollonian spheres are not
affected by the boundary point $p$ in the circular part,
 we have
\[\begin{split}
\bar\alpha_G(\gamma_1(t);\gamma_1'(t))
\le\bar\alpha_{B^n(p,\delta(p))}(\gamma_1(t);\gamma_1'(t))
&= \frac{1}{\delta(p)-|y-p|}+\frac{1}{\delta(p)+|y-p|} \\
&=\frac{2\delta(p)}{\delta(p)^2-|y-p|^2},
\end{split}\]
where the first equality holds since the Apollonian metric is
isotropic in balls (since it equals the hyperbolic metric).
For $\gamma_2(t)$, by monotony in the
domain of definition, we see that
\[\begin{split}
\bar\alpha_G(\gamma_2(t);\gamma_2'(t))
& \le \bar\alpha_{B^n(p,\delta(p))\setminus\{p\}}(\gamma_2(t);\gamma_2'(t))\\
& = \frac{1}{|p-\gamma_2(t)|}+\frac{1}{\delta(p)-|p-\gamma_2(t)|}.
\end{split}\]
Hence we have
\begin{align*}
\ak_G(x,y)\le\alpha_G(\gamma)&\le\int_{\gamma_1}\frac{2\delta(p)}
{\delta(p)^2-|y-p|^2}\,|dy|
+\int_{|y-p|}^{|x-p|}\frac{1}{t}+\frac{1}{\delta(p)-t}\,dt\\
&=\frac{2\delta(p)\ell(\gamma_1)}{\delta(p)^2-|y-p|^2}+\log\bigg(\frac{|x-p|}{|y-p|}
\frac{\delta(p)-|y-p|}{\delta(p)-|x-p|}\bigg)\\
&\le\frac{32}{7}\frac{\ell(\gamma_1)}{\delta(p)}+\log\bigg(\frac{|x-p|}{|y-p|}
\frac{\delta(p)-|y-p|}{\delta(p)-|x-p|}\bigg).
\end{align*}
Since $u\mapsto u^3(\delta(p)-u)$ is increasing for
$0<u<3\delta(p)/4$ and we have $|y-p|\le |x-p|\le 3\delta(p)/4$,
the inequality $|x-p|^3(\delta(p)-|x-p|)\ge
|y-p|^3(\delta(p)-|y-p|)$ holds. This inequality is equivalent to
$$\log\bigg(\frac{|x-p|}{|y-p|}\frac{\delta(p)-|y-p|}{\delta(p)-|x-p|}\bigg)\le
4\log\frac{|x-p|}{|y-p|}.
$$
Using $\alpha_D\aeq\ak_D$ we easily get
$\alpha_D(x,y)\age \ell(\gamma_1)/\delta(p)$.
We have thus shown that
$$ \ak_G(x,y)\le K\alpha_D(x,y)+4\log\{|x-p|/|y-p|\}\le
(K+4)\alpha_G(x,y), $$
for some constant $K$.

It remains to consider the case $x\not\in B^n(p,3\delta(p)/4)$ and
$y\in B$. Let $w\in
S^{n-1}(p,\frac{3}{4}\delta(p))$ be such that
$|y-w|=d(y,S^{n-1}(p,\frac{3}{4}\delta(p)))$. Let
$\gamma:=\gamma_1\cup\gamma_2$, where $\gamma_1=[y,w]$ and
$\gamma_2$ is a path connecting $w$ and $x$ such that
$\alpha_G(\gamma_2)=\ak_G(x,w)$. As we discussed in the previous
case, we have
\[
\alpha_G(\gamma_1)\le 4\log \frac{3\delta(p)}{4|y-p|}
\le 4\log\frac{|x-p|}{|y-p|}\le 4\alpha_G(x,y).\]
Since $x,w\not\in B$, it follows by the previous cases
that
$$\alpha_G(\gamma_2)=\ak_G(w,x)\ale \alpha_D(w,x)\le 2j_D(w,x).
$$
% It is not difficult to see that $j_D(x,w)\ale \alpha_D(x,y)$,
% details are again provided in \cite{Sa}.
If $\delta(w)\le \delta(x)$, using the first inequality of
(\ref{subd-eqv}) and the triangle inequality we see that the
inequalities
$$\alpha_G(\gamma_2)=\ak_G(w,x)\ale \ak_D(w,x)\aeq\alpha_D(w,x)\le
2j_D(w,x)\le 2\log\left(4+\frac{4|x-p|}{\delta(p)}\right)
$$
hold. But for $s\ge 3/2$, we have $\log(4+2s)\le 5\log s$. Since
$|y-p|\le\delta(p)/2$, we obtain
\[\label{gamma2}\alpha_G(\gamma_2)\ale
\log\left(4+\frac{4|x-p|}{\delta(p)}\right)\le
5\log\frac{|x-p|}{|y-p|} \le 5\alpha_G(x,y).\]
We then move on to the case $\delta(w)\ge \delta(x)$.
If $|x-y|\ge 3\delta(x)$, we see (by the triangle inequality) that
$$\alpha_G(x,y)\ge \sup_{b\in\partial D}\log
\frac{|x-y|-|b-x|}{|b-x|}=\log\left(\frac{|x-y|}{\delta(x)}-1\right)
$$
holds. Using this and the fact that $\frac{|x-p|}{|y-p|}\ge \log (3/2)$
we get
$$\alpha_G(x,y)\ge\left\{ \begin{array}{ll}
\log\left(\displaystyle\frac{|x-y|}{\delta(x)}-1\right) &
\mbox{for $\displaystyle\frac{|x-y|}{\delta(x)}\ge 3$},\\
\log\displaystyle\frac32 &
\mbox{otherwise}.\end{array}\right.
$$
Since $|x-y|\ge \delta(p)/4$, we get the following upper bound
for the length of the curve (the first inequality follows as before):
$$\alpha_G(\gamma_2)\ale j_D(x,w) \ale
\log\left(1+\frac{|x-y|+\delta(p)}{\delta(x)}\right)
\le \log\left(1+\frac{5|x-y|}{\delta(x)}\right).
$$
The function $f(s)=(s-1)^4-(1+5s)$ is increasing for $s\ge 3$, so
$f(s)\ge f(3)=0$. Thus for $|x-y|/\delta(x)\ge 3$, we get
\[\label{gamma2'}\alpha_G(\gamma_2)\ale
\log\left(1+\frac{5|x-y|}{\delta(x)}\right)\le 4\log
\left(\frac{|x-y|}{\delta(x)}-1\right)\le 4\alpha_G(x,y).\]
On the other hand, if $|x-y|/\delta(x)< 3$, then $\alpha_G(\gamma_2)$
is bounded above by $4\log 2$ and $\alpha_G(x,y)$ is bounded below
by $\log(3/2)$, so the inequality $\alpha_G(\gamma_2)\ale \alpha_G(x,y)$ is
clear.
We have now verified the inequality in all the possible cases, so the
proof is complete.
\end{proof}
%Moreover, if $s<3$, then $\log(1+5s)\le 4\log 2$. This gives us
%$$2\log\left(1+\frac{5|x-y|}{\delta(x)}\right)\le \frac{8\log 2}{\log 3/2}
%\log\frac{3\delta(p)}{4|y-p|}\le \frac{8\log 2}{\log 3/2}
%\log\frac{|x-p|}{|y-p|},
%$$
%where the first inequality holds since $2|y-p|\le \delta(p)$ and
%second holds as $3\delta(p)\le 4|x-p|$. Thus for
%$\delta(x)<\delta(w)$ and $|x-y|/\delta(x)<3$, we obtain
%\be\label{gamma2''}\alpha_G(\gamma_2)\le K_3\alpha_G(x,y),\ee for
%some constant $K_3$.
%Thus from equations (\ref{gamma1}) to
%(\ref{gamma2''}) it is clear that $\ak_G(x,y)\ale\alpha_G(x,y)$.
%This completes the proof.
% \end{proof}

\begin{corollary}\label{1-pt-remove-cor}
Let $D\psubset\Rn$ be a bounded domain. Suppose $(p_i)_{i=1}^k$ is a finite
non-empty sequence of points in $D$ and define $G:=D\setminus
\{p_1,p_2,\ldots,p_k\}$. Assume that $\alpha_D\aeq\ak_D$ and
$j_D\aeq k_D$. Then Inequality 9, $\alpha_G\aeq\ak_G\mle j_G\aeq
k_G$, holds.
\end{corollary}

\begin{proof}
Since $j_D\aeq k_D$, $D$ is uniform and thus so $G$ is, as can be seen
from the definition, which implies that $j_G\aeq k_G$ holds. Let
$\epsilon_0>0$ be such that the sphere
$S^{n-1}(p_1,\epsilon)\subset D$ for all
$\epsilon\in(0,\epsilon_0)$. Let $x,y\in S^{n-1}(p_1,\epsilon)$ be
diametrically opposite. Then we see that $\alpha_G(x,y)\to 0$ as
$\epsilon\to 0$, but on the other hand $j_G(x,y)=\log 3$. Hence
$\alpha_G(x,y)/j_G(x,y)\to 0$ as $\epsilon\to 0$, which implies
that $\alpha_G\mle j_G$. We have thus proved that $\alpha_G\mle
j_G\aeq k_G$. So, it remains to prove $\alpha_G\age\ak_G$, since
$\alpha_G\le\ak_G$ always holds. For $1\le i\le k$, define
$G_i=G_{i-1}\setminus \{p_i\}$, where $G_0=D$. Since $D$ is
bounded and $\alpha_D\aeq\ak_D$, we
conclude by Theorem~\ref{1-pt-remove} that $\alpha_{G_1}\aeq\ak_{G_1}$. Inductively, we get
$\alpha_{G_i}\aeq\ak_{G_i}$ for all $i$, $1\le i\le k$. Since
$G_k=D\setminus \{p_1,\ldots,p_k\}=G$, we have shown that $\alpha_G\aeq\ak_G$.
\end{proof}

% Before generalizing Corollary \ref{1-pt-remove-cor} to unbounded domain,
% we present the following result as a generalization of
% Theorem \ref{1-pt-remove}
% to unbounded domain.
One ingredient in the proofs of some of the inequalities
in Theorem \ref{1-pt-remove} was the following reformulated
result, which shows that
removing a point from the domain (i.e.\ adding a boundary point)
does not affect the inequality $\alpha_G\approx \tilde\alpha_G$:

\begin{theorem}
Let $D\psubset\Rn$ be a domain with an exterior point.
Let $p \in D$ and $G:=D\setminus\{p\}$. If
$\alpha_D\aeq\ak_D$, then $\alpha_G\aeq \ak_G$ as well.
\end{theorem}

Note that by M\"obius invariance, one may assume that
the exterior point is in fact $\infty$, in which case the
domain is bounded, as was the assumption in the original source.
This assumption was of a
technical nature, and in the following result we show that
indeed it can be replaced by a
much weaker assumption that the complement of $D$ is not
contained in a hyperplane. This assumption is sharp, for
if $\partial G\subset \R^{n-1}$, then $\alpha_G(x,\tilde x) = 0$,
where $\tilde x$ is the reflection of $x$ in $\R^{n-1}$, but
$\tilde \alpha_G(x,\tilde x) > 0$, so that
$\alpha_G \not\approx \tilde \alpha_G$ and the theorem becomes
vacuous.

\begin{theorem}\label{hps2-main}
% Let $D\psubset\Rn$ be an unbounded domain such that there exists
% $n+1$ finite boundary points which form extreme points of an
% $n$-simplex, that is, they do not lie in a hyperplane.
% Suppose that $p$ is a point in $D$ and define $G:=D\setminus\{p\}$. If
% $\alpha_D\aeq\ak_D$, then $\alpha_G\aeq \ak_G$ as well.
Let $D\psubset\Rn$ be a domain whose boundary is not contained in
a hyperplane. Let $p \in D$ and $G:=D\setminus\{p\}$. If
$\alpha_D\aeq\ak_D$, then $\alpha_G\aeq \ak_G$ as well.
\end{theorem}
\begin{proof}
In this proof we denote by $\delta$ the distance to the
boundary of $D$, but not of $G$. It is enough to prove the inequality
$\alpha_G\age\ak_G$, because
other way inequality always holds. Let $x,y\in G$ and denote
$B:=B^n(p,\delta(p)/2)$. Let $\gamma_{xy}$ be a path connecting $x$ and
$y$ such that $\alpha_G(\gamma_{xy})=\ak_G(x,y)$, see the definition of the
inner metric.

Let us first consider the case when $x,y\in D\setminus B$ and
$\gamma_{xy}\cap B=\emptyset$.

Let $z\in\partial D$ be such that $\delta(p)=|p-z|$. Let $S$ be the
collection of $n+1$ boundary points of $D$ where they form the
vertices of an $n$-simplex. Denote by $B_t:=B^n(c,t)$ the largest ball
with radius $t$ and centered at $c$ such that $B_t$ is inside the
$n$-simplex $[S]$. Define $l=t/2$. Denote by $B_l:=B^n(c,l)$ the ball
with radius $l$ and centered at $c$. Define $\Omega=\Rn\setminus S$. Let
$B_T\subset \Omega$ be the largest ball with radius $T$ and tangent to
$B_l$, see Figure \ref{hps2-fig1}.

\begin{figure}
\begin{center}
\includegraphics[width=6cm]{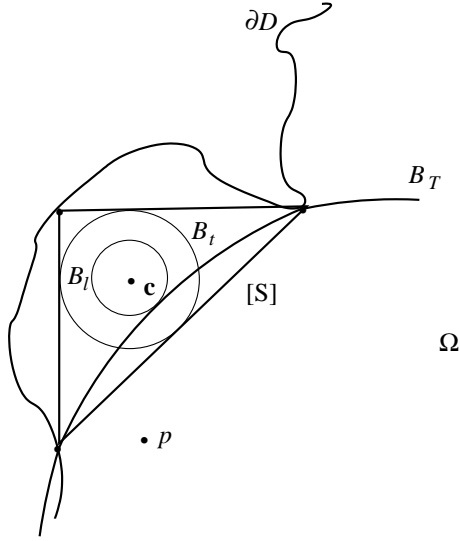}
%\vspace{-.5cm}
\end{center}
\caption{The largest ball $B_T$ tangent to $B_l$ and contained in
$\Omega=\Rn\setminus S$.}\label{hps2-fig1}
\end{figure}

Choose $L=5\max\{|p-c|,T\}$. Consider the ball $B^n(c,L)$
centered at $c$ with radius $L$ and denote it by $B_L$. Then we see that
$S\cup\{\infty\}\subset \partial D$. Since
$$\partial \Omega = \partial (\Rn\setminus S)=S\cup\{\infty\}\subset \partial D,
$$ we see that
\begin{equation}\label{1-pt-remove-unbdd-eq1}
\bar\alpha_D(w;r)\ge \bar\alpha_\Omega(w;r)
\end{equation}
for $r\in S^{n-1}$.

For the moment we assume $w\in \Rn\setminus B_L$. We now estimate the
density of the Apollonian spheres in $\Omega$ passing through $w$ and
in the direction $r$. We denote $F$ the ray from $w$ along the direction of $r$.
Consider a
sphere $S_1$ with radius $R_1$ and centered at $x\in F$ such that $S_1$
is tangent to $B_l$. Denote $\theta:=\widehat{xwc}$. Construction of
$B_T$ gives that, for $|\theta|<\pi/2$ we see that the Apollonian
spheres passing through $w$ and in the direction $r$ are smaller in
size than the sphere $S_1$.
This gives
\begin{equation}\label{1-pt-remove-unbdd-eq2}
\bar\alpha_\Omega(w;r) \ge \frac{1}{2R_1} =
\frac{l+(d+l)\cos\theta}{(d+l)^2-l^2},
\end{equation}
where $d:=d(w,B_l)$ and $R_1$ is obtained using the cosine formula for
the triangle $\Delta \,xwc$.
%$$2R_1=\frac{(d+l)^2-l^2}{l+(d+l)\cos\theta}.
%$$

Now the ball with radius $R_2$ and center at $q$ passing through $w$
and $p$ gives
$$\frac{|w-p|}{2}=R_2\cos(\theta-\psi),
$$
where $\psi=\widehat{cwp}$ and $q\in F$. If the Apollonian spheres are
affected by the boundary point $p$ then
\begin{eqnarray*}
\bar\alpha_G(w;r) & = & \frac{1}{R_2}+\frac{1}{r_+}\\
& \le &\frac {1}{R_2} + \bar\alpha_D(w;r)\\
& = & \frac{2\cos (\theta-\psi)}{|w-p|}+\bar\alpha_D(w;r)
\end{eqnarray*}
hold, where $r_+$ denotes the radius of the smaller Apollonian sphere
which
touches $\partial D$. Denote $\phi:=\widehat{wpc}$. Since $p\in B_L$,
using the sine formula in the triangle $\Delta\,wpc$ we get
$$\sin \psi=\frac{|p-c|}{|w-p|}\sin \phi\le \frac{|p-c|}{|w-p|}.
$$
Then we see that
$$\cos (\theta-\psi)\le \cos\theta + \sin\psi\le \cos\theta + \frac{|p-c|}{|w-p|}.
$$
 Thus we get
\begin{eqnarray*}
\bar\alpha_G(w;r) & \le & \frac{2\left(\cos\theta
+ |p-c|/|w-p|\right)}{|w-p|}+ \bar\alpha_D(w;r)\\
& = & \frac{2\cos\theta}{|w-p|}+ \frac{2|p-c|}{|w-p|^2}+
\bar\alpha_D(w;r).
\end{eqnarray*}
Since $w\not\in B_L$, triangle inequality gives $|w-p|\approx d$.
We then obtain
\begin{eqnarray}\label{1-pt-remove-unbdd-eq5}
\nonumber\bar\alpha_G(w;r) & \ale & \frac{\cos\theta}{d}+\frac{l}{d^2}
+ \bar\alpha_D(w;r)\\
& \approx & \frac{l+(d+l)\cos\theta}{(d+l)^2-l^2}
+\bar\alpha_D(w;r).
\end{eqnarray}

Now we see that if $\bar\alpha_\Omega(w;r)=0$ then $\partial \Omega$ is
contained in a hyperplane, which contradicts our assumption. Thus if
$w\in \overline{B_L}$ then $\bar\alpha_\Omega(w;r)>0$ and since the
density function is continuous it has a greatest lower bound, i.e.
there exists a constant $k>0$ such that for $r\in S^1$ we have
\begin{equation}\label{1-pt-remove-unbdd-eq3}
\bar\alpha_\Omega(w;r)\ge k.
\end{equation}
Therefore, (\ref{1-pt-remove-unbdd-eq2}) and
(\ref{1-pt-remove-unbdd-eq3})
together give
\begin{equation}\label{1-pt-remove-unbdd-eq4}
\bar\alpha_\Omega(w;r)\ge
\min\left\{\frac{l+(d+l)\cos\theta}{(d+l)^2-l^2},k\right\}.
\end{equation}

In $B_L$, we also see that if the Apollonian spheres are affected by the
boundary point $p$ then
\begin{eqnarray}\label{1-pt-remove-unbdd-eq6}
\nonumber\bar\alpha_G(w;r) & \le &
\frac{1}{|w-p|}+\frac{1}{2r_+}\\
\nonumber & \le & \frac{1}{|w-p|}+\bar\alpha_D(w;r)\\
\nonumber & \le & \frac{2}{\delta(p)}+\bar\alpha_D(w;r)\\
& \approx & k +\bar\alpha_D(w;r),
\end{eqnarray}
where $r_+$ denotes the radius of the smaller Apollonian sphere which
touches
$\partial D$. Thus (\ref{1-pt-remove-unbdd-eq1}),
(\ref{1-pt-remove-unbdd-eq5}), (\ref{1-pt-remove-unbdd-eq4}) and
(\ref{1-pt-remove-unbdd-eq6}) give
\begin{eqnarray}\label{density-inq2}
\nonumber\bar\alpha_G(w;r) & \ale &
\min\left\{\frac{l+(d+l)\cos\theta}{(d+l)^2-l^2},k\right\}
+ \bar\alpha_D(w;r)\\
& \ale & \bar\alpha_D(w;r).
\end{eqnarray}
%Then inequality (\ref{1-pt-remove-unbdd-eq1}) gives
%\be\label{density-inq2}
%\bar\alpha_G(w;r) \ale \bar\alpha_D(w;r)
%\ee
Hence we get the relation
$\alpha_G(\gamma_{xy})\le K\alpha_D(\gamma_{xy})$
for some constant $K$. This gives
\begin{equation}\label{1-pt-remove-unbdd-eq8}
\ak_G(x,y)\ale\ak_D(x,y)\aeq\alpha_D(x,y)\le\alpha_G(x,y),
\end{equation}
where the second inequality holds by assumption and the third
holds trivially, as $G$ is a subdomain of $D$.

Let us now consider the case when $x,y\in D\setminus B$ and
$\gamma_{xy}$ intersects $B$.

Let $\gamma$ be an intersecting
part of $\gamma_{xy}$ from $x_1$ to $x_2$ (if there are more
intersecting parts, we proceed similarly). Let $\gamma'$ be the shortest
circular arc on $\partial B$ from $x_1$ to $x_2$,
as shown in the Figure~\ref{HPS2-fig2}.

\begin{figure}
\begin{center}
\includegraphics[width=6cm]{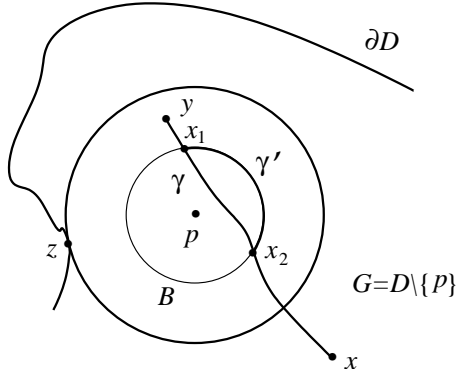}
\vspace{-.5cm}
\end{center}
\caption{The geodesic path $\gamma_{xy}$ of $\ak_G$ connecting $x$ and $y$
intersects $B$.}\label{HPS2-fig2}
\end{figure}

Using the density bounds (\ref{1-pt-remove-unbdd-eq1}) and
(\ref{1-pt-remove-unbdd-eq3}) we get
$k \le \bar\alpha_D(u;r)\le 2/\delta(u).
$
Then we see that
$$\alpha_D(\gamma)\ge \ell(\gamma)/k ~\mbox{ and }~ \alpha_D(\gamma')\le
4\ell(\gamma')/\delta(p)
$$
hold. But since $\ell(\gamma)\ge
|x_1-x_2|$ and $\ell(\gamma')\le \frac{\pi}{2}|x_1-x_2|$, we have
$\ell(\gamma')\ale\ell(\gamma)$. This shows that
$\alpha_D(\gamma_{xy}')\ale\alpha_D(\gamma_{xy})$ holds. Since
$\gamma_{xy}'\subset G\setminus B$, (\ref{density-inq2}) implies that
$\alpha_G(\gamma_{xy}')\ale \alpha_D(\gamma_{xy}')$. So we get
$$\alpha_G(x,y)\ge\alpha_D(x,y)\aeq\ak_D(x,y)=\alpha_D(\gamma_{xy})
\age\alpha_D(\gamma_{xy}')\age\alpha_G(\gamma_{xy}')\ge\ak_G(x,y).
$$
Thus we have shown that $\alpha_G(x,y)\age\ak_G(x,y)$ holds for
all $x,y\in D\setminus B$.

Proof of rest of the cases are same as in that of Theorem \ref{1-pt-remove}.
\end{proof}

Of course, we can iterate Theorem \ref{hps2-main}, to remove any finite set
of points from our domain. Exactly like in Corollary \ref{1-pt-remove-cor},
we get

\begin{corollary}
% Let $D\psubset\Rn$ be an unbounded domain such that there exists
% $n+1$ finite boundary points which form extreme points of an
% $n$-simplex, that is, they do not lie in a hyperplane. Suppose
% $(p_i)_{i=1}^k$ is a finite
% non-empty sequence of points in $D$ and define $G:=D\setminus
% \{p_1,p_2,\ldots,p_k\}$. Assume that $\alpha_D\aeq\ak_D$ and
% $j_D\aeq k_D$. Then Inequality 9, $\alpha_G\aeq\ak_G\mle j_G\aeq
% k_G$, holds.
Let $D\psubset\Rn$ be a domain whose boundary does
not lie in a hyperplane. Suppose $(p_i)_{i=1}^k$ is a finite
non-empty sequence of points in $D$ and define $G:=D\setminus
\{p_1,p_2,\ldots,p_k\}$. Assume that $\alpha_D\aeq\ak_D$ and
$j_D\aeq k_D$. Then Inequality 9, $\alpha_G\aeq\ak_G\mle j_G\aeq
k_G$, holds.
\end{corollary}

%%%%%%%%%%%%%%%%%%%%%%%%%%%%%%%%%%%%%%%%%

\subsection{The tenth inequality}
\label{tenSubsect}

In this subsection we construct a domain in $\R^3$ which is topologically
equivalent (in $\overline{\R^3}$) to a ball in which the inequalities
$\alpha_G\mle \ak_G \mle j_G \aeq k_G$ hold.

\begin{proposition} Define $R_1:=  \{e_1 + t e_3\colon t\in [0,\infty)\}$
and $R_2:= \{ -e_1 + t e_3\colon t\in [0,\infty)\}$. In the domain
$$ G:= \R^3\setminus \left(R_1\cup R_2 \cup \overline{B^3}(e_1-e_3,1)\right) $$
Inequality~10, $\alpha_G\mle \ak_G \mle j_G \aeq k_G$, holds.
\end{proposition}

\begin{proof}
It is easy to see that $G$ is uniform, if we handle the cases when $\abs{x-y}$ is
small and when it is large separately. In the former case, both $x$ and $y$ are
near a single boundary component of $G$ and hence we need only consider the
boundary components
separately. If $\abs{x-y}$ is large, then we may choose the path to curve away
from all boundary components.
Since $G$ contains a boundary component which is a ray, it is clear that
$G$ is not quasi-isotropic, hence $\ak_G\mle k_G$ by Lemma~\ref{qiEquivs} and
it remains only to prove that $\alpha_G\mle \ak_G$.

Set $x:= te_3+2e_2$ and $y:=t e_3 - 2e_2$ for $t\in [3/2,\infty)$.
It is clear that $\alpha_G(x,y)\to 0$ as $t\to \infty$, since the
rays $R_1$ and $R_2$ do not affect this distance. We next derive
a lower bound for $\ak_G(x,y)$ which is independent of $t$.
In so doing we may forget about the boundary points in the sphere
$S^2(e_1-e_3,1)$ since this only makes the bound smaller.
Denote $B^x:=B^3(x,1)$ and $B^y:=B^3(y,1)$.
Let $u\in B^x,~\theta\in S^1$ and denote $d:=d(u,\partial
B^y)$.  We see that any ball
which intersects $B^x$ and $B^y$ will also intersect $R_1$ or
$R_2$. Thus the Apollonian spheres through $u$ in direction $\theta$
with respect to $G$ are smaller in
size than with respect to $\R^3\setminus B^y$. Since $\R^3\setminus
B^y$ is isotropic, we get
$$\bar\alpha_G(u;\theta)\ge\bar\alpha_{\R^3\setminus B^y}(u;\theta)
=\bar\alpha_{\R^3\setminus B^y}(u)=\frac{1}{d}-\frac{1}{d+2}\ge\frac{1}{12},
$$
where we used $d\le 4$ for the last inequality (the minus sign occurs because the
corresponding Apollonian sphere contains $\infty$).
Let $\gamma$ be a path connecting $x$ and $y$. Then it certainly
connects $x$ to $\partial B^x$; denote this part by $\gamma'$.
By Theorem~\ref{innerThm} we get $\ak_G(x,y)\ge
\frac{1}{12}\inf_{\gamma'}\ell(\gamma') = \frac{1}{12}$.
Since $\alpha_G(x,y)\to 0$ as $t\to 0$, we see
that $\alpha_G\mle\ak_G$.
\end{proof}

%%%%%%%%%%%%%%%%%%%%%%%%%%%%%%%%%%%%%%%%%
%%%%%%%%%%%%%%%%%%%%%%%%%%%%%%%%%%%%%%%%%
%%%%%%%%%%%%%%%%%%%%%%%%%%%%%%%%%%%%%%%%%

\section{Noncomparability}\label{noncompSect}
In this short section we sort out the possible inequalities when $j_G$ and $\ak_G$
are not comparable. It turns out that there is just one possibility in this case.
For if $j_G\notcomp \ak_G$, then it follows without any geometrical
considerations that none of the inequalities $\alpha_G\aeq j_G$,
$\alpha_G\aeq \ak_G$, $j_G\aeq k_G$ or $\ak_G\aeq k_G$ can hold, since
if for instance $\alpha_G\aeq j_G$, then
$j_G\aeq \alpha_G \ale \ak_G$, contrary to assumption. Hence only the
possibility $\alpha_G\mle j_G\mle k_G$ and $\alpha_G\mle \ak_G\mle k_G$
remains, which is the case of least possible comparability among the
metrics. Unfortunately, this occurs in quite many domains.

\begin{proposition} Let $G$ be a simply connected planar domain which is
not quasi-isotropic. Then
$\alpha_G\mle j_G\mle k_G$, $\alpha_G\mle \ak_G\mle k_G$ and
$j_G \notcomp \ak_G$.
\end{proposition}

\begin{proof} Since $G$ is not quasi-isotropic, we have $j_G \not\ale \ak_G$
by Corollary~\ref{qiEquivs}. Since $G$ is simply connected and does
not have the comparison property, the inequality $j_G \not\age \ak_G$
follows from Proposition~\ref{twelveProp}. Hence $j_G \notcomp \ak_G$ which
implies the inequalities $\alpha_G\mle j_G\mle k_G$ and
$\alpha_G\mle \ak_G\mle k_G$, as was shown above.
%\endproof
 \end{proof}

\begin{example} The domain $H^2\setminus [0,e_2]$ satisfies
the assumptions of the previous lemma.
\end{example}

%\vskip 1cm
%This chapter has been organized from the articles\\
%{\bf H\"{a}st\"{o}, P., Ponnusamy, S. and Sahoo, S.K.} (2006)
%Inequalities and geometry of the Apollonian and related metrics.
%{\em Rev. Roumaine Math. Pures  Appl.} \textbf{51}(4), 433--452.\\
%and\\
%{\bf H\"{a}st\"{o}, P., Ponnusamy, S. and Sahoo, S.K.},
%Equivalence of the Apollonian and its inner metric.
%Submitted to the {\em Proceedings of the International Conference on
%Analysis and Mathematical Physics} edited by
%Gustafsson, B. and Vasi\'lev, A.
%%%%%%%%%%%%%%%%%%%%%%%%%%%%%%%%%%%%%%%%%
%%%%%%%%%%%%%%%%%%%%%%%%%%%%%%%%%%%%%%%%%
%%%%%%%%%%%%%%%%%%%%%%%%%%%%%%%%%%%%%%%%%

\chapter{UNIFORM, JOHN AND QUASI-ISOTROPIC DOMAINS}\label{chap3}
This chapter concerns characterization of
uniform domains in terms of the lower bound of the $\lambda$-Apollonian metric.
In addition, we consider relationship
between quasi-isotropic domains and quasidisks. In Section
\ref{chap3-sec1}, we recall a number of results which
motivate us to investigate the main contexts of this chapter.
This section also includes basic definitions, notational
descriptions, and some elementary results for proving our main results.
In Section \ref{chap3-sec2}, we state and prove a few technical lemmas and
as a consequence, we establish the proof of Theorem
\ref{thm1.4} and solutions to  a number of related questions in terms of
examples.
%\ref{thm1.5}--\ref{thm1.7}.
Finally, Section \ref{chap3-sec4} is devoted to discussions on
relationship between simply connected quasi-isotropic domains
and John disks.
%\ref{thm1.8} and \ref{thm1.9}.

\vskip 1cm
Results of this chapter are from the published article:
{\bf M. Huang, X. Wang, S. Ponnusamy and S.K. Sahoo} (2008)
Uniform domains, John domains and quasi-isotropic domains.
{\em J. Math. Anal. Appl.} {\bf 343}, 110--126.

\section{Introduction and Preliminaries}\label{chap3-sec1}
Throughout the chapter, we always assume that $D$ is a proper
subdomain of the complex plane $\mathbb{C}$ possessing at least two
finite boundary points, and that constants such as $b$ and $c$ are
positive.
% As in \cite{Vai}, a domain $D$ is called uniform provided there
% exists a constant $c$ with the property that each pair of points
% $z_{1},z_{2}\in D$ can be joined by a rectifiable arc $\gamma\subset D$
% satisfying
% \begin{enumerate}
% \item $\ell(\gamma)\leq c\,|z_{1}-z_{2}|$,
% \item $\ds\min_{j=1,2}\ell (\gamma [z_j, z])\leq c\, \dist (z, \partial D)$,
% for all $z\in \gamma$.
% \end{enumerate}
% Here $\ell(\gamma)$ denotes the Euclidean length of $\gamma$,
% $\gamma[z_{j},z]$ the part of $\gamma$ between $z_{j}$ and $z$.

As in \cite{NV}, a simply connected domain $D$ is called a $b$-John
disk if for any two points $z_{1},z_{2}\in D$, there is a
rectifiable arc $\gamma\subset D$ joining them with
$$\min_{j=1,2}\ell (\gamma [z_j, z])\leq b\; \dist (z, \partial D)
$$
for all $z\in \gamma$, where $b$ is a constant. Sometimes we simply call
$D$ a {\it John disk}\index{John disk}\index{domain!John}
if it is a $b$-John disks for some positive
constant $b$.

The class of so-called John domains and uniform domains enjoy an
important role in many areas of modern mathematical analysis, see
\cite{MS,NV,Vai}. Martio and Sarvas \cite{MS} were the first who
introduced uniform domains and since then its importance along with
John domains throughout the function theory is well documented, see
\cite{Ge2,Vai}. It is well-known that a simply connected planar
domain $D$ is a quasidisk if and only if $D$ is a uniform domain
(see \cite[Lemma 6.4]{Ha2}); a Jordan domain $D$ is a quasidisk if
and only if both $D$ and $D^{*}:=\mathbb{C} \backslash \overline{D}$
are John disks, and every quasidisk is a John disk (see \cite{Kil}).
Hence John disks can be thought of as ``one-sided quasidisks".

For any $z_1,$ $z_2\in D$, the {\it $\lambda$-length} \index{$\lambda$-length}
\cite{Bro1,Bro} between them is defined by
$$\lambda_D(z_1,z_2)=\inf \{\ell(\gamma):\; \gamma\subset D\;
\mbox{is a rectifiable arc joining}\; z_1\; \mbox{and}\; z_2 \}.
$$

A point $w$ in the boundary $\partial D$ of $D$ is said to be {\it
rectifiably accessible}\index{rectifiably accessible}
if there is a half open rectifiable arc
$\gamma$ in $D$ ending at $w$. Let $\partial_{r}D$ denote the subset
of $\partial D$ which consists of all the rectifiably accessible
points, that is
$$\partial_{r}D=\{w \in \partial D:w\; \mbox{ is
rectifiably accessible}\}.
$$
% Beardon \cite{Be} has studied the Apollonian metric $\alpha_D$ defined by
% (see also \cite{GH,Ha2,Ha5})
% $$\alpha_D(z_1,z_2)=\sup_{w_1,w_2\in \partial D}
% \log |z_1,z_2,w_1,w_2|.
% $$
The {\it $\lambda$-Apollonian metric}
\index{$\lambda$-Apollonian metric}\index{metric!$\lambda$-Apollonian}
$\alpha'_{D}$ (see \cite{Bro,Wap}) is defined by
$$\alpha'_D(z_1,z_2)=\sup_{w_1,w_2\in \partial_r D}
\log |z_1,z_2,w_1,w_2|_{\lambda_D}.
$$
Here
% $$|z_1,z_2,w_1,w_2|=\frac{|z_{1}-w_{1}|\,|z_{2}-w_{2}|}
% {|z_{1}-w_{2}|\,|z_{2}-w_{1}|}
% $$
% and
$$|z_1,z_2,w_1,w_2|_{\lambda_D}=
\frac{\lambda_D(z_1,w_1)\lambda_D(z_2, w_2)}
{\lambda_D(z_1, w_2)\lambda_D(z_2, w_1)}
$$
for all $z_1,$ $z_2\in D$.
%and so we omit the details.

\begin{lemma}\label{lem2.1}
For all $z_1$, $z_2\in D$ we have
$$\alpha_D(z_1,z_2)\leq \alpha'_D(z_1,z_2).
$$
\end{lemma}
\begin{proof}
From the properties of the Apollonian balls approach
(see Subsection \ref{ABA}) we have
\begin{equation}\label{qxge1}
\sup_{w_1\in\partial D}\frac{|z_1-w_1|}{|z_2-w_1|}\ge 1,
\end{equation}
for $z_1,z_2\in D$. Suppose $w\in\partial D$ is a point such that
\begin{equation}\label{qxge1-eq1}
\sup_{w_1\in\partial D}\frac{|z_1-w_1|}{|z_2-w_1|}=
\frac{|z_1-w|}{|z_2-w|}.
\end{equation}
Then we must have
$$\lambda_D(z_2,w)=|z_2-w|.
$$
For a proof, we assume a contradiction.
Then there exists a point $w'$ (different from $w$) in $\partial D$ such that
\begin{equation}\label{qxge1-eq2}
|z_2-w|=|z_2-w'|+|w'-w|
\end{equation}
with $\lambda_D(z_2,w')=|z_2-w'|$.
Now, by the triangle inequality, we see that
$$|z_1-w'|\ge |z_1-w|-|w-w'|=|z_1-w|-|z_2-w|+|z_2-w'|.
$$
This gives that
\begin{eqnarray*}
\frac{|z_1-w'|}{|z_2-w'|}
& \ge & 1+\frac{|z_1-w|-|z_2-w|}{|z_2-w'|}\\
& > & 1+\frac{|z_1-w|-|z_2-w|}{|z_2-w|}\\
& = & \frac{|z_1-w|}{|z_2-w|},
\end{eqnarray*}
where the strict inequality holds by (\ref{qxge1}) and (\ref{qxge1-eq2}),
which is a contradiction due to (\ref{qxge1-eq1}). Thus, we conclude that
\begin{equation}\label{qxge1-eq3}
\sup_{w_1\in\partial D}\frac{|z_1-w_1|}{|z_2-w_1|}\le
\sup_{w_1\in\partial_r D}\frac{\lambda_D(z_1,w_1)}{\lambda_D(z_2,w_1)},
\end{equation}
because $|z_1-w_1|\le \lambda_D(z_1,w_1)$ always holds. Similarly, we obtain
\begin{equation}\label{qxge1-eq4}
\sup_{w_2\in\partial D}\frac{|z_2-w_2|}{|z_1-w_2|}\le
\sup_{w_2\in\partial_r D}\frac{\lambda_D(z_2,w_2)}{\lambda_D(z_1,w_2)}.
\end{equation}
Relations (\ref{qxge1-eq3}) and (\ref{qxge1-eq4}) together give
the required conclusion.
\end{proof}

Recall that the symbol $\delta(z)$ stands for $\dist (z, \partial D)$,
the Euclidean distance from $z$ to the boundary $\partial D$ of $D$.
Also as in \cite{GO}, for $z_1,$ $z_2\in D$ the
metric $j_{D}(z_{1},z_{2})$ is defined by
$$j_{D}(z_{1},z_{2})=\log\left(1+\frac{|z_{1}-z_{2}|}
{\delta(z_{1})}\right)\left(1+\frac{|z_{1}-z_{2}|}{\delta(z_{2})}\right).
$$
The metric $j'_{D}$ is obtained by replacing the Euclidean distance
in the definition of $j_{D}$ metric by the $\lambda$-length (c.f.
\cite{Bro}).

The following Bernoulli inequalities are crucial to prove certain
inequalities in our context:
%(c.f. \cite{Se} and \cite[p. 390]{Ha2}).

\begin{lemma}\label{lem2.2}
For $x\ge 0$, we have
$$\log(1+cx)\leq c\; \log(1+x)\;\;\mbox{if}\;\; c\geq 1
$$
and
$$\log(1+cx)\geq c\; \log(1+x)\;\;\mbox{if}\; \; 0\leq c\leq 1.
$$
\end{lemma}

These two inequalities follow, for example,  from the fact that
$c\mapsto\log (1+cx)/c$ is a decreasing function of $c$ in $(0,\infty)$,
for each fixed $x>0$.

We denote by $xy$ the line through $x$ and $y$,  and by
$(x,y]$ the semi closed (or semi open) segment between $x$ and $y$.

% All the definitions of metrics,
% which are not defined before, %$h_{D}$, $j_{D}$, $\alpha_{D}$, $j'_D$, $\alpha'_D$,
% %$k_D$ and $\widetilde{a}_D$
% in the following results will be presented later.
%in Section \ref{chap3-sec2}.
%About the characteristic properties of quasidisks, it was proved
%by Gehring that %(Theorem $4.6$ in \cite{Ge})
We begin our discussion with the following simple characteristic
property of quasidisks due to Gehring.

\noindent{\bf Theorem A \cite[Theorem 11]{Ge99}}. {\it A simply
connected domain $D$ is a quasidisk if and only if there is a
constant $c$ such that
$$h_{D}(z_{1},z_{2})\leq c\;j_{D}(z_{1},z_{2})
$$
for all $z_{1},z_{2}\in D$.
}

Later in 2000, Gehring and Hag obtained the following interesting
characterization. %(Theorem $3.1$ in \cite{GH})

\noindent {\bf Theorem B \cite[Theorem 3.1]{GH}}.
{\it A simply connected domain $D$ is a
quasidisk if and only if there is a constant $c$ such that
$$h_{D}(z_{1},z_{2})\leq c\;\alpha_{D}(z_{1},z_{2})
$$
for all $z_{1},z_{2}\in D$.
}

% In the unit disk $\D :=\{z:\,|z|<1\}$ case, we actually have
% $h_{\D}(z_{1},z_{2})=a_{\D}(z_{1},z_{2})$, see \cite{Be}.
Next, we recall the following simple and useful characteristic properties
of John disks due to N\"akki and V\"ais\"al\"a \cite{NV}.
%(Theorem $4.5$ in \cite{NV})

\noindent{\bf Theorem C \cite{NV}}.
{\it  Let $D$ be a simply connected proper subdomain in $\C$. Then
the following conditions are equivalent:
\begin{enumerate}
\item[{\rm (1)}] $D$ is a $b$-John disk;
\item[{\rm (2)}] For each $z\in \IR^{2}$ and $r> 0$, any two points in
$D\setminus \overline{\D}(z,r)$ can be joined by an arc in
$D\setminus \overline{\D}(z,\frac{r}{c})$, where the
constants $b$ and $c$ depend only on each other and
$\D(z,r)$ denotes the open disk of radius $r$  centered at
the point $z$;
\item[{\rm (3)}] For every straight crosscut $\gamma$ of $D$ dividing
$D$ into subdomains $D_{1}$ and $D_{2}$, we have $\min
\limits_{j=1,2} \diam (D_{j})\leq c\; \diam (\gamma)$, where the
constants $b$ and $c$ depend only on each other and $\diam(\gamma)$
means the diameter of $\gamma$.
\end{enumerate}
}

\medskip
By using $h_D$ and $j'_D$, Kim and Langmeyer obtained the following
necessary and sufficient conditions for $b$-John disks.
%(Theorem $4.1$ in \cite{Kil})

\noindent{\bf Theorem D \cite[Theorem 4.1]{Kil}}.
{\it A simply connected domain $D$ is a $b$-John disk if and only if there
exists a constant $c\geq 1$ such that
$$h_D(z_1,z_2)\leq c\; j'_D(z_1,z_2)
$$
for all $z_{1},z_{2}\in D$.
Here the constants $b$ and $c$ depend only on each other. }

On the other hand, Broch in his Ph.D thesis characterized John disks in
terms of a bound for hyperbolic distance $h_D$ with an additive constant.
%(Theorem $6.2.9$ in \cite{Bro})

\noindent
{\bf Theorem E \cite[Theorem 6.2.9]{Bro}}. {\it A simply
connected (Jordan) domain $D$ is a $b$-John disk if and only if
there are constants $c$ and $d$ such that
$$h_D(z_1, z_2)\leq c \; \alpha'_D(z_1, z_2)+d
$$
for all pairs of $z_1, z_2\in D$,
where $c$ and $d$ depend only on $b$, and $b$ depends only on $c$ and $d$. }

In view of a comparison with Theorem B, Broch
raised the following.
%conjecture: %({\it Conjecture} $6.2.12$ in
%\cite{Bro}).

\noindent{\bf Conjecture F \cite[Conjecture 6.2.12]{Bro}}.
{\it A simply connected (Jordan) domain $D$ in $\C$ is a $b$-John disk if
and only if there is a constant $c$ such that
$$h_D(z_1, z_2)\leq c\; \alpha'_D(z_1, z_2)
$$
for all $z_1$, $z_2\in D$. Here the constants $b$ and $c$ depend
only on each other.}

Recently, in \cite{Wap}, Wang, Huang, Ponnusamy and Chu proved that the
sufficiency part in Conjecture~F is actually true.
%That is (Corollary $2.3$ in \cite{Wap})

\noindent{\bf Theorem G \cite[Corollary 2.3]{Wap}}.
{\it Suppose that $D$ is simply connected
and that there is a constant $c$ such that
$$h_D(z_1,z_2)\leq c\; \alpha'_D(z_1,z_2)
$$
for all $z_1$, $z_2\in D$. Then $D$ is a $b$-John disk with
$b=b(c)$.}

In addition to the above result, the authors in the same article \cite{Wap}
constructed two examples, one for a bounded John disk and the other
for an unbounded John disk, and showed that the necessity in
Conjecture~F fails to hold. In view of this development and the
importance of these domains in function theory, the following
question is natural.

\begin{problem}\label{question-1.1}
Is it true that $D$ is a uniform domain if and only if there is a
constant $c$ such that
$$j_D(z_1, z_2)\leq c\; \alpha'_D(z_1, z_2)
$$
for all $z_1$, $z_2\in D$.
%, where the constants $b$ and $c$ depend only on each other?
\end{problem}

%\medskip

We also collect a number of relevant results for completeness.

\noindent{\bf Theorem H}.
{\it Let $D$ be a simply connected domain. Then for
all $z_1$, $z_2\in D$ we have
$$\frac{1}{2}\;k_D(z_1,z_2)\leq h_D(z_1,z_2)\leq 2\;k_D(z_1,z_2),
%\eqno{(1.1)}
$$
$$k_{D}(z_{1},z_{2})\geq
\log\left(1+\frac{|z_{1}-z_{2}|}{
\dist (z_j, \partial D)}\right)\quad (j=1,2)
%\eqno{(1.2)}
$$
and \be\label{eq1.3} j_D(z_1,z_2)\leq 2 k_D(z_1,z_2).
%\eqno{(1.3)}
\ee
}

Note that the first two inequalities in Theorem H are due to Gehring
and Osgood \cite{GO} while the inequality (\ref{eq1.3}) may be obtained,
%by H\"{a}st\"{o} \cite{Ha2}.
for instance, from \cite[Lemma 2.1]{GP} and \cite[Exercise 2.40 ]{Vu}
(see also \cite{Vu2} and  \cite[Section 5]{Ha2}).
In fact, the inequality (\ref{eq1.3}) holds
for all proper subdomains $D$ of $\IR^n$. On the other hand,
concerning $K$-quasi-isotropic domains (see Chapter \ref{chap2}
for the definition), H\"ast\"o proved the following analogous result.
 %(Corollary $5.4$ in \cite{Ha2})

\noindent{\bf Theorem I \cite[Corollary 5.4]{Ha2}}.
{\it If a domain $D\psubset \IR^n$ is $K$-quasi-isotropic, then
$$k_D(z_1,z_2)/K\leq \widetilde{a}_D(z_1,z_2)\leq 2k_D(z_1,z_2)
$$
for all  $z_1$, $z_2\in D$, where the second inequality
always holds.}

\noindent{\bf Theorem J \cite[Corollary 5.4]{Ha5}}.
{\it For $D\psubset \IR^n$ the following are equivalent:
\begin{enumerate}
\item $D$ is quasi-isotropic;
\item $\widetilde{a}_D\approx k_D$; and
\item $j_D\lesssim \widetilde{a}_D$.
\end{enumerate}}

% Again, for the notations of relations between metrics in Theorem J,
% we refer to Section \ref{chap3-sec4}.
Here is a simple result which illustrates
the usefulness of our investigation and the proof of it is a consequence of
Theorems D, H and I.

\begin{theorem}\label{th1}
Let $D$ be a simply connected domain. If $D$ is a
$K$-quasi-isotropic domain and there exists a constant $c$ such that
$$\widetilde{a}_D(z_1, z_2)\leq c\,j'_D(z_1,z_2)
$$
for all $z_1$ and $z_2$ in $D$,
then $D$ is a $b$-John disk, where $b$ depends only on $c$
and $K$.
\end{theorem}

Theorem \ref{th1} stimulates us to discuss the relation between
$K$-quasi-isotropic domains and John domains. Then we ask the
following.

\begin{problem}\label{question-1.3}
Suppose that $D$ is a simply connected domain. Is it true that $D$ is a
$K$-quasi-isotropic domain if and only if $D$ is a $b$-John
disk, where constants $K$ and $b$ depend only on each other?
\end{problem}

The main aim of this chapter is to discuss Problems \ref{question-1.1}
and \ref{question-1.3} which will be presented in the following sections.

%Concerning Problem \ref{question-1.1}, our main results are as follows:

%%%%%%%%%%%%%%%%%%%%%%%%%%%%%%%%%%%%%%%%%%%%%%%%%%%
%%%%%%%%%%%%%%%%%%%%%%%%%%%%%%%%%%%%%%%%%%%%%%%%%%%
%%%%%%%%%%%%%%%%%%%% SECTION 2 %%%%%%%%%%%%%%%%%%%%
%%%%%%%%%%%%%%%%%%%%%%%%%%%%%%%%%%%%%%%%%%%%%%%%%%%
%%%%%%%%%%%%%%%%%%%%%%%%%%%%%%%%%%%%%%%%%%%%%%%%%%%

\section{Uniformity and $\lambda$-Apollonian Metric}\label{chap3-sec2}

In this section we present complete solution to Problem \ref{question-1.1}.
We show that the necessary part is true in simply connected domains.
More precisely, we give a number of examples in which the necessary as well as
the sufficiency parts are not true in general.

There exist a number of alternative characterizations of uniform domains.
However, it is a non-trivial task to verify whether a given domain is
uniform. In \cite[Examples 2.4(1)]{NV} N\"akki and V\"ais\"al\"a stated that
the exterior of a ball is a John domain. Although it seems from the definition that
the exterior of a ball is uniform, because of independent interest, we present a
proof below. %which is needed for a solution to Example \ref{thm1.6}.

\begin{lemma}\label{uni-Bbar}
The domain
$D=\mathbb{C}\backslash \overline{\D}$ is uniform.
\end{lemma}
\begin{proof}
Let $z_{1}$, $z_{2}\in D$, and recall the notation
$\delta(z)=\dist (z,\partial D)$. Without loss of
generality we assume that $\delta(z_{1})=\ds\min_{j=1,2}\delta(z_{j})$.

If $\delta(z_{1})\geq \frac{1}{4}$, then we pick
$z'_{2}\in [0,z_{2}]$ such that $\delta(z'_{2})=\delta(z_{1})$.
Therefore, $z_1$ and $z_2'$ divide the
circle $S(0,|z_1|)=\{z\in \mathbb{C}:\; |z|=|z_1|\}$ into two
parts: $\gamma_1$ and $\gamma_1'$ with $\ell(\gamma_1)\leq
\ell(\gamma_1')$. Define $\gamma=\gamma_{1}\cup [z'_{2},z_{2}]$. Then we
have
$$\ell(\gamma)\leq \frac{\pi}{2}|z_{1}-z'_{2}|+|z_{2}-z'_{2}|
\leq \frac{\pi+2}{2}|z_{1}-z_{2}|.
$$
Given a $z\in \gamma$, if $z\in \gamma_1$, then we have
$$\ell(\gamma_1[z_1,z])\leq \pi(\delta(z)+1)\leq 5\pi \delta(z).
$$
If $z\in [z_2',z_2]$, then
$$\ell(\gamma_1)+\ell([z_2',z])\leq 5\pi \delta(z_1)+\delta(z)
-\delta(z_2')\leq 5\pi \delta(z).
$$
Consequently, for each $z\in \gamma$, we have
$$\min_{j=1,2}\ell (\gamma [z_j, z])\leq 5\pi \delta(z).
$$
Now we assume that $\delta(z_{1})< \frac{1}{4}$. We
need to examine two cases.

\medskip
{\bf Case I:} Let $\delta(z_{2})\geq \frac{1}{2}$.
\smallskip

Consider the half line $L$ starting from the origin $O$ and passing through
$z_{1}$, and let $z'_{1}\in L$ with $\delta(z'_{1})=\delta(z_{2})$. Then $z_1'$ and
$z_2$ divide the circle $S(0, \delta(z_2))$ into two parts: $\gamma_2$ and
$\gamma_2'$ with $\ell(\gamma_2)\leq \ell(\gamma_2')$.

Now, we let $\gamma=\gamma_{2}\cup [z_{1}, z'_{1}]$. First notice that
\begin{eqnarray*}
\ell(\gamma) & \leq & |z_{1}-z'_{1}|+\ell(\gamma_{2})\\
& \leq & \delta(z_{2})+\pi(1+\delta(z_{2}))\\
& \leq & 2(3\pi+1)|z_{1}-z_{2}|,
\end{eqnarray*}
because $|z_{1}-z_{2}|\geq |z_{2}|-|z_{1}|\geq \delta(z_{2})-\frac{1}{4}
\geq \frac{1}{2}\delta(z_{2})$.
For $z\in \gamma$, if $z\in[z_{1},z'_{1}]$, we find that
$$\ell([z_{1},z])\leq \delta(z).
$$
If $z\in \gamma_{2}$, we see that
\begin{eqnarray*}
\ell(\gamma_{2}(z_{2},z)) & \leq & \ell(\gamma_2)\leq
\frac{\pi}{2}|z'_1-z_2|\\
& \leq & \frac{\pi}{2}(|z_{2}|+|z'_{1}|)= \pi(\delta(z_{2})+1)\\
& \leq  & 3\pi \delta(z)
\end{eqnarray*}
and so, for each $z\in\gamma$,
$$\min_{j=1,2}\ell (\gamma [z_j, z])\leq 3\pi \delta(z).
$$

\medskip
{\bf Case II:} Let $\delta(z_{2})< \frac{1}{2}$.

\smallskip
{\bf Subcase I:} First we consider the range $\frac{\pi}{18}\leq \angle z_{1}Oz_{2}\leq
\pi$.

\smallskip

Let $L_{1}$ be the half line starting from $O$ and passing through
$z_{1}$, and let $L_{2}$ be the half line starting from $O$ and
passing through $z_{2}$. Choose $z'_{1}\in L_{1}$ and $z'_{2}\in
L_{2}$ with $\delta(z'_{1})=\delta(z'_{2})=\frac{1}{2}$.
Then $z_1'$ and $z_2'$ divide
the circle $S(0,\frac{3}{2})$ into two parts: $\gamma_3$ and
$\gamma_3'$ with $\ell(\gamma_3)\leq \ell(\gamma_3')$. Let
$\gamma=[z_{1},z'_{1}]\cup\gamma_{3}\cup [z'_{2},z_{2}]$.
As in the previous case, this yields that
\begin{eqnarray*}
\ell(\gamma) & \leq &
\frac{1}{2}+\frac{1}{2}+\ell(\gamma_{3})\leq
1+\frac{3\pi}{2}\\
& \leq & 2\pi \frac{\sin\frac{17\pi}{36}}
{\sin\frac{\pi}{18}}|z_{2}-z_{1}|.
\end{eqnarray*}
Now, for $z\in \gamma$, if $z\in [z_j,z_j']$ ($j=1,2$), we then have
$$\ell([z_{j},z])\leq \delta(z).
$$
On the other hand, for $z\in \gamma_{3}$, we have
\begin{eqnarray*}
\ell(\gamma(z_{1},z)) & \leq &
|z'_{1}|-|z_{1}|+\ell(\gamma_{3})\leq \frac{1}{2}+\pi\left(1+\frac{1}{2}\right)\\
& \leq & (3\pi+1)\delta(z),
\end{eqnarray*}
since $\delta(z)=\frac{1}{2}$. The above observations imply that
$$\min_{j=1,2}\ell (\gamma [z, z_j])\leq (3\pi+1) \delta(z)
$$
for each $z\in\gamma$.

\smallskip
{\bf Subcase II:} Consider the case $\angle z_{1}Oz_{2}< \frac{\pi}{18}$.
\smallskip

\begin{figure}
\begin{center}
\includegraphics[width=8cm]{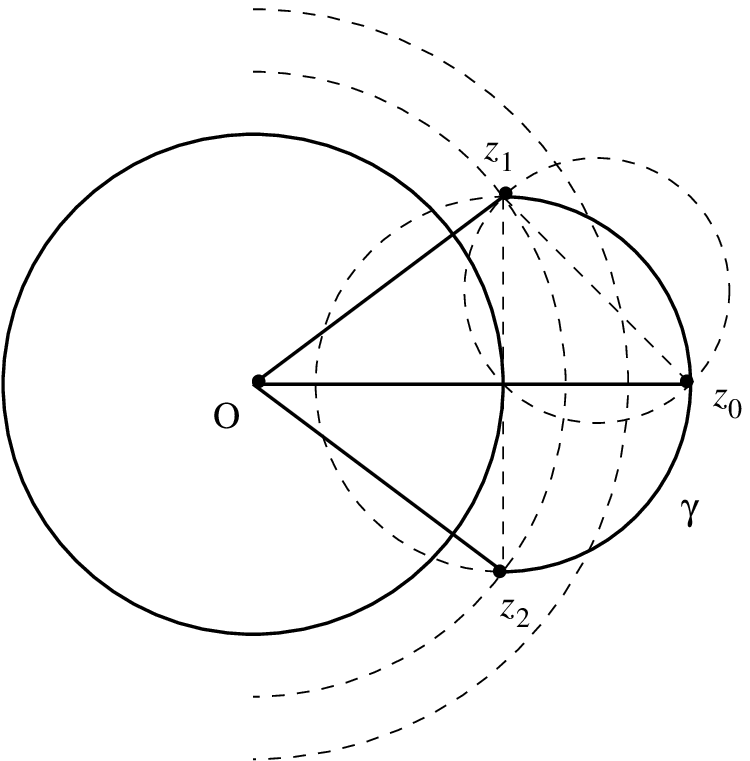}
\end{center}
\caption{The domain $\protect D=\C\backslash \overline{\D}.$}\label{w_fig3}
\end{figure}

Let $\gamma$ be the half circle of
$S\left(\frac{z_{1}+z_{2}}{2},\frac{|z_{1}-z_{2}|}{2}\right)$
divided by $z_1$ and $z_2$, which satisfies the condition $\angle
Oz_1z_0>\frac{\pi}{2}$, where $z_0\in \gamma$ with
$\ell(\gamma[z_{1},z_{0}])=\ell(\gamma[z_{2},z_{0}])$, see Figure \ref{w_fig3}.

Clearly,
$$\ell(\gamma)\leq \frac{\pi}{2}|z_{1}-z_{2}|.
$$
Now, we claim that
$$\min_{j=1,2}\ell (\gamma [z_j, z])\leq \pi \delta(z)
$$
for $z\in\gamma$. To establish the claim, we first observe that for any $z\in
\gamma[z_{1},z_{0}]$, $\angle zz_{1}O\geq \frac{13\pi}{18}$. Hence
for such $z$, we obtain that
\begin{eqnarray*}
|z| & = &
\sqrt{|z_{1}|^{2}+|z_{1}-z|^{2}-2|z_{1}||z_{1}-z|
\cos\angle Oz_{1}z}\\
& \geq & \sqrt{1+|z_{1}-z|^{2}+|z_{1}-z|}.
\end{eqnarray*}
Since
$$\delta(z)=|z|-1\geq \sqrt{1+|z_{1}-z|^{2}+|z_{1}-z|}-1
$$
and
$$|z_{1}-z|\leq 2\left(\sqrt{1+|z_{1}-z|^{2}+|z_{1}-z|}-1\right) \leq 2\delta(z),
$$
we deduce that
%\be\label{uni-Bbar-eq1}
$$\ell(\gamma[z_{1},z])\leq \frac{\pi}{2}|z_{1}-z|\leq \pi \delta(z).
$$
%\eqno (3.1)
%\ee
On the other hand, if $z\in \gamma [z_0,z_2]$,  we can find $z'\in \gamma[z_1,z_0]$
with $\ell(\gamma[z_1,z'])=\ell(\gamma[z,z_2])$ which shows that
%\be\label{uni-Bbar-eq2}
$$\ell(\gamma[z,z_2])= \ell(\gamma[z_1,z'])\leq \pi \delta(z')\leq \pi \delta(z).
$$
%\eqno{(3.2)}
%\ee
%(\ref{uni-Bbar-eq1}) and (\ref{uni-Bbar-eq2})
The last two inequalities complete the proof of our claim.
%the proof of Lemma \ref{uni-Bbar} is completed.
\end{proof}
%\medskip

We also need the following result from \cite{MS} which says that
quasiconformal images of uniform domains are uniform.

\begin{lemma}\label{qc-uni}
Let $f:\, \IR^{n}\to \IR^{n}$ be a quasiconformal
mapping and $D\subset \IR^{n}$ be a uniform domain.
Then $f(D)$ is a uniform domain.
\end{lemma}

\begin{theorem}\label{thm1.4}
Suppose $D$ is a simply connected domain. If $D$ is uniform, then
there exists a constant $c$ such that
$$j_{D}(z_{1},z_{2})\leq c\;\alpha'_{D}(z_{1},z_{2})
$$
for all $z_1,$ $z_2\in D$.
\end{theorem}

\begin{proof}
Recall that a simply connected subdomain of the plane is uniform if
and only if it is a quasidisk, and in general a uniform domain is a
quasicircle domain, cf. \cite{GeHa2}. Since $D$ is given to be
uniform, it is a quasidisk. Further, Theorems B, H and Lemma
\ref{lem2.1} imply that
$$j_{D}(z_{1},z_{2})\leq 4 h_{D}(z_{1},z_{2})\leq
4c'\;\alpha_{D}(z_{1},z_{2})\leq c\;\alpha'_{D}(z_{1},z_{2})
$$
for all $z_1,$ $z_2\in D$, where $c=4c'$ is a constant.
\end{proof}

It is natural to ask whether Theorem \ref{thm1.4} holds for multiply
connected domains. The following example shows that the answer is
negative.

\begin{example}\label{thm1.5}
Let $D=\mathbb{C}\backslash D_1$ where $D_{1}=\{x+iy:\, |x|\leq
\frac{1}{2},|y|\leq \frac{1}{2}\}$. Then $D$ is a uniform domain,
but there does not exist any constant $c$ such that
$$j_{D}(z_{1},z_{2})\leq c\;\alpha'_{D}(z_{1},z_{2})
$$
for  all $z_1,$ $z_2\in D$.
\end{example}

\noindent{\em Solution.}
At first, we prove that $D=\C\backslash D_1$ is uniform, where
$D_{1}=\{x+iy:\, |x|\leq \frac{1}{2},|y|\leq \frac{1}{2}\}$. Since
$D_{1}$ is a quasidisk, there exists a $K$-quasiconformal mapping
$f:\, \C\to \C$ such that $D_{1}=f(\D)$. By Lemmas \ref{uni-Bbar}
and \ref{qc-uni}, $D=\C\backslash D_1$  is a uniform domain.
%\medskip

Next we prove that there does not exist any constant $c$ such that
$$j_{D}(z_{1},z_{2})\leq c\;\alpha'_{D}(z_{1},z_{2})
$$
for  all $z_1,$ $z_2\in D$.

%Let $L_{1}, L_{2},L_{3}$, and $L_{4}$ denote the boundaries of $\partial D$
%defined by
Now $\partial D=L_1\cup L_2\cup L_3 \cup L_4$, where
\begin{eqnarray*}
L_{1}&=&\{x+iy:\, x=-\frac{1}{2},\, |y|\leq \frac{1}{2}\},\\
L_{2}&=&\{x+iy:\, y=-\frac{1}{2},\, |x|\leq \frac{1}{2}\},\\
L_{3}& = &\{x+iy:\, x=\frac{1}{2},\,|y|\leq \frac{1}{2}\},\\
L_{4}& =& \{x+iy:\, y=\frac{1}{2},\, |x|\leq \frac{1}{2}\},
\end{eqnarray*}
%$L_{1}=\{x+iy:\, x=-\frac{1}{2},|y|\leq
%\frac{1}{2}\}$, $L_{2}=\{x+iy:\, y=-\frac{1}{2},|x|\leq
%\frac{1}{2}\}$, $L_{3}=\{x+iy:\, x=\frac{1}{2},|y|\leq \frac{1}{2}\}$
%and $L_{4}=\{x+iy:\, y=\frac{1}{2},|x|\leq \frac{1}{2}\}$,
see Figure \ref{thm1.6-fig1}. Then $\partial D=\cup^{4}_{j=1}L_{j}$.
For $x\le -\frac{3}{2}$, consider the two points
$z_{1}=x+\frac{1}{2}i$ and $z_{2}=x-\frac{1}{2}i$ in $D$. Clearly
$\delta(z_2)\ge 1$.

\begin{figure}
\begin{center}
\includegraphics[width=11cm]
{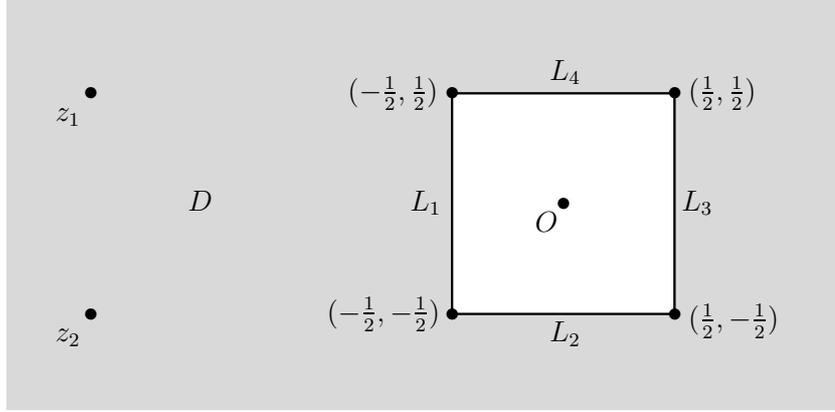}
\end{center}
\caption{$D$, the exterior of the region bounded by
$L_1, L_2, L_3, L_4$.}\label{thm1.6-fig1}
\end{figure}

We will prove the inequalities
\be\label{thm1.6-eq1}
\frac{\lambda_{D}(z_{1},a)}{\lambda_{D}(z_{2},a)}\leq
\sqrt{1+\frac{1}{\delta(z_{2})^{2}}}
\ee
and
\be\label{thm1.6-eq2}
\frac{\lambda_{D}(z_{2},b)}{\lambda_{D}(z_{1},b)}\leq
\sqrt{1+\frac{1}{\delta(z_{1})^{2}}}
\ee
for all $a,b\in \partial_{r} D$. Obviously it suffices to prove the inequality
(\ref{thm1.6-eq1}), as the proof of the inequality (\ref{thm1.6-eq2})
follows by symmetry.
%is essentially the same.
%\noindent
%{\bf Case I:} The case $a\in L_{1}$.\\

We begin by observing that for the case $a\in L_{1}$, we have
$$\frac{\lambda_{D}(z_{1},a)}{\lambda_{D}(z_{2},a)}\leq
\frac{\sqrt{|z_{1}-z_{2}|^{2}+\delta(z_{2})^{2}}}
{\delta(z_{2})}=\sqrt{1+\frac{1}{\delta(z_{2})^{2}}}
$$
since
$$\lambda_{D}(z_{1},a)\leq |z_1+\frac{1}{2}+\frac{i}{2}|=\sqrt{|z_{1}-z_{2}|^{2}+\delta(z_{2})^{2}}
~\mbox{ and } ~\lambda_{D}(z_{2},a)\geq \delta(z_{2}).
$$
%\noindent
%{\bf Case II:} The case $a\in L_{2}$.\\

Secondly,  for the case $a\in L_{2}$, we find that
$$\frac{\lambda_D(z_1,a)}{\lambda_D(z_2,a)}=
\frac{\sqrt{1+\delta(z_2)^2}+|a+\frac{1}{2}+\frac{i}{2}|}
{\delta(z_2)+|a+\frac{1}{2}+\frac{i}{2}|}
\le \sqrt{1+\frac{1}{\delta(z_{2})^{2}}}.
$$
%\noindent
%{\bf Case III:} $a\in L_{4}$. \\

Thirdly, for the case $a\in L_{4}$, we easily obtain that
$$\frac{\lambda_D(z_1,a)}{\lambda_D(z_2,a)}=
\frac{\delta(z_1)+|a+\frac{1}{2}-\frac{i}{2}|}
{\sqrt{1+\delta(z_2)^2}+|a+\frac{1}{2}-\frac{i}{2}|} \le
\sqrt{1+\frac{1}{\delta(z_{2})^{2}}}.
$$

%\begin{eqnarray*}
%\frac{\lambda_D(z_1,a)}{\lambda_D(z_2,a)}& \leq &
%\frac{1+\sqrt{|z_{1}-z_{2}|^{2}+\delta(z_2)^2}
%}{1+\delta(z_2)}\\
%& = & \sqrt{1+\frac{1}{\delta(z_2)^2}}.
%\end{eqnarray*}
%\noindent
%{\bf Case IV:} $a\in L_{3}$.\\
Finally, for the last case $a\in L_{3}$, we see that there exists a
point $p=\frac{1}{2}+is\in L_3$ with $-\frac{1}{2}<s<0$, such that
$$1+\delta(z_2)+\left |p-\frac{1}{2}-\frac{i}{2}\right |=
1+\sqrt{1+\delta(z_2)^2}+\left |p-\frac{1}{2}+\frac{i}{2}\right |.
$$
Indeed, for any such point we have $s=(\delta(z_2)-\sqrt{1+\delta(z_2)^2})/2.$
Similarly, one can see that there exists a point
$q=\frac{1}{2}+it$ with $0<t<\frac{1}{2}$, such that
$$1+\delta(z_2)+\left |q-\frac{1}{2}+\frac{i}{2}\right |=
1+\sqrt{1+\delta(z_2)^2}+\left |q-\frac{1}{2}-\frac{i}{2}\right |.
$$

Now, if the point $a$ lies below the point $p$
%then we have
%$$\lambda_D(z_1,a)=1+\sqrt{1+\delta(z_2)^2}+
%|a-\frac{1}{2}+\frac{i}{2}|
%$$
%and
%$$\lambda_D(z_2,a)=
%1+\delta(z_2)+|a-\frac{1}{2}+\frac{i}{2}|.
%$$
then by the same argument as in the second case, we obtain
(\ref{thm1.6-eq1}). If the point $a$ lies above the point $q$, then
arguing as in the third case verifies the inequality
(\ref{thm1.6-eq1}). If the point $a$ lies between the points $p$ and
$q$, then it follows that there exists a point $z=\frac{1}{2}+ir$
$\left (-\frac{1}{2}\le r<s\right )$ such that \be\label{thm1.6-eq3}
\frac{1+\delta(z_2)+|z-\frac{1}{2}-\frac{i}{2}|}
{1+\delta(z_2)+|z-\frac{1}{2}+\frac{i}{2}|} =\sqrt{1+\frac{1}{\delta(z_2)^2}}.
\ee In fact, it is easy to see that
$$r=\frac{(\delta(z_2)-\sqrt{1+\delta(z_2)^2}) (\delta(z_2)+\frac{3}{2})}
{\delta(z_2)+\sqrt{1+\delta(z_2)^2}}\ge -\frac{1}{2},
$$
since $\delta(z_2)\ge 1$. Consequently,
\begin{eqnarray*}
\frac{\lambda_D(z_1,a)}{\lambda_D(z_2,a)} & = & \frac{1+\delta(z_2)-{\rm
Im}\,(a)+\frac{1}{2}}
{1+\delta(z_2)+{\rm Im}\,(a)+\frac{1}{2}}\\
& \le & \frac{1+\delta(z_2)-s+\frac{1}{2}}
{1+\delta(z_2)+s+\frac{1}{2}}\\
& \leq & \frac{1+\delta(z_2)-r+\frac{1}{2}}
{1+\delta(z_2)+r+\frac{1}{2}}\\
& = & \sqrt{1+\frac{1}{\delta(z_2)^2}},
\end{eqnarray*}
where the last equality occurs by (\ref{thm1.6-eq3}).
The proof of (\ref{thm1.6-eq1}) is completed.
Combining (\ref{thm1.6-eq1}) and (\ref{thm1.6-eq2}) gives that
\begin{eqnarray*}
\alpha'_D(z_1,z_2) & = & \sup_{w_1,w_2\in \partial_r
D}\log\frac{\lambda_D(z_1,w_1)\lambda_D(z_2, w_2)}
{\lambda_D(z_1, w_2)\lambda_D(z_2, w_1)}\\
& \leq & \log\left(1+\frac{1}{\delta(z_{1})^{2}}\right),
\end{eqnarray*}
since $\delta(z_{1})=\delta(z_{2})$.

Finally, suppose on the contrary that there exists a constant $c$ such that
$$j_D(z_1,z_2)\le c\,\alpha'_D(z_1,z_2)
$$
for all $z_1,z_2\in D$. Denote $y=\frac{1}{\delta(z_1)}$.
Then, on one hand, we have
$$2\log (1+y)\leq c\,\log (1+y^2).
$$
But, on the other hand, we see that
$$\lim\limits_{y\rightarrow 0} \frac{\log(1+y)^2}{\log(1+y^{2})}=\infty.
$$
This contradiction completes the solution.  \hfill $\Box$

As a consequence of the following example, we conclude that the
converse part of Theorem \ref{thm1.4} is not true in general.

\begin{example}\label{thm1.6}
Let $D=\{x+iy:\,x>0,|y|< 1\}$. Then there exists a constant $c$ such
that
$$j_{D}(z_{1},z_{2})\leq c\;\alpha'_{D}(z_{1},z_{2})
$$
for all $z_1,$ $z_2\in D$, but $D$ is not a uniform domain.
\end{example}
\noindent{\em Solution.}
Theorem C implies that $D$ is not a John disk and hence it is not
uniform. Since $D$ is convex, we have $j_D\le 2 \alpha_D$, by
\cite[Theorem 4.2]{Se}. Thus by Lemma \ref{lem2.1} we have
$j_D(z_1,z_2)\le 2 \alpha'_D(z_1,z_2)$ for all $z_1,z_2\in D$.
%Here we note that $c=2$.
\hfill $\Box$

In the following, we construct an example and show that the
converse part of Theorem \ref{thm1.4} does not hold in the case of
$D$ being multiply connected domains.

\begin{example}\label{thm1.7}
Let $r=\frac{2\tan (\pi/36)}{1+2\tan (\pi/36)}$ and
$D=D_{1}\backslash D_{2}$, where $D_{1}=\{x+iy:\,x> 0,|y|<1\}$ and
$D_{2}=\{x+iy:\, r\leq x\leq 2-r,|y|\leq
1-r\}$. Then
there exists a constant $c$ such that
$$j_{D}(z_{1},z_{2})\leq c\;\alpha'_{D}(z_{1},z_{2})
$$
for all $z_1,$ $z_2\in D$, but $D$ is not a uniform domain.
\end{example}

\noindent{\em Solution.}
For our solution, we need the following lemma whose proof is easy to
obtain by using basic trigonometry and so we omit the details.

\begin{lemma}\label{angle}
Let $D$ be the same as that in Example {\rm \ref{thm1.7}} and $x_{1}=r+(1-r)i$,
$y_{1}=r+i$, $x_{2}=2-r+(1-r)i$. Then $x_{1}$, $y_{1}$, $x_{2}\in
\partial D$ and  $\angle
x_{1}x_{2}y_{1}=\frac{\pi}{36}$.
\end{lemma}

%Now we are ready to solve Example \ref{thm1.7}.

Now, let $z_1,$ $z_2$ be any two points in $D$. Without loss of
generality we may assume that $\delta(z_{1})=\ds\min_{j=1,2}\delta(z_{j})$,
where $\delta(z_j) = \dist (z_j, \partial D)$.
Let $z\in \partial D$ be such that $\delta(z_1)=|z_{1}-z|$.

If $|z_{1}-z_{2}|\geq 3\;\delta(z_{1})$, then
\begin{eqnarray*}
\alpha_{D}(z_{1},z_{2}) & \geq &
\log\frac{|z-z_{2}|}{|z-z_{1}|}\\
& \geq & \log\left(1+\frac{|z_{1}-z_{2}|}{3\delta(z_{1})}\right).
\end{eqnarray*}

By Lemmas \ref{lem2.1} and \ref{lem2.2}, we have
\begin{eqnarray}\label{thm1.7-eq1}
\nonumber j_{D}(z_{1},z_{2}) & \leq &
2\,\log\left(1+\frac{|z_{1}-z_{2}|}{\delta(z_{1})}\right)\\
\nonumber & \leq & 6\,\log \left(1+\frac{|z_1-z_2|}{3\delta(z_1)}\right)\\
\nonumber &\leq & 6\, \alpha_{D}(z_{1},z_{2})\\
&\leq & 6\, \alpha'_{D}(z_{1},z_{2}).
\end{eqnarray}

If $|z_{1}-z_{2}|< 3\,\delta(z_{1})$ and $z\in
\partial D_{1}$, then, by Theorem 4.2 in \cite{Se} and Lemma \ref{lem2.1},
it follows that
\be\label{thm1.7-eq2}
j_{D}(z_{1},z_{2})\leq 2\log\left(1+\frac{|z_{1}-z_{2}|}{\delta(z_{1})}\right)
\leq 2a'_{D_{1}}(z_{1},z_{2})\leq 2\alpha'_{D}(z_{1},z_{2}).
\ee
In the following, we always assume that $|z_{1}-z_{2}|< 3\,\delta(z_{1})$ and
$z\in \partial D_{2}$.

Following the notation of Lemma \ref{angle}, let $x_1=r+(1-r)i$,
$x_{2}=2-r+(1-r)i$, $x_{3}=2-r+(r-1)i$ and
$x_{4}=r+(r-1)i$, see Figure \ref{w_fig2}.
\begin{figure}
\begin{center}
\includegraphics[width=11cm]
{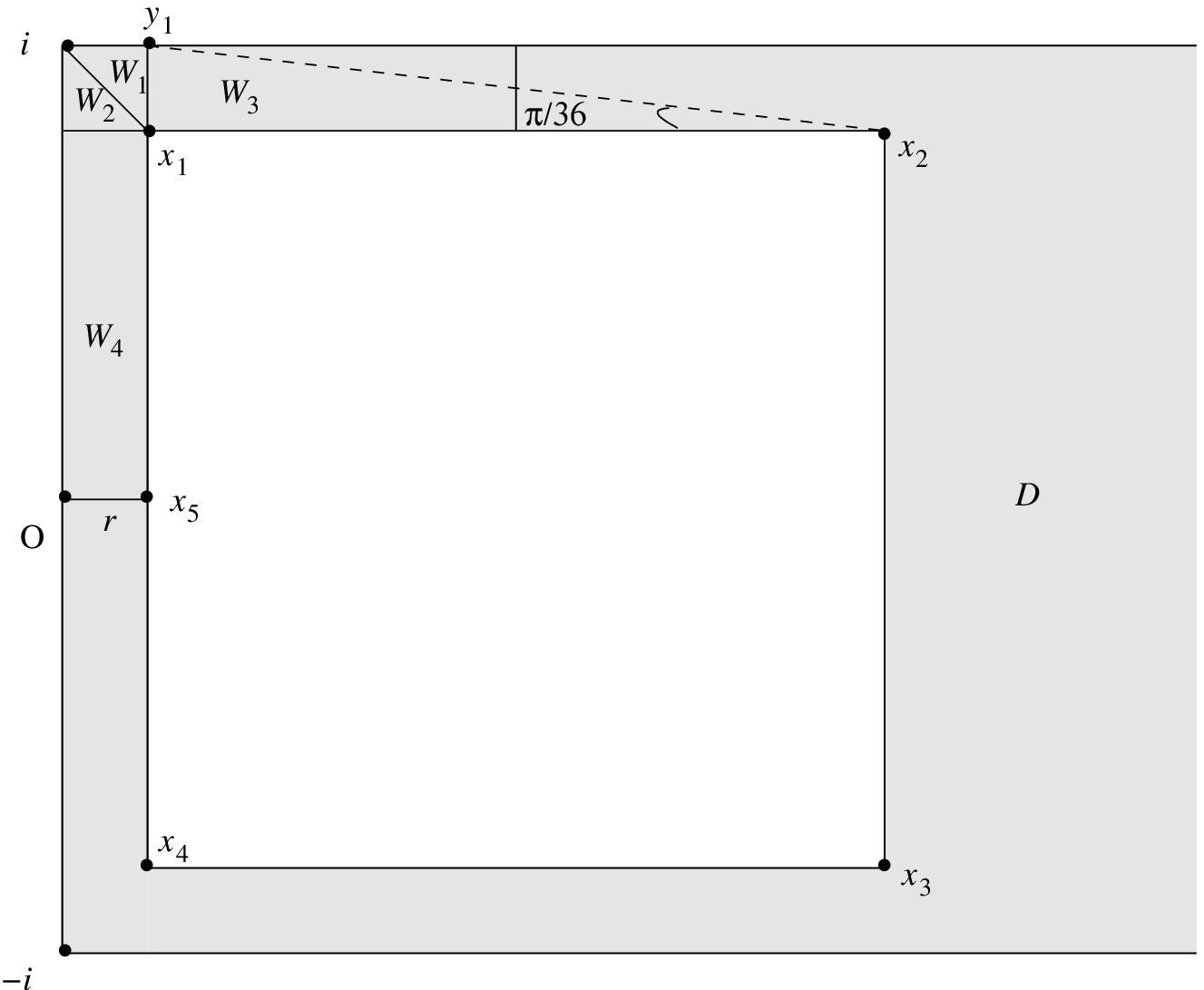}
\end{center}
\caption{$D$, the domain exterior of the rectangle (with vertices $x_1,x_2,x_3,x_4$) and
contained in the half strip.}\label{w_fig2}
\end{figure}

{\bf Case I:} First we consider the case $z=x_{1}$.

\smallskip

Without loss of generality, we may assume $z_{1}\in W_{1}$, where
$W_1$ denotes the closure of the triangular domain with vertices
$r+(1-r)i$, $r+i$ and $i$. Then
$$z_2\in W_1\cup W_2\cup W_3\cup W_4,
$$
where $W_2$ denotes the closure of the triangular domain with the
vertices $r+(1-r)i$, $i$ and $(1-r)i$; $W_3$ the closure of the
rectangular domain with the vertices $r+(1-r)i$, $1+(1-r)i$, $1+i$
and $r+i$; and $W_4$ the closure of the rectangular domain with
the vertices $0$, $r$, $r+(1-r)i$ and $(1-r)i$, see Figure
\ref{w_fig2}.
We divide our discussions into two subcases.

\smallskip

{\bf Subcase I:} The subcase $z_{2}\in W_{1}$.
\smallskip

In the following, when we mention an angle we always mean that one
which is not greater than $\pi$.
\smallskip

Obviously,
$\angle z_{2}z_{1}z\geq \frac{3\pi}{8}$. Next we obtain that, if
$\angle z_{2}z_{1}z\geq \frac{19 \pi}{36}$, then
$$\frac{|z-z_{2}|}{|z-z_{1}|}\geq\frac{\sqrt{|z-z_{1}|^{2}+
|z_{2}-z_{1}|^{2}+2 c_0|z-z_{1}|\, |z_{2}-z_{1}|}}{|z-z_{1}|}
\geq\sqrt{1+\frac{2c_0|z_{2}-z_{1}|}{\delta(z_{1})}},
$$
where $c_0=\sin\frac{\pi}{36}$, which yields that
$$\alpha_{D}(z_{1},z_{2})\geq \log \frac{|z-z_{2}|}{|z-z_{1}|}\geq\log
\sqrt{1+\frac{2c_0|z_{2}-z_{1}|}{\delta(z_{1})}}.
$$
Appealing to Lemmas \ref{lem2.1} and \ref{lem2.2}, we obtain
\be\label{thm1.7-eq3}
j_{D}(z_{1},z_{2})\leq 2\log\left(1+\frac{|z_{1}-z_{2}|}{\delta(z_{1})}\right)
\leq \frac{2}{c_0}\alpha'_{D}(z_{1},z_{2}).
\ee

If  $\angle z_{2}z_{1}z\leq\frac{19 \pi}{36}$, then, by Lemma $3.1$,
there must exist a point $a\in [x_{1},x_{2}]$ or $[x_1,x_5]$ such
that
$$\angle z_{2}z_{1}a= \frac{19 \pi}{36}.
$$
Elementary computations show that there exists a constant $c_1$
such that
$$\frac{|a-z_{2}|}{|a-z_{1}|}
\geq \sqrt{1+\frac{2c_0|z_{1}-z_{2}|}{|a-z_{1}|}}
$$
and
$$\alpha_{D}(z_{1},z_{2})\geq
\log \sqrt{1+\frac{c_1|z_{1}-z_{2}|}{\delta(z_{1})}},
$$
where we can take
$$c_1=\min \left\{c_0,
\frac{\sin\frac{\pi}{36}\sin\frac{\pi}{12}} {1+2\sin\frac{\pi}{12}},
\frac{2\sin\frac{\pi}{36}\sin\frac{\pi}{18}}{2\sin\frac{\pi}{36}+
\sin\frac{\pi}{18}}\right\} =\frac{\sin\frac{\pi}{36}\sin\frac{\pi}{12}}
{1+2\sin\frac{\pi}{12}}.
$$
 Therefore, \be\label{thm1.7-eq4}
j_{D}(z_{1},z_{2})\leq
2\log\left(1+\frac{|z_{1}-z_{2}|}{\delta(z_{1})}\right) \leq
\frac{4}{c_1}\alpha'_{D}(z_{1},z_{2}). \ee

{\bf Subcase II:} The subcase $z_2\in W_2\cup W_3\cup W_4$.
\smallskip

Arguing as in {\it Subcase I}, we see that there exists
a constant $c_2$ such that
\be\label{thm1.7-eq5}
j_{D}(z_{1},z_{2})\leq \frac{4}{c_2}\alpha'_{D}(z_{1},z_{2}),
\ee
where we can take
\begin{eqnarray*}
c_2 \!\! & \!= \!&\!\!  \min\left\{c_0,
\frac{2\sin\frac{\pi}{36}\sin\frac{\pi}{18}}
{\sin\frac{\pi}{18}+\sin\frac{7\pi}{36}},
\frac{2\sin\frac{\pi}{36}\sin\frac{\pi}{18}}{\sin\frac{\pi}{18}+
\sin\frac{5\pi}{18}},
\frac{\sin\frac{\pi}{36}\sin\frac{7\pi}{36}}{2(\sin\frac{7\pi}{36}+
\sin\frac{11\pi}{36})}, \frac{\sin\frac{\pi}{36}
\sin\frac{5\pi}{36}} {2(\sin\frac{5\pi}{36}+
\sin\frac{13\pi}{36})}\right\}\\
\!\! & \!= \!&\!\!  \frac{\sin\frac{\pi}{36}
\sin\frac{5\pi}{36}} {2(\sin\frac{5\pi}{36}+
\sin\frac{13\pi}{36})}.
\end{eqnarray*}

\smallskip
{\bf Case II:} The case $z\in (x_{1},x_{4}]$.
\smallskip

Let $x_{5}=r$. Without loss of generality, we may assume that
$z\in (x_{1},x_{5}]$. Clearly $z_1\in W_4$. Then we see that
$$z_2\in U_1\cup W_3\cup U_2,
$$
where $U_1$ denotes the closure of the rectangular domain with
the vertices ${\rm Im}\,(z_1)i$, $r+{\rm Im}\,(z_1)i$, $r+i$ and $i$; and $U_2$
the closure of the rectangular domain with the vertices
$(r-1)i$, $r+(r-1)i$, $r+{\rm Im}\,(z_1)i$ and ${\rm Im}\,(z_1)i$.

\smallskip

{\bf Subcase III:} The subcase $z_{2}\in U_1$.
\smallskip

First we observe that $\angle z_{2}z_{1}z>\frac{\pi}{4}$.
If $\angle z_{2}z_{1}z\geq\frac{19\pi}{36}$, then
\be\label{thm1.7-eq6}
j_{D}(z_{1},z_{2})\leq \frac{2}{c_0}\alpha'_{D}(z_{1},z_{2}).
\ee
On the other hand, if $\angle z_{2}z_{1}z<\frac{19\pi}{36}$, then there must exist a
point $a\in [x_{1},x_{4}]$ such that
$$\angle z_{2}z_{1}a=\frac{19 \pi}{36}.
$$
Obviously $\angle zz_{1}a\leq \frac{5\pi}{18}$ and $\angle
z_{1}az\geq \frac{2\pi}{9}$. Hence,
$$\frac{|a-z_{2}|}{|a-z_{1}|}
\geq \sqrt{1+\frac{2c_0|z_{1}-z_{2}|}{|a-z_{1}|}}
$$
and
$$\alpha_{D}(z_{1},z_{2})
\geq \log \sqrt{1+\frac{c_{3}|z_{1}-z_{2}|}{\delta(z_{1})}},
$$
where $c_{3}$ can be taken as
$$c_3=\frac{2\sin\frac{\pi}{36}\sin\frac{2\pi}{9}}
{\sin\frac{5\pi}{18}+\sin\frac{2\pi}{9}}.
$$
Consequently,
\be\label{thm1.7-eq7}
j_{D}(z_{1},z_{2})\leq \frac{4}{c_{3}}\alpha'_{D}(z_{1},z_{2}).
\ee

\smallskip

{\bf Subcase IV:} The subcase $z_2\in W_{3}$.
\smallskip

Our choice of point ensures that $\angle z_{2}z_{1}z\geq\frac{\pi}{4}$ or $\angle
z_{1}z_{2}b\geq\frac{\pi}{4}$, where
$$b={\rm Re}\,(z_2)+(1-r)i.
$$
Then there must exist a point $a\in [x_{1},x_{4}]$ or $a\in
[x_{1},x_{2}]$ such that
$$\angle z_{2}z_{1}a= \frac{19 \pi}{36}\;\; \mbox{or}\;\;
\angle z_{1}z_{2}a=\frac{19 \pi}{36}.
$$
We have known that there exists a constant $c_4$ such that
\be\label{thm1.7-eq8}
j_{D}(z_{1},z_{2})\leq \frac{4}{c_{4}}\alpha'_{D}(z_{1},z_{2}),
\ee
where we can take
$$c_4=\frac{\sin\frac{\pi}{36}\sin\frac{2\pi}{9}}{2(\sin\frac{2\pi}{9}
+\sin\frac{5\pi}{18})}.
$$

\smallskip

{\bf Subcase V:} The subcase $z_2\in U_{2}$.
\smallskip

If $\angle z_{2}z_{1}z\geq \frac{19 \pi}{36}$, then
$$\frac{|z-z_{2}|}{|z-z_{1}|}
\geq \sqrt{1+\frac{2c_0|z_{1}-z_{2}|}{\delta(z_{1})}}
$$
and
\be\label{thm1.7-eq9}
j_{D}(z_{1},z_{2})\leq\frac{2}{c_0}\alpha'_{D}(z_{1},z_{2}).
\ee
If $\angle z_{2}z_{1}z\leq\frac{19 \pi}{36}$, then $\frac{17
\pi}{36}\leq\angle z_{1}z_{2}z\leq \frac{\pi}{2}$ and there must
exist a point $a\in [x_{1},x_{4}]$ such that
$$\angle z_{1}z_{2}a= \frac{19 \pi}{36}.
$$
Therefore, there exists a constant $c_5$ such that
\be\label{thm1.7-eq10} j_{D}(z_{1},z_{2})\leq
\frac{4}{c_{5}}\alpha'_{D}(z_{1},z_{2}), \ee where we can take
$$c_5=\min\left\{c_0,\;
\frac{\sin\frac{\pi}{36}\sin\frac{17\pi}{36}}{2(\sin\frac{17\pi}{36}
+\sin\frac{\pi}{36})}\right\}=\frac{\sin\frac{\pi}{36}\sin\frac{17\pi}{36}}{2(\sin\frac{17\pi}{36}
+\sin\frac{\pi}{36})}.
$$
\smallskip

By the symmetry of $D$, there is only one possibility about the
place of $z$ which needs to be discussed, which is:

\smallskip

{\bf Case III:} The case $z\in [x_2,x_3]$.
\smallskip

By the above discussions, we may assume that $\delta(z_1)\geq r$.
Then $z_1$, $z_2\in W_5$, where $W_5$ denotes
the closure of the rectangular domain with the vertices
$2+(r-1)i$, $3-r+(r-1)i$, $3-r+(1-r)i$ and $2+(1-r)i$.

If $\angle z_{2}z_{1}z\geq
 \frac{19 \pi}{36}$, then
\be\label{thm1.7-eq11}
j_{D}(z_{1},z_{2})\leq \frac{2}{c_0}\alpha'_{D}(z_{1},z_{2}).
\ee
If $\angle z_{2}z_{1}z<\frac{19\pi}{36}$, then there exists a
point $a\in [x_2, x_3]$ such that
$$\angle z_{2}z_{1}a=\frac{19 \pi}{36}\;\;\mbox{or}\;\; \angle
z_{1}z_{2}a=\frac{19 \pi}{36}.
$$
We have known that there is a constant $c_6$ such that
\be\label{thm1.7-eq12} j_{D}(z_{1},z_{2})\leq
\frac{4}{c_{6}}\alpha'_{D}(z_{1},z_{2}), \ee where we can take
$$c_{6}=\min\left\{\frac{2\sin\frac{\pi}{36}\sin\frac{17\pi}{36}}
{\sin\frac{17\pi}{36}
+\sin\frac{\pi}{36}},\;\frac{2\sin\frac{\pi}{36}
\sin\frac{4\pi}{9}}{\sin\frac{\pi}{18}+\sin\frac{4\pi}{9}}\right\}
=\frac{2\sin\frac{\pi}{36}
\sin\frac{4\pi}{9}}{\sin\frac{\pi}{18}+\sin\frac{4\pi}{9}}.
$$

Finally, we let
$$c=\max \left\{6,\;\frac{4}{c_{i}}:\, i=0,\ldots, 6\right\}.
$$
Then $c>0$ and equations (\ref{thm1.7-eq1}) -- (\ref{thm1.7-eq12})
show that
$$j_{D}(z_{1},z_{2})\leq c\;\alpha'_{D}(z_{1},z_{2})
$$
for all $z_{1}$,$z_{2}\in D$.
At last, the proof of $D$ being not a uniform domain easily follows from
the definition.
%We completed the solution of Example \ref{thm1.7}.
 \hfill $\Box$

\section{John Disks and Quasi-isotropic Domains}\label{chap3-sec4}

In this section, we give answer to Problem \ref{question-1.3}. We see
that neither John disks are quasi-isotropic nor conversely.

From the definition of inner metric we have seen that $d_1 \approx d_2$ implies
$\tilde{d_1} \approx \tilde{d_2}$. Hence we have the following result which in view
of Theorem J says that domains satisfying comparison property are quasi-isotropic.
%see also \cite{Hap}.

%\begin{prop}\label{prop4.1}
% $$\alpha_{D}\lesssim j_{D}\lesssim k_{D}\;\;\mbox{and}\;\;
 %\alpha_{D}\lesssim \tilde{a}_{D}\lesssim k_{D}.$$
%\end{prop}

\begin{proposition}\label{prop4.2}
For a domain $D\psubset \IR^n$ if $\alpha_D\approx j_D$, then $\tilde{a}_{D}\approx k_{D}$
holds.
\end{proposition}

% \begin{nonsec}{\bf Solution to Example \ref{thm1.8}}
% \end{nonsec}
The following examples show that the necessary part in Problem \ref{question-1.3} does not
hold irrespective of whether $D$ is bounded or unbounded.

\begin{example}\label{thm1.8}
Let $D$ be the domain as in Example {\rm \ref{thm1.6}}. Then $D$ is
a quasi-isotropic domain, but $D$ is not a John disk.
\end{example}

\noindent{\em Solution.}
% \subsection{\bf Solution to Example \ref{thm1.8}}
%Let $D=\{x+iy:\,x>0,|y|< 1\}$.
It follows from item (3) in Theorem C that $D$
is not a John disk.
% and therefore, it is not a uniform domain.

Since $D$ is convex, by  \cite[Theorem 4.2]{Se}, it follows that
$j_{D}\leq 2\alpha_{D}$ and so, we have
$$j_{D}(z_{1},z_{2})\leq 2 \alpha_{D}(z_{1},z_{2})\leq 2 \tilde{a}_{D}(z_{1},z_{2})
$$
for all $z_{1},z_{2}\in D$. Thus, Theorem J implies that $D$ is
quasi-isotropic.  \hfill $\Box$

\begin{example}\label{thm1.9}
Let $D=D_{1}\backslash D_{2}$, where $D_{1}=\{x+iy:\, -1<
x<0,-1< y< 0\}$ and $D_{2}=\{x+iy: \, x^{2}+y^{2}\leq1,x<0,y<0\}$.
Then $D$ is quasi-isotropic, but not a John disk.
\end{example}

\noindent{\em Solution.}
First we prove that $D$ defined in Example \ref{thm1.9} is quasi-isotropic.
By Theorem J and Proposition \ref{prop4.2}, it suffices to prove the following:
\be\label{thm1.9-eq1}
\alpha_{D}\thickapprox j_{D}.
\ee
It is well-known that $\alpha_{D}(z_{1},z_{2})\leq j_{D}(z_{1},z_{2})$
for all $z_{1},z_{2}\in D$. So, in order to prove (\ref{thm1.9-eq1}),
we only need to prove
\be\label{thm1.9-eq2}
j_D\lesssim \alpha_D.
\ee
For any $z_{1},z_{2}\in D$, without loss of generality, we may
assume that $\delta(z_{1})=\ds\min_{j=1,2}\delta(z_j)$. Let $z\in \partial D$
be such that
$\delta(z_1)=|z-z_1|$.

We need to deal with two cases.

\smallskip

{\bf Case I:} The case $|z_{1}-z_{2}|\geq 3\delta(z_{1})$.

We see from the first part of the solution of Example \ref{thm1.7} that
\be\label{thm1.9-eq3}
j_D(z_1,z_2)\le 6\, \alpha_D(z_1,z_2).
\ee

%We have
%\begin{eqnarray*}
%\alpha_{D}(z_{1},z_{2}) & \geq & \log\frac{|z-z_{2}|}{|z-z_{1}|}\\
%& \geq & \log\left(1+\frac{|z_{1}-z_{2}|}{3\delta(z_{1})}\right).
%\end{eqnarray*}

%By Lemma \ref{lem2.2}, we know
%\begin{eqnarray*}
%j_{D}(z_{1},z_{2}) & \leq &
%2\log\left(1+\frac{|z_{1}-z_{2}|}{\delta(z_{1})}\right)\\
%& \leq & 6\alpha_{D}(z_{1},z_{2}).\hspace{9.3cm}(4.3)
%\end{eqnarray*}

\smallskip

{\bf Case II:} The case $|z_{1}-z_{2}|< 3\,\delta(z_1)$.

Define
%Let $L_{1}=\{x+iy: \, x=-1,-1\leq y \leq 0\}$,
%$L_{2}=\{x+iy: \,y=-1,-1\leq x \leq 0\}$ and
%$L_{3}=\{x+iy: \,x^{2}+y^{2}=1,x<0,y<0\}$.
\begin{eqnarray*}
L_{1}&=&\{x+iy: \, x=-1,-1\leq y \leq 0\}\\
L_{2}&=&\{x+iy: \,y=-1,-1\leq x \leq 0\}\\
L_{3}&=&\{x+iy: \,x^{2}+y^{2}=1,x<0,y<0\}.
\end{eqnarray*}
If $z\in L_{1}\cup L_2$, then $\angle z_{2}z_{1}z\geq
\frac{\pi}{2}$ and there must exist a point $z_{0}\in L_{1}\cup L_2$ such that
$$\angle z_{0}z_{1}z=\frac{\pi}{6} ~\mbox{ and }~
\frac{2\pi}{3}\leq \angle z_{2}z_{1}z_{0}\leq \pi.
$$
Then it follows that
\begin{eqnarray*}
\alpha_{D}(z_{1},z_{2}) & \geq & \log\frac{|z_{2}-z_{0}|}{|z_{1}-z_{0}|}\\
& = & \log \frac{\sqrt{|z_{1}-z_{2}|^{2}+|z_{1}-z_{0}|^{2}-
2\cos\angle z_{2} z_{1}z_{0}|z_{1}-z_{2}||z_{1}-z_{0}|}}{|z_1-z_0|}\\
& \geq & \log\sqrt{1+\frac{\sqrt{3}|z_{1}-z_{2}|}{2\delta(z_{1})}}.
\end{eqnarray*}

Moreover, by Lemma \ref{lem2.2}, we get \be\label{thm1.9-eq4}
j_{D}(z_{1},z_{2})\leq \frac{8}{\sqrt{3}}\alpha_{D}(z_{1},z_{2}). \ee
%\smallskip
If $z\in L_{3}$, then $\angle z_{2}z_{1}z\geq \frac{\pi}{3}$ and
there must exist a point $z_{0}\in L_{3}$ such that $\angle
z_{2}z_{1}z_{0}=\frac{5\pi}{9}$.
\smallskip
It follows that there must exist $c_1$ such that
\be\label{thm1.9-eq5} |z_{1}-z_{0}|\leq c_1\,\delta(z_{1}), \quad
c_1=1+\frac{\sin\frac{2\pi}{9}}{\sin\frac{\pi}{9}}. \ee Then
(\ref{thm1.9-eq5}) implies that
\begin{eqnarray*}
\alpha_{D}(z_{1},z_{2}) & \geq &
\log\frac{|z_{2}-z_{0}|}{|z_{1}-z_{0}|}\\
& = & \log\frac{\sqrt{|z_{1}-z_{2}|^{2}+2\sin
\frac{\pi}{18}|z_{1}-z_{2}||z_{1}-z_{0}|}}{|z_{1}-z_{0}|}\\
& \geq & \log\sqrt{1+\frac{2\sin
\frac{\pi}{18}|z_{1}-z_{2}|}{c_1\,\delta(z_{1})}}.
\end{eqnarray*}
On the other hand, Lemma \ref{lem2.2} yields that
\be\label{thm1.9-eq6}
j_{D}(z_{1},z_{2})\leq 2\log\left(1+\frac{|z_{1}-z_{2}|}{\delta(z_{1})}\right)
\leq c_2\,\alpha_{D}(z_{1},z_{2}), \quad c_2=\frac{c_{1}}{\sin\frac{\pi}{18}}.
\ee
%where $c_2=\frac{c_{1}}{\sin\frac{\pi}{18}}$.
Therefore, using (\ref{thm1.9-eq3}), (\ref{thm1.9-eq4}) and (\ref{thm1.9-eq6}),
we obtain (\ref{thm1.9-eq2}).
%\smallskip

Now we prove that $D$ is not a John disk. For this, we let
$$x_{0}=-1,~ x_{1}=-1-ti\in L_{1} ~\mbox{ and }~ x_{2}=s-ti\in L_{3} \qquad
(0<t<1/4).
$$
Then $|x_{1}-x_{2}|=|1+s|$ and the straight crosscut
$[x_{1},x_{2}]$ divides the domain $D$ into two subdomains which
are denoted by $D_{1}$ and $D_{2}$. Obviously,
$$\min_{j=1,2}\diam(D_{j})=|x_{0}-x_{2}|=
\sqrt{|1+s|^{2}+t^{2}}=\sqrt{2+2s}.
$$

Suppose on the contrary that $D$ is a John disk. Then Theorem C implies that
there exists a constant $c$ such that
$$\min_{j=1,2}\diam(D_{j})\leq c|x_{1}-x_{2}|.
$$
But then
$$\lim_{s\to -1}\frac{\sqrt{2+2s}}{|1+s|}=\infty
$$
which is a contradiction. This completes the solution.
\hfill $\Box$

\begin{remark}\label{remark-1.10}
{\rm
By \cite[Example 4.4]{Ha2} it follows
that $\mathbb{H}\setminus [0,i]$ is not quasi-isotropic  but is a
John disk by Theorem C, where $\mathbb{H}$ denotes the upper
half-plane. As another motivation to Problem  \ref{question-1.3}, we observe by a
proof similar to \cite[Example 4.4]{Ha2} that there exist bounded simply
connected domains (eg. $\D\setminus[0,1)$) which are John but are not
quasi-isotropic. Also there exist doubly connected domains which are John
domains, but not quasi-isotropic. For instance, \cite[Example
3.11]{HPWS} gives that $\D\setminus\{0\}$ is not
quasi-isotropic but is clearly a John domain.
}
\end{remark}
%\medskip

%\noindent {\bf Remark 1.2}.\
\begin{remark}
{\rm
% Examples \ref{thm1.8} and \ref{thm1.9}
% show that the necessary part in Problem \ref{question-1.3} does not
% hold irrespective of whether $D$ is bounded or unbounded.
Remark \ref{remark-1.10}
shows that the sufficiency part in Problem \ref{question-1.3}
does not hold whether $D$ is bounded or unbounded. These
observations clearly provide us a solution to Problem
\ref{question-1.3}.
Examples \ref{thm1.8} and \ref{thm1.9} show that the
inequality $\widetilde{a}_D(z_1, z_2)\leq c\,j'_D(z_1,z_2)$ in
Theorem \ref{th1} cannot be removed.
}
\end{remark}

%\medskip

%\noindent {\bf Remark 1.3}.\
% \begin{remark}
% {\rm
% Examples \ref{thm1.8} and \ref{thm1.9} show that the
% inequality $\widetilde{a}_D(z_1, z_2)\leq c\,j'_D(z_1,z_2)$ in
% Theorem \ref{th1} cannot be removed.
% }
% \end{remark}

\chapter{ISOMETRIES OF SOME HYPERBOLIC-TYPE PATH METRICS}\label{chap4}
This chapter begins with a survey on isometries of
some hyperbolic-type path metrics and continues with some
old results as well as some new results.
Section \ref{chap4-sec1} deals with definitions and
brief introduction to the isometry problems of
such metrics.
In Section~\ref{kSect}, the work of H\"ast\"o \cite{Ha10} %on steps (2) and (3)
for the quasihyperbolic metric is described.
In addition, we relate how Herron,
Ibragimov and Minda \cite{HIM} used circular geodesics and the
curvature of the K--P metric to take care of its isometries
in most domains. Finally, in Section~\ref{newSect} we show
how the isometry problem can be solved for the K--P metric in
doubly connected domains using a new concept which we call the
hyperbolic medial axis, and we also present some new results for
the quasihyperbolic metric and Ferrand's metric.

\vskip 1cm
The results of this chapter have been published in:
{\bf P. H\"ast\"o, Z. Ibragimov, D. Minda, S. Ponnusamy and S.K. Sahoo}
Isometries of some hyperbolic type path metrics and the hyperbolic medial axis.
{\em In the Tradition of Ahlfors-Bers, IV (Ann Arbor, MI, 2005)}, 63--74,
{\em Contemp. Math.} {\bf 432}, Amer. Math. Soc., Providence, RI, 2007.

%%%%%%%%%%%%%%%%%%%%%%%%%%%%%%%%%%%%%%%%%%%%%%%%%%%%%%%%%%
%%%%%%%%%%%%%%%%%% SECTION 1 %%%%%%%%%%%%%%%%%%%%%%%%%%%%%
%%%%%%%%%%%%%%%%%%%%%%%%%%%%%%%%%%%%%%%%%%%%%%%%%%%%%%%%%%

\section{Introduction}\label{chap4-sec1}

A conformal path metric is a special kind of
Finslerian metric, in which the density depends only on the
location, not on the direction. If $D$ is a connected
subset of $\Rn$ and $p$ is a non-negative
real valued function defined on $D$, then we can define such a
metric by (\ref{conf-metric}) with the density function $p(z)$.
% $$d_p(x,y) = \inf_\gamma\int_{\gamma} p(z)\, |dz|,
% $$
% where $|dz|$ represents integration with respect to path-length, and
% the infimum is taken over all paths $\gamma$ joining
% $x,y\in D$.
If $p$ is a $C^2$ function, then we are in the
standard Riemannian setting, but there is nothing preventing us
from considering also a more general $p$.

In fact, choosing $p(z) = \delta(z)^{-1}$, where $\delta$ is
the distance-to-the-boundary function, gives us the
well-known quasihyperbolic metric\index{quasihyperbolic metric}
\index{metric!quasihyperbolic}.
Despite the prominence of this metric,
there have been almost no investigations of its geometry (some exceptions
are \cite{Li, MO}).
Part of the reason for this lack of geometrical investigations
is probably that the density of the quasihyperbolic metric is not
differentiable in the entire domain, which places the
metric outside the standard framework of Riemannian metrics.

At least two modifications of the quasihyperbolic metric have
been proposed that go some way to alleviate this problem.
J.\ Ferrand \cite{Fe} suggested replacing the density $\delta(z)^{-1}$
with $\sigma_D(x)$ defined by (\ref{Ferrand-density}).
Note that $\delta(x)^{-1} \le \sigma_D(x) \le 2 \delta(x)^{-1}$,
so the Ferrand metric\index{Ferrand's metric}\index{metric!Ferrand's}
and the quasihyperbolic metric
are bilipschitz equivalent.
Moreover, the Ferrand metric is M\"obius invariant, whereas the
quasihyperbolic metric is only M\"obius quasi-invariant.
A second variant was proposed more recently by R.\ Kulkarni and
U.\ Pinkall \cite{KP}, see also \cite{HMM}. The K--P metric is
defined by the density (\ref{KP-metric}).
This density satisfies the same estimates as Ferrand's density, i.e.\
$\delta(x)^{-1} \le \mu _D(x) \le 2 \delta(x)^{-1}$, and the K--P
metric is also M\"obius invariant. Although the Ferrand and the K--P
metrics are in some sense better behaved than the quasihyperbolic
metric, they suffer from the shortcoming that it is very difficult
to get a grip on the density, even in simple domains.

Despite this, D.\ Herron, Z.\ Ibragimov and D.\ Minda \cite{HIM}
recently managed to solve the isometry problem for the
K--P metric in most cases. Recall that by the isometry
problem\index{isometry problem}
for the metric $d$ we mean
characterizing mappings $f:\, D\to \Rt$ which satisfy (\ref{isometry})
%with \[ d_D(x,y)= d_{f(D)} (f(x), f(y)) \]
for all $x,y\in D$. Notice that in some sense we
are here dealing with two different metrics, due
to the dependence on the domain. Hence the usual
way of approaching the isometry problem is by looking at
some intrinsic features of the metric which are then
preserved under the isometry. Since irregularities
(e.g.\ cusps) in the domain often lead to more
distinctive features, this implies that the
problem is often easier for more complicated domains.
The work by Herron, Ibragimov and Minda \cite{HIM} bears out this
heuristic -- they were able to show that all isometries of the K--P
metric are M\"obius mappings except possibly in simply and doubly
connected domains. Their proof is based on studying the circular
geodesics of the K--P metric.
% We recall that there are three steps in characterizing isometries
% of a conformal metric like the quasihyperbolic metric:
% \begin{enumerate}
% \item[(1)] to show that they are conformal or anti-conformal;
% \item[(2)] to show that they are M\"obius; and
% \item[(3)] to show that they are similarities.
% \end{enumerate}
% Note that step (2) is trivial in dimensions $3$ and higher, and that
% step (3) is not relevant for M\"obius invariant metrics like the K--P
% metric and Ferrand's metric.
% The first step has been carried out by Martin and Osgood
% \cite[Theorem~2.6]{MO} for arbitrary domains assuming only that
% the density is continuous, so there is no more work to do there.
% In Section~\ref{kSect} the work of H\"ast\"o \cite{Ha10} on steps (2) and (3)
% for the quasihyperbolic metric is described.
For the quasihyperbolic metric,
formulas for the curvature were worked out in \cite{MO} (see
Section~\ref{kSect}) and were used in that paper to prove
that all the isometries of the disk are similarity mappings.
The proofs in \cite{Ha10} are based
on both the curvature and its gradient and work for domains with
$C^3$ boundary. %In Subsection~\ref{muSect}
% Also we relate how Herron,
% Ibragimov and Minda \cite{HIM} used circular geodesics and the
% curvature of the metric to take care of the second step for the K--P
% metric for most domains. Finally, in Section~\ref{newSect} we show
% how the isometry problem can be solved for the K--P metric in all
% doubly connected domains using a new concept which we call the
% hyperbolic medial axis, and we also give some minor new results for
% the quasihyperbolic metric and Ferrand's metric.

Besides the aforementioned works, the proofs in this chapter were
inspired by the work on the isometries of other (non-path) metrics
\cite{HI2,HI,HIL}.

%%%%%%%%%%%%%%%%%%%%%%%%%%%%%%%%%%%%%%%%%%%%%%%%%%%%%%%%%%%%%%%%%%%%%%

\subsection*{Notation}\label{genSubsect}

% If $D\subset \Rn$, we denote by $\partial D$ and $\overline{D}$
% its boundary and closure, respectively.
% For $x\in D\varsubsetneq\Rn$ we denote $\delta(x)=d(x,\partial D)=
% \min \{|x-z|\colon z\in \partial D\}$.
We tacitly identify $\Rt$ with $\C$, and speak about
real and imaginary axes, etc.
% We will often work with a mapping $f\colon D \to \Rn$. In such
% cases we will use a prime to denote quantities on the image
% side, e.g.\ $x' = f(x)$, $D' = f(D)$ and $\delta'(x)= d(x,\partial D')$,
% and so on.
By $B(x,r)$ we denote a disk with center $x$ and radius $r$,
and by $(x,y]$ or $[x,y)$ the half-open ( or semi open) segment between $x$ and $y$.

%, so its disks are the
%(open) disks of $\Rn$, complements of closed disks and half-planes.
The cross-ratio\index{cross-ratio} $|a,b,c,d|$ is defined by
\[ |a,b,c,d| = \frac{|a-c|\,|b-d|}{|a-b|\,|c-d|}\]
for distinct points $a,b,c,d\in\Rnbar$, with the understanding
that $|\infty-x| / |\infty-y| = 1$ for all $x,y\in \Rn$.
A homeomorphism $f\colon \Rnbar \to \Rnbar$ is a
M\"obius mapping\index{M\"obius mapping} if
\[ |f(a),f(b),f(c),f(d)| = |a,b,c,d| \]
for every quadruple of distinct points $a,b,c,d\in \Rnbar$.
A mapping of a subdomain of $\Rnbar$ is M\"obius if it is the
restriction of a M\"obius mapping defined on $\Rnbar$.
For more information on M\"obius mappings, see, for example,
\cite[Section~3]{Be2}.
Note that a M\"obius mapping can always be decomposed as
$i\circ s$, where $i$ is an inversion or the identity
and $s$ is a similarity\index{similarity} (i.e. a homeomorphism satisfying
$|s(x)-s(y)|=c\,|x-y|$ for some positive constant $c$).
% If $i$ is an inversion in a circle with center $x$, then
% we say that $i$ is centered at $x$, for short.

%%%%%%%%%%%%%%%%%%%%%%%%%%%%%%%%%%%%%%%%%%%%%%%%%%%%%%%%%%%%%%%%%%%%%%
%%%%%%%%%%%%%%%%%%%%%%%%%%%%%%%%%%%%%%%%%%%%%%%%%%%%%%%%%%%%%%%%%%%%%%
%%%%%%%%%%%%%%%%%%%%%%%%%%%%%%%%%%%%%%%%%%%%%%%%%%%%%%%%%%%%%%%%%%%%%%

\section{Isometries of $k_D$ and $\mu_D$; Known Results}\label{kSect}

\subsection{Isometries of the quasihyperbolic metric $k_D$}

By $f\in C^k$ we mean that $f$ is a $k$ times continuously differentiable
function. By a $C^k$ domain we mean a domain whose boundary can be locally
represented as the graph of a $C^k$ function.
To carry out step (2) of the isometry program
discussed in Section \ref{isometry-sec},
the following result was proved in \cite[Proposition~2.2]{Ha10}:

\begin{proposition}\label{c1lem2}
Let $D\subsetneq \Rt$ be a $C^1$ domain, and let $f\colon D \to \Rt$
be a quasihyperbolic isometry which is also M\"obius. If $D$ is not
a half-plane, then $f$ is a similarity.
\end{proposition}

Note that if we do not assume $C^1$ boundary, then there are some
domains with non-similarity isometries: punctured planes
$\Rt\setminus\{a\}$ and sector domains (i.e.\ domains whose boundary
consists of two rays). In both cases inversions centered at the
distinguished boundary point ($a$ or the vertex of the sector) are
also isometries. The previous proposition strongly suggests that
these are all the examples of domains with non-similarity
isometries.

An immediate consequence is the solution of the isometry problem
in higher dimensions for $C^1$ domains \cite[Corollary~2.3]{Ha10}:

\begin{corollary}
Let $D$ be a $C^1$ domain in $\Rn$, $n\ge 3$, which
is not a half-space. Then every quasihyperbolic isometry
is a similarity mapping.
\end{corollary}

The \textit{medial axis}\index{medial axis} of $D$ is the set of centers of
maximal balls (with respect to the inclusion order) in $D$.
The medial axis is denoted by $\MA(D)$. For
some mathematical investigations of the medial axis, see \cite{CCM,
Da}, and for an application to the quasiworld see \cite{Bi}.

By $R_\zeta$ we denote
the reciprocal of the curvature of $\partial D$ at the boundary
point $\zeta$.
The principal tool in \cite{Ha10} for attacking the main problem in
the isometry program (namely, step (3)) is the following curvature
formula based on the estimates of Martin and Osgood
\cite{MO}.

\begin{proposition}[Proposition~3.2, \cite{Ha10}]\label{MOprop}
Let $D\subsetneq \Rt$ be a $C^2$ domain and $z\in D\setminus \MA(D)$
have closest boundary point $\zeta \in \partial D$. Then
\[ \kappa_D(z) = - \frac{R_\zeta}{R_\zeta - \delta(z)}
= - \frac1{1 - \delta(z)/ R_\zeta }. \]
If $z$ lies on the medial axis, then $\kappa_D(z)= - \infty$.
\end{proposition}

Using this proposition, the following theorem was proved
in \cite[Theorem~4.3]{Ha10}.

\begin{theorem}
Let $D\subsetneq \Rt$ be a $C^3$ domain that is not a half-plane.
Then every isometry $f\colon D \to \Rt$ of the quasihyperbolic metric is a
similarity mapping.
\end{theorem}

The idea of the proof is the following: Let $z\in D\setminus \MA(D)$
and let $\zeta_z$ be its unique nearest boundary point. Then the
half-open segment $[z,\zeta_z)$ is a geodesic half-line with respect
to the quasihyperbolic metric. The proof of the theorem is based on
showing that this type of geodesic is somehow special and is thus mapped to
another geodesic half-line of the same type. There are a couple of
different cases based on the curvature $R_\zeta$ at the nearest
boundary point, but essentially this part of the proof is based on
Proposition~\ref{MOprop}. It is then shown that an isometry maps a
segment to a segment, which implies that it is a M\"obius mapping.
The proof is concluded by applying Proposition~\ref{c1lem2},
which says that the M\"obius isometry is a similarity.

In fact, the smoothness assumption on the boundary of the
domain can be dropped to $C^2$, except in two special
cases, namely, when the domain is strictly convex or strictly concave!
Corollary~\ref{k-cor}, proved below, takes care of the
concave case, so only the convex case remains.

%%%%%%%%%%%%%%%%%%%%%%%%%%%%%%%%%%%%%%%%%%%%%%%%%%%%%%%%%%%%%%%%%%%%%%
%%%%%%%%%%%%%%%%%%%%%%%%%%%%%%%%%%%%%%%%%%%%%%%%%%%%%%%%%%%%%%%%%%%%%%
%%%%%%%%%%%%%%%%%%%%%%%%%%%%%%%%%%%%%%%%%%%%%%%%%%%%%%%%%%%%%%%%%%%%%%

\subsection{Isometries of the K--P metric $\mu_D$}\label{muSect}

{\it{Extremal disks}} and {\it{circular geodesics}} are important
objects when discussing the isometries of the K--P metric. Consider
a domain $D$ in $\C$ with $\card(\p D)\geq 2$. A disk or a
half-plane $B\subset D$ with $\card(\p B\cap\p D)\ge 2$ is called an
extremal disk. We call $\Gamma$ a circular geodesic in $D$ if there
exists an extremal disk $B\subset D$ such that $\Gamma$ is a
hyperbolic geodesic line in $B$ with endpoints in $\partial
B\cap\partial D$. While the definition may not suggest the
importance of extremal disks, it is the existence of a unique
extremal disk associated to each point in the domain that plays a
crucial role in the study of the K--P metric.

More precisely, given
a domain $D\subset\C$ with $\card(\p D)\geq 2$ and a point $z\in D$,
let $i_z$ be the inversion in a circle centered at $z$ with radius
$1$. Then the complement of $i_z(D)$ is a compact set in $\C$ and
hence by Jung's Theorem (see \cite[11.5.8, p.~357]{Ber}) there
exists a unique disk $B$ of smallest radius whose closure contains
the set. In particular, $\card(\p B\cap\p D)\geq 2$ and $z\in
i_z(\Cc\setminus\overline{B})\subset D$. Hence the set
$i_z(\Cc\setminus\overline{B})$ is an extremal disk, which is
properly called the extremal disk at $z$ and is denoted by $B_z$. An
observant reader will notice that $B_z$ is also the extremal disk
for each point of a circular geodesic contained in $B_z$. Using more
delicate arguments it is proved that $B_z$ is the extremal disk for
each point of $\hat{K}_z$ and only for these points, where
$\hat{K}_z$ is the hyperbolic convex hull of the set $\p B_z\cap\p
D$ in $B_z$ (see \cite[Proposition 2.5]{HIM}). In particular, for
each pair of points $z,w\in D$, we have either $\hat{K}_z=\hat{K}_w$
or $\hat{K}_z\cap\hat{K}_w=\emptyset$.

Another important property of
the extremal disks is that $\mu_D(z)=\lambda_{B_z}(z)$, where
$\lambda_{B_z}$ is the density of the hyperbolic metric in $B_z$. In
particular,
$$
\mu_D(z)=\sup_{a,b\in\p B_z}\frac {|a-b|}{|a-z|\,|z-b|}
$$
and if $\gamma_z$ is any hyperbolic geodesic in $B_z$ passing
through $z$, then
$$
\mu_D(z)=\frac {|a(z)-b(z)|}{|a(z)-z|\,|z-b(z)|},
$$
where $a(z)\in\p B_z$ and $b(z)\in\p B_z$ are the endpoints of
$\gamma_z$. Hence by the monotonicity of the Ferrand metric we obtain
$$
\mu_D(z)=\sup_{a,b\in\p B_z}\frac
{|a-b|}{|a-z|\,|z-b|}=\sigma_{B_z}(z)\geq\sigma_D(z),
$$
and $\mu_D(z)=\sigma_D(z)$ if and only if $z$ lies on a circular
geodesic. Using Jung's Theorem we also obtain that $\mu_D(z)\leq
(2/\sqrt{3})\sigma_D(z)$ (see \cite{HIM} for details).

We also
need the following lower bound for the Ferrand (and hence for the
K--P) distance. Given a domain $D\subset\Cc$ with $\card(\p D)\geq
2$ and points $z,w\in D$, we consider the following distance
function
$$
s_D(z,w)=\log\Big (1+\sup_{a,b\in\p D}\frac
{|a-b|\,|z-w|}{|a-z|\,|b-w|}\Big).
$$
The function $s_D$, introduced by Seittenranta
\index{Seittenranta's metric}\index{metric!Seittenranta's}
\cite{Se}, defines a
metric in $D$ and since
$$
\lim_{w\to z}\frac {s_D(z,w)}{|z-w|}=\sup_{a,b\in\p D}\frac
{|a-b|}{|a-z|\,|z-b|}=\sigma_D(z)\qquad\text{for\ each}\quad z\in
D\cap\C,
$$
the Ferrand metric is the inner
metric of Seittenranta's metric.
Hence we have the aforementioned lower bound for the Ferrand and the
K--P distances $\sigma_D(z,w)$ and $\mu_D(z,w)$:
$$
s_D(z,w)\leq\sigma_D(z,w)\leq\mu_D(z,w)\qquad\text{for\  all}\quad
z,w\in D.
$$
In particular, the length of a curve in Seittenranta's metric is
smaller than its length in the K--P metric.
%%%%%  Length of a curve in Seittenranta's metric equals
%%%%%  length in Ferrand's metric.

Next we show that each circular geodesic is a geodesic line for both
the Ferrand and the K--P metrics, justifying its name. Let $\gamma$
be a circular geodesic in $D$ with endpoints $a,b$ and
$z,w\in\gamma$ be arbitrary points. Let $\gamma(z,w)$ be the subarc
of $\gamma$ joining $z$ and $w$. We need to show that
$$
\sigma_D(z,w)=\int_{\gamma(z,w)}\sigma_D(\xi)|d\xi|\qquad\text{and}\qquad
\mu_D(z,w)=\int_{\gamma(z,w)}\mu_D(\xi)|d\xi|.
$$
An easy observation shows that
$$
s_D(z,w)=\log\left(1+\frac {|a-b|\,|z-w|}{|a-z|\,|b-w|}\right)=\log\frac
{|z-b|\,|w-a|}{|z-a|\,|w-b|}
$$
for all $z,w\in\gamma$ such that the points $a,z,w,b$ are in this
order on $\gamma$. In particular,
$s_D(z_1,z_3)=s_D(z_1,z_2)+s_D(z_2,z_3)$ for all points $z_1,z_2,z_3$
in this order on $\gamma$ and, as a consequence, the
$s_D$-length of $\gamma(z,w)$ is equal to $s_D(z,w)$. Then
$$
\sigma_D(z,w)\leq\int_{\gamma(z,w)}\sigma_D(\xi)|d\xi|=\int_{\gamma(z,w)}\frac
{|a-b|}{|a-\xi|\,|\xi-b|}|d\xi|=s_D(z,w)\leq\sigma_D(z,w),
$$
and, similarly,
$$
\mu_D(z,w)\leq\int_{\gamma(z,w)}\mu_D(\xi)|d\xi|=\int_{\gamma(z,w)}\frac
{|a-b|}{|a-\xi|\,|\xi-b|}|d\xi|=s_D(z,w)\leq\mu_D(z,w).
$$
Thus, $\gamma$ is a geodesic for both the Ferrand and the K--P
metrics.

Next we discuss another type of geodesic for the K--P metric. As we
have mentioned above, given $z\in D$, the extremal disk $B_z$ is
also the extremal disk for all points in $\hat{K_z}$, where
$\hat{K_z}$ is the hyperbolic convex hull of the set $\p B_z\cap\p
D$ in $B_z$. In particular, $\mu_D(\xi)=\lambda_{B_z}(\xi)$ for all
$\xi\in\hat{K_z}$ and as a result, all the hyperbolic geodesics in
$B_z\cap\hat{K_z}$, which are circular arcs, are also geodesics in
the K--P metric $\mu_D$.  Hence through every point of $D$ there
passes a K--P geodesic which is a circular arc. Notice also that the
interior of $\hat{K_z}$ is non-empty if and only if $\card(\p
B_z\cap\p D)\geq 3$.

Now we are ready to present the result on the isometries of the K--P
metric.

\begin{theorem}\label{KPIsom} Let $f\colon D\to\Cc$ be a K--P isometry. Assume
that $D'=f(D)$  contains a point $z'$ so that $\card(\p B_{z'}\cap\p
D')\geq 3$. Then $f$ is the restriction of a M\"obius
transformation.
\end{theorem}

\begin{proof} Recall that $f$ is conformal and
$f^{-1}$ is also a K--P isometry. Put
$z=f^{-1}(z')$. Let $\gamma$ be a circular arc containing $z$ which
is also a K--P geodesic segment with the property that $f(\gamma)$
is contained in the interior of $\hat{K}_{z'}$. Since all the
geodesics in $\hat{K}_{z'}$ are circular arcs, so is $f(\gamma)$.
Using auxiliary M\"obius transformations, if necessary, we can
assume that $B_z=B_{z'}=B^2(0,1)$, that $z=z'=0$ and that both $\gamma$ and
$f(\gamma)$ are subarcs of the real interval $(-1,1)$. Then the fact
that $f$ is an isometry implies that $f$ is identity on $\gamma$ and
hence it is also identity on $D$, up to a M\"obius map. This
completes the proof.
\end{proof}

There is an alternative way to prove the Theorem~\ref{KPIsom} based on the
following result for holomorphic functions, which can be thought of
as an extension of Schwarz's Lemma. This approach also extends to
prove a similar theorem for the Ferrand metric (see
Theorem~\ref{Fisom}).

\begin{theorem}[Fact~2.1, \cite{HIM}]\label{ESL}
Let $D$ and $D'$ be hyperbolic
regions. Assume that $f$ is holomorphic in some neighborhood of $a\in D$
and takes values in $D'$. Let $\lambda$ and $\lambda'$ be the
densities of the hyperbolic metrics in $D$ and $D'$, respectively,
and let $f^\star[\lambda'](z)=\lambda'(f(z))|f^\prime(z)|$ be the pullback
of $\lambda'$ in $D$ to a neighborhood of $a$. Suppose that $f^\star[\lambda'](z)\leq\lambda(z)$
for all $z$ near $a$, with equality
holding at $z=a$. Then $f\colon D\to D'$ is a holomorphic covering
projection; in particular, $f^\star[\lambda']=\lambda$.
\end{theorem}

\begin{proof}[Analytic proof of Theorem~\ref{KPIsom}]
Let $f\colon D\to D'$ be a K--P isometry, hence conformal. Observe first that
\[ f^\star[\mu_{D'}]=\mu_{D'}(f(z))|f^\prime(z)|=\lim_{w\to z}\frac
{\mu_{D'}(f(w),f(z))}{|f(w)-f(z)|}\frac
{|f(w)-f(z)|}{|w-z|}=\lim_{w\to z}\frac
{\mu_D(w,z)}{|w-z|}=\mu_D(z).\]
The assumption in the theorem implies that there is a point $b=f(a)$
contained in $G'$, where $G'$ is the interior of $\hat K_b$ and
$K_b=B_b\cap\p D'$. Then in $f^{-1}(G')\cap B_a$ we have
\[ f^\star[\lambda_{B_b}]=f^\star[\mu_{D'}]=\mu_D\leq\lambda_{B_a}, \]
with equality holding at the point $z=a$ (see the proof of
\cite[Theorem 4.10]{HIM}). Theorem~\ref{ESL} now implies that $f$
maps $B_a$ conformally onto $B_b$ and hence is a M\"obius
map.
\end{proof}

%%%%%%%%%%%%%%%%%%%%%%%%%%%%%%%%%%%%%%%%%%%%%%%%%%%%%%%%%%%%%%%%%%%%%%
%%%%%%%%%%%%%%%%%%%%%%%%%%%%%%%%%%%%%%%%%%%%%%%%%%%%%%%%%%%%%%%%%%%%%%
%%%%%%%%%%%%%%%%%%%%%%%%%%%%%%%%%%%%%%%%%%%%%%%%%%%%%%%%%%%%%%%%%%%%%%

\section{Isometries of $\mu_D$, $\sigma_D$ and $k_D$; New Results}\label{newSect}

If a disk touches the boundary of a domain in exactly $k$ points,
then we call it \textit{$k$-extremal}. In this section we are
interested only in domains in which every extremal disk is
$2$-extremal\index{$2$-extremal domain}\index{domain!$2$-extremal}
-- we call such a domain also $2$-extremal, as there is no danger of
confusion. Examples of $2$-extremal domains include parallel strips,
angular sectors with angular openings strictly less than $\pi$,
annuli and many other domains and their images under M\"obius
mappings.

\begin{theorem}\label{Fer-eq-KP}
If $D$ is $2$-extremal domain in $\R^2$, then circular geodesics
foliate $D$. In particular, $\mu_D=\sigma_D$.
\end{theorem}

\begin{proof}
We will show that each point of $D$ lies on a circular geodesic and
that circular geodesics of $D$ are disjoint. Indeed, given an
arbitrary point $x\in D$, since $\card(B_x\cap\p D)=2$, the interior
of the set $\hat{K}_x$ is empty, whence $\hat{K}_x$ is a circular
geodesic containing $x$. Next if $\gamma_1$ and $\gamma_2$ are two
circular geodesics in $D$ and if $x\in\gamma_1\cap\gamma_2$, then
the endpoints of $\gamma_1$ and $\gamma_2$ belong to the set $\p
B_x\cap\p D$. Since $\card(B_x\cap\p D)=2$, we conclude that
$\gamma_1=\gamma_2$. The second part of the theorem now follows from
the fact that $\mu_D(x)=\sigma_D(x)$ whenever $x$ lies on a circular
geodesic (see Section~\ref{muSect}).
\end{proof}

Herron, Ibragimov and Minda proved that every planar $2$-extremal domain is
either simply or doubly connected, see \cite{HIM}. Given a
$2$-extremal disk $B$ in a domain $D$, we denote the unique circular
geodesic in $B$ by $\gamma(B)$. The (Euclidean) midpoint of the
circular geodesic is called the hyperbolic centere\index{hyperbolic center}
of $B$ and denoted
by $\HC(B)$. Let $E(D)$ be the set of all $2$-extremal disks  in $D$.

\begin{definition}
The set of hyperbolic centers of disks in $E(D)$ is called the
{\it{hyperbolic medial axis}}\index{hyperbolic medial axis} of the domain
$D$ and is denoted by
$\HMA(D)$.
\end{definition}

The hyperbolic medial axis is a modification of the usual medial
axis, whose definition was presented in Section~\ref{kSect}. In
certain respects the hyperbolic medial axis is better behaved than
the medial axis; for example, this is the case for smoothness and
localization properties. A more thorough investigation of these
issues is underway \cite{HI3}.

\begin{theorem}\label{c1-og-thm}
If $D\subset \R^2$ is a $2$-extremal domain, then $\HMA(D)$ is
locally the graph of a $C^1$ curve. If $B$ is a $2$-extremal disk in
$D$, then the circular geodesic $\gamma(B)$ and $\HMA(D)$ are
orthogonal at the hyperbolic center $\HC(B)$.
\end{theorem}

\begin{proof}
Let $B$ be a $2$-extremal disk corresponding to the boundary point
$a$ and $b$. Let $B_a$ be the disk in $B$ with $a$ and $\HC(B)$ as
boundary points; $B_b$ is defined similarly. Note that $B_a$ and
$B_b$ are horodisks in $B$. The circle $\partial B_a$ is
tangent to $\partial B$ at $a$, so it is orthogonal to $\gamma(B)$
there, hence also at $\HC(B)$. Thus both $\partial B_a$ and
$\partial B_b$ are orthogonal to $\gamma(B)$ at $\HC(B)$ and, in
particular, $\HMA(D) \cap U \subset U\setminus (B_a\cup B_b)$
for some sufficiently small neighborhood $U$ of $\HC(B)$.
It is clear that $\HMA(D)$ is orthogonal to $\gamma(B)$ and has
smoothness $C^{1}$ at $\HC(B)$.
\end{proof}

For metric densities which are at least $C^2$ smooth it is
well-known that geodesics are locally unique \index{locally unique}
(i.e., through a given
point in a given direction there is only one geodesic). For metrics
defined by densities with less smoothness this is not the case. For
instance for the quasihyperbolic metric in the strip $\{x\in \Rt
\colon |x_2|<1\}$ we know that a geodesic consists of a circular
arc, a segment lying in the real axis and a second circular arc (any
two of these three pieces may of course be degenerate). In
particular, geodesics are not locally unique on the real axis in the
real direction.

It was shown in Theorem~\ref{Fer-eq-KP} that there is a unique
circular geodesic through every point in a $2$-extremal domain.
We next prove a stronger statement: there is no geodesic
which is tangent to a circular geodesic.

\begin{lemma}\label{uniqGeod}
Smooth geodesics of the K--P metric are locally unique in $2$-extremal domains
in the direction of the circular geodesic.
\end{lemma}

\begin{proof}
Using an auxiliary M\"obius mapping we may restrict ourselves to the
circular geodesic $(-1,1)\subset \R$. More specifically, we show
that there is no other geodesic through the origin which is parallel
to the real axis there.

As before we denote by $\mu_D$ the density of the K--P metric in our
domain. By $\tilde \mu$ we denote the density of the K--P metric in
the domain $\{ x\in \Rt\colon |x_1|<1\}$. Obviously, $\tilde\mu(x) =
2(1-x_1^2)^{-1}$. As in the proof of Theorem~\ref{c1-og-thm} we find
that the level sets of $\mu_D$ are constrained by a pair of balls.
We restrict our attention to a small neighborhood of the origin.
Then the radii of these balls are greater than some constant $r>0$,
so we see that the level-sets of $\mu_D$ are approximated by the
level-sets of $\tilde\mu$ near the real axis. More precisely,
\[ \begin{split}
\frac1{\mu_D(x)} \ge \frac1{\tilde\mu(|x_1| + \Delta x)}
& = 1 - \Big(|x_1| + r \Big(1-\sqrt{1-x_2^2}\Big) \Big)^2 \\
& \ge 1 - \big(|x_1| + r x_2^2\big)^2
= 1 - x_1^2 + O(|x_1|\, x_2^2 +x_2^4).
\end{split}\]
Here $\Delta x$ is the maximal distance between the level-set of $\mu_D$
and $\tilde\mu$ at distance $x_2$ from the real axis.
A similar lower bound applies, so we have
\[ | \mu_D(x) - (1-x_1^2)^{-1} | \le C |x_1|\, x_2^2 \]
provided $x_2 =O(x_1)$.

Now suppose that there is a second smooth geodesic
through the origin that is parallel to the real axis.
Locally such a geodesic can be represented by
\[ y = f(x) = c_2 x^2 + O(x^3). \]
We also define $F\colon \R \to \Rt$ by $F(x) = (x,f(x))$. We assume
that $c_2>0$; the cases of negative coefficient or lower order
leading term are similar. We will show that for small enough
$\epsilon>0$, the segment $L_1=[0, F(\epsilon)]$ is shorter than the
curve $L_2=\{F(x) \colon 0<x<\epsilon\}$. Thus the latter curve is
certainly not a geodesic, which proves local uniqueness.
We also introduce the function $G\colon \R \to \Rt$ which parameterizes
$L_1$: $G(x) = (x,\frac{x}{\epsilon}f(\epsilon))$.

We start by calculating the length of $L_2$:
\[ \mu_D(L_2) = \int_{L_2} \mu_D(z)\, dz =
\int_0^\epsilon \mu(F(x))\sqrt{1 + f'(x)^2}\, dx.\]
We know that
\[ \mu(F(x)) = 1 - x^2 + O(x (c_2x^2)^2) = 1 - x^2 + O(x^5).\]
Thus we find that
\[\begin{split}
\mu(L_2)
& = \int_0^\epsilon \mu(F(x)) (1 + \tfrac12f'(x)^2 + O(f'(x)^4))\, dx \\
& = \int_0^\epsilon \big(1 - x^2 + O(x^5)\big)\big(1 + 2 c_2^2 x^2 +
O(x^4)\big)\, dx  \\
& = \int_0^\epsilon (1 + (2 c_2^2 -1) x^2 + O(x^4)) \, dx \\
& = \epsilon + (\tfrac23 c_2^2-\tfrac13)\epsilon^3  + O(\epsilon^5) .
\end{split}\]
For $L_1$ we calculate
\[ \begin{split}
\mu(L_1) & =
\int_0^\epsilon \mu(G(x))\sqrt{1 + (f(\epsilon)/\epsilon)^2}\, dx \\
& = (1 + \tfrac12c_2^2 \epsilon^2 + O(\epsilon^4))
\int_0^\epsilon \big(1-x^2 + O(x (\tfrac{x}\epsilon f(\epsilon))^2)\big)\,
dx \\
& = (1+ \tfrac12c_2^2 \epsilon^2) (\epsilon - \tfrac13 \epsilon^3) +
O(\epsilon^4) \\
& = \epsilon + (\tfrac12c_2^2 - \tfrac13) \epsilon^3 + O(\epsilon^4).
\end{split}
\]
A comparison with the expression for $\mu(L_2)$ shows that
$\mu(L_1) < \mu(L_2)$ whenever $\epsilon$ is sufficiently small, so
$L_2$ is not a geodesic.
\end{proof}

\begin{lemma}\label{orthogLem}
Let $D$ be $2$-extremal and doubly connected domain in $\Rt$.
Then a simple $C^1$ curve $\gamma$ in $D$ which is not contractible is
orthogonal to a circular geodesic at some point.
\end{lemma}

\begin{proof}
The claim clearly holds in the special case $D^c \subset \R$.
Thus we assume that $D$ has a non-degenerate boundary component.
Since our domain is doubly connected, it is a ring domain of the
type $G\setminus K$, where $G$ is open and $K$ is a
closed subset of $G$. We assume without loss of generality
that $\infty \not\in \overline{G}$.

By Theorem~\ref{Fer-eq-KP}, $D$ is foliated by circular geodesics.
We think of circular geodesics as directed curves which start
at $K$. We denote the circular geodesic through $x$ by $C_x$, and the
tangent of this curve at $x$ by $T_x$.
If $\gamma$ is not orthogonal to any of the circular geodesics,
then either $T_x \cdot \nabla \gamma(x) > 0$ for all $x$, or
$T_x \cdot \nabla \gamma(x) < 0$ for all $x$, where $\nabla$
denotes the gradient vector.

A point $x\in \gamma$
divides $C_x$ in two parts whose lengths are denoted by $l_K(x)$ and
$l_G(x)$.
Let $L_r$ be the set of points $x\in D$ such that $l_G(x) = r\, l_K(x)$.
As in Theorem~\ref{c1-og-thm}, we find that the simple closed curve
$L_r$ is $C^1$ and orthogonal to all circular geodesics.
Let $x_0 \in \gamma$ and $r = l_G(x_0) /l_K(x_0)$. Then $x_0\in L_r$.
Now if $T_x \cdot \nabla \gamma(x) > 0$ for all $x$, then
we see that $\gamma$ will not cross $L_r$ again. Therefore $\gamma$ cannot
be a closed curve, which is a contradiction. The same conclusion
holds if $T_x \cdot \nabla \gamma(x) < 0$ for all $x$.
Thus there must be a point of orthogonality between the curves.
\end{proof}

We are now ready to prove the main result of this chapter. Note that this result
combined with the results from Section~\ref{muSect} takes care of the
isometry problem for the K--P metric,
except in some cases of simply connected planar domains.

\begin{theorem}
Let $D$ be $2$-extremal and doubly connected domain in $\Rt$.
Then every isometry of $\mu_D$ is a M\"obius mapping.
\end{theorem}

\begin{proof}
Let $f$ an isometry of $\mu_D$.
As in the previous proof, we may assume that
$D=G\setminus K$, where $G$ is open and bounded, and $K$ is a
closed subset of $G$.
So every circular geodesic connects $\partial G$ to $\partial K$.
By Theorem
\ref{c1-og-thm}, for each $2$-extremal disk $B$ the hyperbolic medial
axis $\HMA(D)$ is orthogonal to the circular geodesic $\gamma(B)$ at
the hyperbolic center $\HC(B)$. We recall that every
isometry is a conformal mapping in the usual, Euclidean sense.
Hence, if we show that the
isometry coincides with a M\"obius map on an arc of a circle, then
it follows that $f$ is M\"obius.

By \cite[Theorem~B]{HIM} we know that every isometry of a domain which
is not 2-extremal is M\"obius. Note that $f^{-1}$ is also an
isometry. Thus, it is enough for us to consider the case when $f(D)$
is also 2-extremal. Now, Theorem \ref{c1-og-thm} shows that $\HMA(D)$ is
a $C^1$ curve and thus image of $\HMA(D)$ under $f$ is
a simple closed $C^1$ curve in $f(D)$. By Lemma~\ref{orthogLem},
$f(\HMA(D))$ is orthogonal to a circular geodesic $C'$ at some
point, say $f(x)$. Since $f^{-1}$ is an isometry, we find that
$f^{-1}(C')$ is a geodesic line in $D$. Since $f^{-1}$ is conformal,
we see that $f^{-1}(C')$ is orthogonal to $\HMA(D)$ at $x$.
But by Theorem~\ref{c1-og-thm}, $\HMA(D)$ is orthogonal to
a circular geodesic $C$ at $x$, and since geodesics are unique by
Lemma~\ref{uniqGeod}, it follows that $C=f^{-1}(C')$.
Since $f$ maps an arc of a circle to an arc of a circle and is a
K--P isometry, we easily conclude as in the first
proof of Theorem~\ref{KPIsom} that $f$ coincides with a M\"obius map
on $C$, which completes the proof.
\end{proof}

We end this section by presenting two results on the isometries of
the Ferrand and the quasihyperbolic metrics.

\begin{corollary}\label{k-cor}
Let $K\subset\Rt$ be convex, closed and non-degenerate,
and set $D=\Rt\setminus K$.
Then every isometry $f\colon D\to \Rt$ of the quasihyperbolic
metric is a similarity mapping.
\end{corollary}

\begin{proof}
It is clear that $\MA(D)=\emptyset$. From Proposition~\ref{MOprop}
we see that this implies $\MA(f(D))=\emptyset$. From this it
follows easily that $f(D)=\Rt\setminus K'$, where $K'$ is convex. Moreover,
we easily see that $k_G = \mu_G$ if $G$ is the complement of a
convex closed set. Thus $k_D = \mu_D$ and $k_{f(D)} = \mu_{f(D)}$.
Therefore $f$ is an isometry of the K--P metric, so the claim
follows from the previous theorem.
\end{proof}

The next result deals with Ferrand isometries in the special case
of a domain with a circular arc as part of its boundary. Although this
is quite a restrictive assumption, we would like to point out that so far
no results whatsoever have been derived for the isometries of this metric.

\begin{theorem}\label{Fisom}
Let $D\subset\R^2$ be a domain, and $f\colon D\to\R^2$ be a Ferrand
isometry. Assume that there exists a disk $B\subset D$
with the property that $\p B\cap\p D$ contains an arc $\gamma$.
Then $f$ is the restriction of a M\"obius transformation.
\end{theorem}

We will prove this claim using Theorem~\ref{ESL}. In order to
conform with the notation of that theorem, we will actually prove
the following claim, which is easily seen to be equivalent to
the previous theorem.

\begin{lemma}
Let $D\subset\R^2$ be a domain, and $f\colon D\to\R^2$ be a Ferrand
isometry. Assume that there exists a disk $B'\subset D'$, $D'=f(D)$,
with the property that $\p B'\cap\p D'$ contains an arc $\gamma'$.
Then $f$ is the restriction of a M\"obius transformation.
\end{lemma}

\begin{proof}
Since $f$ is conformal we see as in the second proof of Theorem~\ref{KPIsom}
that
\[
f^\star[\sigma_{D'}]=\sigma_{D'}(f(z))|f^\prime(z)|=
%\lim_{w\to z}\frac {\sigma_{D'}(f(w),f(z))}{|f(w)-f(z)|}\frac
%{|f(w)-f(z)|}{|w-z|}=\lim_{w\to z}\frac{\sigma_D(w,z)}{|w-z|}=
\sigma_D(z).
\]
Let $G'$ be the interior of the hyperbolic convex hull of $\gamma'$
in $B'$. Then one can easily see that
$\sigma_{D'}(x)=\lambda_{B'}(x)$ for all $x\in G'$. Let
$G=f^{-1}(G')$. First we claim that there exists a point $a\in G$
with $\sigma_D(a)=\mu_D(a)$. Using this claim we obtain
\[
f^\star[\lambda_{B'}]=f^\star[\sigma_{D'}]=\sigma_D\leq\mu_D\leq\lambda_{B_a},
\]
with equality holding at the point $x=a$. The proof is then completed
by invoking Theorem~\ref{ESL}. Thus, it remains to prove the claim.

Observe first that if there exist points $x, y\in G$ with
$\hat{K}_x\cap\hat{K}_y=\emptyset$, then due to the connectedness of $G$
there exists a point $a\in G\cap\p\hat{K}_x$ (i.e., $a$ lies on a
circular geodesic and hence $\sigma_D(a)=\mu_D(a)$; see
Section~\ref{muSect}). We can now assume that $\hat{K}_x=\hat{K}_y$
for all $x,y\in G$. In particular, all the points of $G$ share a
common extremal disk, which we can assume to be the unit disk $B$
about the origin. Put $K=\p B\cap\p D$, and let $\hat{K}$ be the
hyperbolic convex hull of $K$ in $B$. Note that $\hat{K}_x=\hat{K}$
for each $x\in G$. Since $G\subset\hat{K}\subset B$, $B$ is not
$2$-extremal. Hence we have a conformal map $f^{-1}$ of $B'$ into
$B$ with a property that $|f^{-1}(z)|\to 1$ as $z\to\gamma'$. Then the
Schwarz Reflection Principle implies that $f^{-1}$ has an analytic
continuation onto $\gamma'$ and by the identity theorem it can not
map $\gamma'$ onto a single point. Thus, the set
$f^{-1}(\gamma')\subset\p B\cap\p D$ contains an open arc, say
$\gamma$. Then $\sigma_D(a)=\mu_D(a)$ for each point of the
hyperbolic convex hull of $\gamma$, as required.
\end{proof}

%%%%%%%%%%%%%%%%%%%%%%%%%%%%%%%%%%%%%%%%%%%%%%%%%%%%%%%%%%%%%%%%%%%%%%
%%%%%%%%%%%%%%%%%%%%%%%%%%%%%%%%%%%%%%%%%%%%%%%%%%%%%%%%%%%%%%%%%%%%%%
%%%%%%%%%%%%%%%%%%%%%%%%%%%%%%%%%%%%%%%%%%%%%%%%%%%%%%%%%%%%%%%%%%%%%%

\chapter{\bf CERTAIN CLASSES OF UNIVALENT FUNCTIONS\\ AND
RADIUS PROBLEMS}\label{chap5}
 This chapter is devoted to the study of certain subclasses of
the class $\es$ of univalent analytic functions with an aim to
obtain coefficient conditions for functions to be
in some subclasses of $\es$ and radius problems.
Section \ref{chap5-sec1} is introductory in nature.
Section \ref{chap5-sec2} contains some
lemmas those are require to prove our results.
In Section \ref{chap5-sec3}, we obtain some coefficient
conditions for functions in $\es_p(\alpha)$ in series form.
Section \ref{chap5-sec4} discusses radius problems.
In Section \ref{chap5-sec5}, we obtain some conditions
for functions to be in the class $\mathcal{U(\lambda,\mu)}$.
Section \ref{chap5-sec6} concludes with some observations.

\vskip 1cm
The results of this chapter are from the published paper:
{\bf S. Ponnusamy and S.K. Sahoo} (2006)
Study of some subclasses of univalent functions and their radius properties.
{\em Kodai Math. J.} {\bf 29}(3), 391--405.

\section{Introduction and Preliminaries}\label{chap5-sec1}

% Denote by  $\mathcal{A}$ the class of all functions $f$, normalized by
% $f(0)=0=f'(0)-1$, that are analytic in the unit disk $\D
% =\{z\in\C:\,|z|<1\}$, and by $\mathcal{S}$ the subclass of univalent
% functions in $\D $. Denote by $\es^*$ the subclass
% consisting of
% functions $f$ in $\mathcal{S}$ that are starlike (with respect to
% origin),
% i.e. $tw\in f(\D)$ whenever $t\in [0,1]$ and $w\in f(\D)$.
%Analytically, $f\in\es^*$ if and only if ${\rm Re\,}
% (zf'(z)/f(z))\ge
% 0$ in $\D$.  A simple generalization of  $\es^*$
% is the so-called class of all starlike
% functions of order $\alpha$, $0\leq \alpha < 1$, denoted by
% ${\es}^*(\alpha )$. Indeed, $f\in \es^*(\alpha)$ if and
% only if ${\rm Re}(zf'(z)/f(z))\ge \alpha$ in $\D$. We set
% ${\es}^*(0) ={\es}^*$.
Recall that a function $f\in \mathcal{A}$ is
said to be in $\mathcal{U(\lambda,\mu)}$ if
$$ \left| f'(z)\left(\frac{z}{f(z)} \right)^{\mu+1}-1\right| \le
\lambda \quad (|z|<1)
$$
for some $\lambda\ge 0$ and $\mu>-1$.
We set $\mathcal{U}(\lambda,
1)=\mathcal{U}(\lambda) $, and $\mathcal{U}(1)=\mathcal{U}$.
In \cite{OP-05prea}, the authors studied a subclass
$\mathcal{P}(2\lambda)$ of
$\mathcal{U(\lambda)}$, consisting of functions $f$ for which
$$ \left|\left(\frac{z}{f(z)}\right)''\right| \le 2\lambda \quad
(|z|<1).
$$
We have the strict inclusion $\mathcal{P}(2)\subsetneq \mathcal{U}
\subsetneq\mathcal{S}$, see \cite{Ak,NOO-89,ON-72}. Moreover, a close
connection between the classes $\mathcal{P}(2\lambda)$ and
$\mathcal{U}(\lambda)$ is
given by ${\mathcal P}(2\lambda ) \subset {\mathcal U}(\lambda ),$ see
\cite{OP-01,OP-05prea}.

At this place it is important to remark that functions in $\mathcal U$
need not be starlike (see \cite{OP-07}). Also functions in $\es^*$
need not be in $\mathcal U$ (see \cite{FR-06}). Extremal functions of many
subclasses of $\es$ are in $\mathcal U$ (see \cite{OP-07}).
For instance if
$$L=\left\{z,\frac{z}{(1\pm z)^2},\frac{z}{1\pm z},\frac{z}{1\pm z^2},
\frac{z}{1\pm z+z^2}\right\},
$$
then each function in this collection is in ${\mathcal U}\cap \es^*$.
In \cite{Ob,PoSi1,PoSi2}, the authors considered the problem of finding
conditions on $\lambda$ and $\mu$ so that each function
in $\mathcal{U}(\lambda, \mu )$ is starlike or in some subsets of
$\mathcal{S}$. For example,  Ponnusamy and Singh \cite{PoSi1}  have shown that
$$\mathcal{U}(\lambda, \mu )\subseteq \es^* \quad
\mbox{if $\mu <0$ and $\displaystyle 0\le \lambda \leq
\frac{1-\mu}{\sqrt{(1-\mu)^2 +\mu ^2}}
:=\lambda ^ *(\mu )$}
$$
and in \cite{Ob}, Obradovi\'c proved that the above inclusion continues to
hold for $0<\mu\le 1$ and with the same bound for $\lambda$. The sharpness
part of these results may be obtained as a consequence of results from
\cite{RoRuSa}. However, it is not known whether each function $f$ in
$\mathcal{U}(1, \mu )$ (or more generally, $\mathcal{U}(\lambda, \mu )$ with $
\lambda ^ *(\mu ) <\lambda \leq 1$) is univalent in $\D$ for
certain
values of $\mu$ in the open interval  $(0,1)$. On the other hand, according to
a result due to Aksentiev \cite{Ak} (see also Ozaki and Nunokawa \cite{ON-72}
for a reformulated version as given by $\mathcal{U}$), we have the inclusion
${\mathcal U}(\lambda ) \subset {\mathcal S}$ for $0\leq \lambda \leq 1$. We
see that the Koebe function $z/(1-z)^2$ belongs to ${\mathcal U}$ but does not
belong to  $\es^* (\alpha )$ for any $\alpha >0$. In fact, the bounded
function $z+z^2/2$ belongs to ${\mathcal U}$ but not in ${\mathcal
S}^*(\alpha)$ for any $\alpha >0$.  That is, ${\mathcal U}\not\subset
{\es}^*(\alpha)$ for any $\alpha >0$. Thus, $\mathcal{U}
\subsetneq\mathcal{S}$ and the inclusion is strict as functions in
$\mathcal{S}$ are not necessarily in $\mathcal{U}$. Further work on these
classes, including some interesting generalizations of these classes, may be
found in \cite{OP-01,OPSV-02,PoV-05}.

A function $f\in \es^*(\alpha)$ is said to be in
$\mathcal{T}^*(\alpha)$ if it can be expressed as
$$f(z)=z- \sum_{k=2}^\infty |a_k|z^k .
$$
Functions of this form are discussed in detail by Silverman  \cite{Si-75}
and others \cite{SST-78}.

In this chapter we shall be mainly concerned with functions
$f\in\mathcal{A}$ of the form
\be\label{eq1}\left(\frac{z}{f(z)}\right)^\mu=1+\sum_{n=1}^\infty b_nz^n,
\quad z\in\D, \ee
where $(z/f(z))^\mu$ represents principal powers (i.e. the principal branch
of $(z/f(z))^\mu$ is chosen). The class of functions $f$
of this form for which  $b_n\geq 0$ is especially interesting and deserves
separate attention.  We remark that if $f\in\mathcal{S}$ then $z/f(z)$ is
nonvanishing and hence, $f\in\mathcal{S}$ may be expressed as
$$f(z) =\frac{z}{g(z)},\quad\mbox{ where }
g(z)=1+\sum_{n=1}^\infty c_nz^n,\quad z\in\D.
$$
These two representations are convenient for our investigation.
Finally, we introduce a subclass  $\es_p(\alpha)$,
$-1\le \alpha \le 1$, of starlike functions in the following way \cite{Ro-91}:
$${\es}_p(\alpha) = \left \{f\in {\mathcal S}:\,
\left |\frac{zf'(z)}{f(z)} -1\right |\leq {\rm Re}\,
\frac{zf'(z)}{f(z)}-\alpha, \quad z\in \D  \right \}.
$$
Geometrically, $f\in{\es}_p(\alpha)$ if and only if the domain
values of $zf'(z)/f(z)$, $z\in \D$, is the parabolic region
$$({\rm Im}\,w)^2\leq (1-\alpha)[2{\rm Re}\,w -(1+\alpha)].
$$
In \cite{Ro-91}, R{\o}nning has shown that
the class ${\es}_p(\alpha)$ must contain non-univalent functions
if $\alpha <-1$, and ${\es}_p(\alpha)\subset {\es}^*$ if
$-1\le \alpha \le 1$. We set
${\es}_p(0)={\es}_p$. The class of uniformly convex
functions was introduced by Goodman in \cite{Go-91a} (see also \cite{Go-91b}
where Goodman studied the class of uniformly starlike functions). Later
R{\o}nning \cite{Ro-93} studied these classes along with the class
$\es_p$. Moreover, from the work of R{\o}nning
\cite{Ro-93}, it follows easily that
$f(z)=z+a_nz^n$ is in ${\es}_p(\alpha)$ if and only if
$(2n-1-\alpha)|a_n|\leq 1-\alpha$.

We refer to Section \ref{coeff-subsec} for the definition
of radius problem. There are
many results of this type in the theory of univalent functions.  For
example, the $\es_p$ radius in $\es^*$ was found by
R{\o}nning in \cite{Ro-93} to be $1/3$. Also,
$\mathcal {P}(2)$ radius in $\mathcal{U}$ has been obtained by
Obradovi\'c and Ponnusamy in \cite{OP-05prea1} and is given by $2/3$.
At this place, it is appropriate to recall the following result:

\vspace{8pt}

{\bf Theorem A.} {\rm \cite[Theorem 4]{Ro-93}}
{\em If $f\in {\mathcal S}$, then  $\frac{1}{r}f(rz)\in \es_p$
if and only if $0<r\leq 0.33217\ldots $.}

\section{Lemmas}\label{chap5-sec2}

For the  proof of our results, we need the following result (see \cite[Theorem
11 in p.193 of Vol-2]{Go}) which reveals the importance of the area theorem in
the theory of univalent functions.

\begin{lemma}\label{lem2} %Let $0<\mu\le 1$
Let $\mu >0$ and $f\in\mathcal{S}$ be of the form $(\ref{eq1})$. Then we
have
$$\sum_{n=1}^\infty (n-\mu )|b_n|^2\leq \mu .
$$
\end{lemma}

Next we recall the well-known coefficient condition that is sufficient
for functions to be in $\mathcal{U}(\lambda)$ or $\mathcal{P}(2\lambda )$ or
$\es^*(\alpha)$, respectively.

\begin{lemma}\label{lem5}{\rm \cite{OPSV-02}}
Let $\phi(z)=1+\sum_{n=1}^\infty b_nz^n$ be a non-vanishing analytic
function in $\D$ and $f(z)=z/\phi(z).$ If $\sum_{n=2}^\infty
(n-1)|b_n|\le \lambda$, then we have
\begin{enumerate}
\item[{\rm (a)}] $f\in\mathcal{U}(\lambda)$
\item[{\rm (b)}] $f\in \mathcal{U}(\lambda)\cap \es^*$ for
$0<\lambda\le\frac{\sqrt{2-|b_1|^2}-|b_1|}{2}=\lambda_* (f);$
\end{enumerate}
\begin{enumerate}
\item[{\rm (c)}] Further, if $\sum_{n=2}^\infty n(n-1)|b_n|\le 2\lambda,$
then we have $f\in\mathcal{P}(2\lambda )$.
\end{enumerate}
\end{lemma}

In \cite{RST84}, it was shown that if
$\phi(z)=1+\sum_{n=1}^\infty b_n z^n$
is a non-vanishing analytic function in $\D$ and
$f(z)=z/\phi(z)$,
then $f\in\es^*(\alpha)$, $0\le \alpha \le 1$, whenever
$$\sum_{k=2}^\infty (k-1+\alpha)|b_k|\le \left\{\begin{array}{ll}
1-\alpha-(1-\alpha)
|b_1| & \mbox{ if $0\le \alpha\le 1/2$}\\
1-\alpha-\alpha |b_1| & \mbox{ if $1/2\le \alpha \le 1$.}\end{array}\right.
$$

\section{Coefficient Conditions for Functions in $\es_p(\alpha)$}\label{chap5-sec3}%$\mathcal{S}_p(\alpha)$}

\begin{theorem} \label{th2}
If a function $f$ of the form $(\ref{eq1})$ with $b_n\geq 0$
and $\mu >0$ is in $\es_p(\alpha)$, we then have
\be\label{th2-eq1}
\sum_{n=1}^{\infty}(2n-\mu(1-\alpha))b_{n}\leq \mu(1-\alpha).
\ee
\end{theorem}
\begin{proof}
Let $f\in \es_p(\alpha)$. Now, it is easy to see that
\be\label{th2-eq2}
z\frac{d}{dz}\left (\frac{z}{f(z)}\right)^{\mu}
= \mu\left[\left(\frac{z}{f(z)}\right)^{\mu}
-\left (\frac{z}{f(z)} \right )^{\mu +1}f'(z)\right].
\ee
Using the identity (\ref{th2-eq2}), we have
$$\left |\frac{zf'(z)}{f(z)}-1\right | \leq
{\rm Re}\, \left (\frac{zf'(z)}{f(z)} \right )-\alpha
\Leftrightarrow \left |\frac{-\frac{z}{\mu} \frac{d}{dz}
\left(\frac{z}{f(z)}\right)^{\mu}}
{\left(\frac{z}{f(z)}\right)^{\mu}}
\right|\leq {\rm Re}\, \frac{\left(\frac{z}{f(z)}\right)^{\mu}
-\frac{z}{\mu}\frac{d}{dz}\left (\frac{z}{f(z)}\right)^{\mu}}
{\left(\frac{z}{f(z)}\right)^{\mu}}-\alpha.
$$
Since $f$ is in the form (\ref{eq1}), the last inequality may be
equivalently written as
$$\frac{1}{\mu}\left|\frac{-\sum_{n=1}^{\infty}nb_nz^n}
{1+\sum_{n=1}^{\infty}b_nz^n} \right | \leq
{\rm Re}\, \left (1- \frac{1}{\mu}
\frac{\sum_{n=1}^{\infty}nb_nz^n}
{1+\sum_{n=1}^{\infty}b_nz^n} \right )-\alpha.
$$
If $z \in \D$ is real and tends to $1^{-}$ through reals,
then from the last inequality we have
$$\frac{1}{\mu}\left (\frac{\sum_{n=1}^{\infty}nb_n}{1+\sum_{n=1}^{\infty}b_n}\right )
\leq  1 -\alpha -\frac{1}{\mu}\left (
\frac{\sum_{n=1}^{\infty}nb_n}{1+\sum_{n=1}^{\infty}b_n}\right ),
$$
from which we obtain the desired inequality (\ref{th2-eq1}).
\end{proof}

The case $\mu =1$ leads to

\begin{corollary}\label{th2-cor}
Let $f\in \es_p(\alpha)$ be such that $z/f(z)= 1+\sum_{n=1}^{\infty}b_nz^n$
with $b_n\geq 0$. Then we have
$$\sum_{n=1}^{\infty}(2n-1+\alpha)b_{n}\leq 1-\alpha.
$$
\end{corollary}

\begin{theorem}\label{th3}
Let $z/f(z)$ be a nonvanishing analytic function of the form
$(\ref{eq1})$ with $\mu >0$. Then the condition
\be\label{th3-eq1}
\sum_{n=1}^{\infty}(2n+\mu(1-\alpha))|b_n|\leq \mu(1-\alpha)
\ee
is sufficient for $f$ to be in the class $\es_p(\alpha)$.
\end{theorem}\begin{proof}
As in the proof of Theorem \ref{th2}, we notice that
$$ \left |\frac{zf'(z)}{f(z)}-1\right |\leq {\rm Re}\,
\left (\frac{zf'(z)}{f(z)} \right )-\alpha
$$
is equivalent to
$$\left|-\frac{\sum_{n=1}^{\infty}nb_nz^n}{1+\sum_{n=1}^{\infty}
b_nz^n} \right |\leq \mu (1-\alpha) -{\rm Re}\,
\left(\frac{\sum_{n=1}^{\infty}nb_nz^n}
{1+\sum_{n=1}^{\infty}b_nz^n} \right ).
$$
Thus, to show that $f$ is in $\es_p(\alpha)$, it suffices
to show that the quotient
$$ -\frac{\sum_{n=1}^{\infty}nb_nz^n}{1+\sum_{n=1}^{\infty}b_nz^n}
$$
lies in the parabolic region
$$({\rm Im}\,w)^2\leq \mu(1-\alpha)[\mu(1-\alpha)+2{\rm Re}\,w].
$$
Geometrically, this condition holds if we can show that
\be\label{th3-eq2}
\left | \frac{\sum_{n=1}^{\infty}nb_nz^n}
{1+\sum_{n=1}^{\infty}b_nz^n} \right |\leq
\frac{\mu(1-\alpha)}{2},
\quad z\in \D .
\ee
From the condition (\ref{th3-eq1}), we obtain that
$$\sum_{n=1}^{\infty}(2n+\mu(1-\alpha))|b_n|\, |z|^n \leq
\mu(1-\alpha)
$$
and so
$$\sum_{n=1}^{\infty}n|b_n|\, |z|^n\leq \frac{\mu(1-\alpha)}{2}
\left (1-\sum_{n=1}^{\infty}|b_n|\, |z|^n\right ).
$$
In view of this inequality, we deduce that
$$\left |\frac{\sum_{n=1}^{\infty}nb_nz^n}
{1+\sum_{n=1}^{\infty}b_nz^n} \right | \leq
\frac{\mu(1-\alpha)}{2} \left (
\frac{1-\sum_{n=1}^{\infty}|b_n|\,|z|^n}{1-\sum_{n=1}^{\infty}
|b_n|\,|z|^n} \right ) =\frac{\mu(1-\alpha)}{2}
$$
which is exactly the inequality (\ref{th3-eq2}) and therefore,
$f\in\es_p(\alpha)$.
\end{proof}

\begin{corollary}\label{th3-cor}
Let $z/f(z)$ be a nonvanishing analytic function in $\D$ of
the form
$z/f(z)= 1+\sum_{n=1}^{\infty}b_nz^n$.
Then the condition
$$\sum_{n=1}^{\infty}(2n+1-\alpha)|b_n|\leq 1-\alpha
$$
is sufficient for $f$ to be in the class $\es_p(\alpha)$.
\end{corollary}

The case $\alpha =0$ of Corollaries \ref{th2-cor} and
\ref{th3-cor} has been obtained
recently by Obradovi\'c and Ponnusamy \cite{OP-05prea1}.

\section{Radius Problems}\label{chap5-sec4}

\begin{theorem}\label{psv-th4}
If $f\in \es$ is given by $(\ref{eq1})$ with $0<\mu<1$,
then $ \frac{1}{r}f(rz)\in
\es_p(\alpha)$ for $0<r\leq r_0$, where $r_{0}$ is the root of the integral equation
\be\label{psv-th4-eq1}
\frac{4r^2(1+\mu(2-\alpha)(1-r^2))}{(1-r^2)^2} +
\frac{r^2\mu^2(3-\alpha)^2}{1-\mu}\int_0^1\frac{dt}{1-r^2
t^{1/(1-\mu)}} =\mu(1-\alpha)^2.
\ee
\end{theorem}\begin{proof}
Let $f\in \mathcal{S}$ be given by $(\ref{eq1})$ with $0<\mu<1$.
Then $z/f(z)$ is nonvanishing in $\D$ and for $0<r\leq1$, we have
$$\left (\frac{z}{\frac{1}{r}f(rz)}\right )^\mu
=1+(b_{1}r)z+(b_{2}r^{2})z^{2}+\cdots .
$$
If
\be\label{psv-th4-eq2}
S:= \sum_{n=1}^{\infty}(2n+\mu(1-\alpha))|b_n|r^{n}\leq \mu(1-\alpha)
\ee
for some $r$, then $\frac{1}{r}f(rz)\in {\es}_{p}(\alpha)$,
by Theorem \ref{th3}. Now, using the Cauchy-Schwarz inequality and
Lemma \ref{lem2}, we see that
\begin{eqnarray*}
%\sum_{n=1}^\infty (2n+\mu(1-\alpha)) |b_n| r^n
S & = & \sum_{n=1}^\infty \sqrt{n-\mu}|b_n|
\frac{2n+\mu(1-\alpha)}{\sqrt{n-\mu}}r^n\\
& \leq & \left(\sum_{n=1}^\infty (n-\mu)|b_n|^2\right)^{\frac{1}{2}}
\left(\sum_{n=1}^\infty \frac{(2n+\mu(1-\alpha))^2}
{n-\mu} r^{2n}\right)^{\frac{1}{2}}\\
& \leq & \sqrt{\mu} \left(\sum_{n=1}^\infty \frac{(2n+\mu(1-\alpha))^2}
{n-\mu} r^{2n}\right)^{\frac{1}{2}}\\
& = & \sqrt{\mu} \left(\sum_{n=1}^\infty \frac{(2n+\mu(1-\alpha))^2
-\mu^2(3-\alpha)^2}{n-\mu} r^{2n}
+\mu^2(3-\alpha)^2\sum_{n=1}^\infty
\frac{r^{2n}}{n-\mu}
\right)^{\frac{1}{2}}\\
& = & \sqrt{\mu} \left(\sum_{n=1}^\infty
4(n+\mu(2-\alpha)) r^{2n}+ \mu^2(3-\alpha)^2\sum_{n=1}^\infty
\frac{r^{2n}}{n-\mu}\right)^\frac{1}{2}\\
&= & \sqrt{\mu}\left(\frac{4r^2(1+\mu(2-\alpha)(1-r^2))}{(1-r^2)^2}
+ \frac{r^2\mu^2(3-\alpha)^2}{1-\mu}\int_0^1\frac{dt}{1-r^2
t^{1/(1-\mu)}} \right)^{\frac{1}{2}}.
\end{eqnarray*}
In particular, if the last expression is less than or equal to
$\mu(1-\alpha)$, then (\ref{psv-th4-eq2}) holds which gives
the condition (\ref{psv-th4-eq1}).
\end{proof}

In the case $\mu=1$, Theorem \ref{psv-th4}
takes the following form which needs a special attention
as we see that the radius quantity depends on the second coefficient
of the given function $f$.

\begin{theorem}\label{psv-th5}
If $f\in \mathcal{S}$ is of the form $z/f(z)=1+\sum_{n=1}^\infty b_nz^n$,
then $ \frac{1}{r}f(rz)\in
\es_p(\alpha)$ for $0<r\leq r_0$, where $r_{0}$, which depends
on the second coefficient of $f$, is the root of the equation
$$\frac{4r^4(1+(3-\alpha)(1-r^2))}{(1-r^2)^2}-(3-\alpha)^2 r^2\ln (1-r^2)
= (1-\alpha -(3-\alpha)(r/2)|f''(0)|)^2.
$$
\end{theorem}\begin{proof}
Note that, for $f\in\mathcal{S}$ satisfying $z/f(z)=1+\sum_{n=1}^\infty b_nz^n$,
we have $b_1=-f''(0)/2$. Proceeding exactly as in the
proof of Theorem \ref{psv-th4} (but with $\mu=1$) and by considering summation
to run from $2$ to $\infty$, we obtain the required conclusion. So we omit the
details.
\end{proof}

We remark that, the case $\alpha=0$ of  Theorem \ref{psv-th5} is due to
Obradovi\'c and Ponnusamy \cite{OP-05prea1}.

Now we prove a generalized version of Lemma \ref{lem5}(a) which is useful
to prove our next result.

\begin{lemma}\label{lem1}
Let $0\leq \alpha <1$ and  $\phi(z)=1+\sum_{n=1}^{\infty}b_nz^n$ be a
non-vanishing analytic function in $\D$ satisfying the coefficient
condition
\be\label{eq11}
\sum_{n=1}^{\infty}(n-1+\alpha )|b_{n}|\leq \lambda (1-\alpha ).
\ee
Then the function $f$ defined by the equation $(z/f(z))^{1-\alpha} =\phi (z)$
is in ${\mathcal U}(\lambda , 1-\alpha )$.
\end{lemma}
\begin{proof}
Let $f$ be given by $(z/f(z))^{1-\alpha} =\phi (z)$, where $\phi (z)\neq 0$ in
$\D$, and we choose here the principal branch so that
$(z/f(z))^{1-\alpha}$ at $z=0$ is $1$. Then the power series representation of
$\phi $ and the coefficient condition (\ref{eq11}), lead to
$$\left |\left (\frac{z}{f(z)}\right )^{2-\alpha } f'(z) -1 \right |
= \left |-\frac{1}{1-\alpha }\sum_{n=1}^{\infty}(n-1+\alpha)b_nz^n \right |
\leq \lambda
$$
and therefore, by the definition of the class, $f$ is in ${\mathcal U}(\lambda
, 1-\alpha )$.
\end{proof}

The following result determines the $\mathcal{U}(\lambda,\mu)$ radius in $\mathcal{S}$.

\begin{theorem}\label{th8}
Suppose that $f\in\mathcal{S}$,  $0\leq \alpha <1$, $\lambda >0$ and
$$r_{\alpha, \lambda} = \frac{\lambda \sqrt{2(1-\alpha)}}
{\left [\sqrt{\left (\alpha +2\lambda ^2(1-\alpha) \right )^2
+4\lambda ^2(1-\alpha)^2(1- \lambda ^2)} +(\alpha +2\lambda ^2(1-\alpha))\right ]^{1/2}}.
$$
Then we have $\frac{1}{r}f(rz)\in \mathcal{U}(\lambda ,1-\alpha )$ for
\be\label{eq15}
0<r\leq r_{\alpha, \lambda} .
\ee
In particular, $\frac{1}{r}f(rz)\in \mathcal{U} (1,1-\alpha )$ for
$0<r\leq \sqrt{(1-\alpha )/(2-\alpha )}$.
\end{theorem}
\begin{proof}
Let $f\in\mathcal{S}$. Then $z/f(z) \neq 0$ in $\D$. So, we may assume
$f$ is of the form
\be\label{eq13}
\left ( \frac{z}{f(z)}\right )^{1-\alpha}=1+\sum_{n=1}^\infty b_nz^n.
\ee
Now, Lemma \ref{lem2} gives
$$\sum_{n=1}^\infty (n-1+\alpha)|b_n|^2\leq 1-\alpha .
$$
On the other side, for $0<r\leq 1$, we obtain from (\ref{eq13}) that
$$\left (\frac{z}{\frac{1}{r}f(rz)} \right )^{1-\alpha}
= 1+\sum_{n=1}^\infty (b_nr^n)z^n.
$$
According to Lemma \ref{lem1}, it suffices to verify the inequality
$$ \sum_{n=1}^{\infty}(n-1+\alpha )|b_{n}r^n|\leq \lambda (1-\alpha )
$$
for $0<r\leq r_{\alpha, \lambda}$. Now, as before, we have
\begin{eqnarray*}
\sum_{n=1}^{\infty}(n-1 +\alpha )|b_{n}r^{n}|
& \leq  &
\left (\sum_{n=1}^{\infty}(n-1+\alpha )|b_{n}|^2 \right )^{1/2}
\left (\sum_{n=1}^{\infty}(n-1 +\alpha)r^{2n}\right )^{1/2} \\
&\leq  & \sqrt{1-\alpha} \left (\frac{r^4}{(1-r^2)^2} +\alpha
\frac{r^2}{1-r^2} \right )^{1/2}\\
&= & \sqrt{1-\alpha} \left (\frac{r}{1-r^2} \right )
\left (\alpha +(1-\alpha)r^2\right)^{1/2} \\
&\leq  & \lambda (1-\alpha ),
\end{eqnarray*}
if $\frac{r}{1-r^2} \sqrt{\alpha +(1-\alpha)r^2} \leq  \lambda
\sqrt{1-\alpha}$. Note that
$$\frac{r}{1-r^2} \sqrt{\alpha +(1-\alpha)r^2} \leq  \lambda \sqrt{1-\alpha}
$$
is equivalent to (\ref{eq15}), and so we complete the proof.
\end{proof}

\section{Conditions for Functions to be in $\mathcal{U}(\lambda,\mu)$}
\label{chap5-sec5}

To present our next result, we consider the class of functions of
Bazilevi\v{c} type, see
\cite{Li01,SM83/84}. The result is simple and surprising as it identifies a
subclass which lies in $\mathcal{U}(\lambda,\mu)$. This generalizes the result
of Obradovi\'c and Ponnusamy, see \cite[Theorem 5]{OP-05prea1}.

\begin{theorem}\label{psv-th6}
Let $0<\mu\le 1$. If $f\in\mathcal{S}$ is given by $(\ref{eq1})$ with $b_n\ge
0$, and satisfies the condition that ${\rm Re}\left(f'(z)\left(\frac{f(z)}{z}
\right)^{\mu-1}\right)>0$.  Then $f\in\mathcal{U}(1,\mu)$.
\end{theorem}\begin{proof}
Using the equation (\ref{th2-eq2}), we notice that
\begin{eqnarray*}
{\rm Re}\left(f'(z)\left(\frac{f(z)}{z} \right)^{\mu-1}\right)>0
&\Leftrightarrow & {\rm Re}\,
\frac{\frac{zf'(z)}{f(z)}} {\left( \frac{z}{f(z)}\right)^\mu}>0\\
&\Leftrightarrow & {\rm Re}\, \frac{\left(\frac{z}{f(z)}\right)^{\mu}-
\frac{z}{\mu}\frac{d}{dz} \left (\frac{z}{f(z)}\right)^{\mu}}
{\left( \frac{z}{f(z)} \right)^{2\mu}}>0\\
&\Leftrightarrow & {\rm Re}\,\frac{1+\sum_{n=1}^\infty(1-n/\mu)b_n z^n}
{(1+\sum_{n=1}^\infty b_n z^n)^2}>0.
\end{eqnarray*}
Since $b_n\ge 0$, allow $z\to 1^{-}$ along the real axis, we get
$${\rm Re}\,\frac{1-\sum_{n=1}^\infty(n/\mu-1)b_n}{(1+\sum_{n=1}^\infty
b_n)^2}\ge 0,
$$
which gives that
$$\sum_{n=1}^\infty (n-\mu)b_n\le \mu
$$
and so by Lemma \ref{lem1}, we have $f\in\mathcal{U}(1,\mu)$.
\end{proof}

\begin{theorem}\label{psvth8}
Let $0<\mu\le 1$. A function $f$ of the form $(\ref{eq1})$ with $b_n\ge 0$ and
$z/f(z)\neq 0$, is in $\mathcal{U}(1,\mu)$ if and only if
\be\label{psvth8-eq1} \sum_{n=1}^\infty (n-\mu)b_n\le \mu.
\ee
\end{theorem}\begin{proof}
In view of Lemma \ref{lem1}, it suffices to prove the necessary part. To do this, we let $f\in\mathcal{U}(1,\mu)$ and $f$ is of the form
(\ref{eq1}). Then using (\ref{th2-eq2}), we get
$$\left|\left(\frac{z}{f(z)}\right)^{\mu+1}f'(z)-1\right|
=\left|\left(\frac{z}{f(z)}\right)^{\mu}-\frac{z}{\mu}\frac{d}{dz}
\left(\frac{z}{f(z)}\right)^{\mu}-1\right|
=\frac{1}{\mu}\left|\sum_{n=1}^\infty (n-\mu)b_n z^n\right|\le 1.
$$
Because $b_n\ge 0$, letting $z\to 1^{-}$ along the real axis, we obtain the coefficient condition (\ref{psvth8-eq1}).
\end{proof}

The following result gives a sufficient condition for starlike
functions of order $\alpha$ to be in the class
${\mathcal U}(\lambda,\mu)$.

\begin{theorem}\label{psvth9}
If $f\in {\es}^*(\alpha)$ is of the form $(\ref{eq1})$ with $b_n\geq 0$ and $\mu >0$, then
\be\label{psvth9-eq1}
\sum_{n=1}^\infty (n-\mu(1-\alpha))b_n\leq \mu(1-\alpha).
\ee
In particular, $f\in{\mathcal U}(1-\alpha,\mu)$.
\end{theorem}\begin{proof}
It is easy to see that
$$ f\in {\es}^*(\alpha) \Leftrightarrow
{\rm Re}\left(\frac{zf'(z)}{f(z)}\right)\geq  \alpha
\Leftrightarrow \left|\frac{{\frac{zf'(z)}{f(z)}-1}}
{{\frac{zf'(z)}{f(z)}+1-2\alpha}}\right|\leq 1.
$$
Now, using this relation and the identity (\ref{th2-eq2}), we have the
following
\begin{eqnarray*}
\left|\frac{{\frac{zf'(z)}{f(z)}-1}}{{\frac{zf'(z)}{f(z)}+1-2\alpha}}\right|
& = & \left|\frac{-z\frac{d}{dz}\left(\frac{z}{f(z)}\right)^\mu
}{2\mu(1-\alpha)\left(\frac{z}{f(z)}\right)^\mu -z\frac{d}{dz}
\left(\frac{z}{f(z)}\right)^\mu}\right|\\
& = & \left|\frac{{-\sum_{n=1}^\infty nb_nz^n}}
{{2\mu(1-\alpha)\left(1+\sum_{n=1}^\infty b_n z^n\right)
-\sum_{n=1}^\infty nb_n z^n}}\right|\leq 1.
\end{eqnarray*}
Since $b_n \ge 0$, if $z\to 1^-$ along the real axis, we see from the last
inequality that
$$\frac{\ds{\sum_{n=1}^\infty nb_n}}{\ds{2\mu(1-\alpha)
-\sum_{n=1}^\infty (n-2\mu(1-\alpha))b_n}}\leq 1.
$$
This gives the desired inequality (\ref{psvth9-eq1}).

Finally, since $n-\mu \leq n-\mu(1-\alpha)$, we have
$$\sum_{n=1}^\infty (n-\mu)b_n \leq \sum_{n=1}^\infty (n-\mu(1-\alpha))b_n\leq
\mu(1-\alpha).
$$
From Lemma \ref{lem1}, we conclude that $f\in {\mathcal U}(1-\alpha,\mu)$.
\end{proof}

As a consequence of Theorem \ref{psvth9}, we next see
that $\mathcal{T}^*(\alpha)\subset\mathcal{U}(1-\alpha)$.

\begin{corollary}\label{psvth9-cor1}
If $f(z)=z-\sum_{n=2}^\infty |a_n|z^n$ is in ${\es}^*(\alpha)$, then
$f\in {\mathcal U}(1-\alpha).$
\end{corollary}\begin{proof}
Let $f\in {\es}^*(\alpha)$ be of the form $f(z)=z-\sum_{n=2}^\infty |a_n|z^n$.
Then $z/f(z)$ is nonvanishing in the unit
disk and so it can be expressed as
$$\frac{z}{f(z)}=\frac{1}{1-|a_{2}|z-|a_{3}|z^{2}- \cdots }
= 1+b_{1}z+b_{2}z^{2}+\cdots ,
$$
where $b_{n}\geq 0$ for all $n\in \IN$. Then by Theorem \ref{psvth9}, $f\in
{\mathcal U}(1-\alpha).$
\end{proof}

From \cite{OPSV-02}, we collect the following result.

\begin{lemma}\label{v-thsis-p.19}
Let  $0\leq \lambda, \gamma \leq 1$ and $f\in \mathcal {U}(\lambda)$.
Define
$$\lambda ^*_{\gamma}=\frac{-|f''(0)|\cos(\pi\gamma /4)
+\sin (\pi\gamma /4)\sqrt{16\cos ^2(\pi\gamma /4)
- |f''(0)|^2}} {2\cos(\pi \gamma /4)}
$$
and let $\lambda ^{\mathcal {R}_{\gamma}}$ be given by the inequality
$$\sin (\pi\gamma /2) \sqrt{4-\lambda^2}\geq  (|f''(0)|+\lambda )
\sqrt{4-(|f''(0)|+\lambda )^2} + \lambda \cos (\pi\gamma /2).
$$
Then
\begin{enumerate}
\item [{\rm (i)}] $f\in \mathcal {U}(\lambda) ~\Ra ~f\in\mathcal {S}_{\gamma }
$ for $0<\lambda\leq \lambda^*_{\gamma}/2$,
\item [{\rm (ii)}] $f\in \mathcal {U}(\lambda) ~ \Ra~ f\in \mathcal
{R}_{\gamma}$ for $0<\lambda\leq \lambda^{\mathcal {R}_{\gamma}}/2$,
\end{enumerate}
where
\begin{eqnarray*}
\mathcal{R}_\gamma: & = & \left \{f\in\mathcal{A}:\, |\arg
f'(z)|\leq \frac{\pi\gamma}{2}\right \} \quad \mbox{and }\\
\es_\gamma: & = &\left \{f\in\mathcal{A}:\,
\left |\arg \left ( zf'(z)/f(z) \right )\right |\leq \frac{\pi\gamma}{2}\right \}.
\end{eqnarray*}
\end{lemma}

Using the containment results of Lemma \ref{v-thsis-p.19}
and Corollary \ref{psvth9-cor1}, one can derive a number of
interesting results. For instance, we obtain the following:

\begin{corollary}\label{psvth9-cor2}
If $0\leq \gamma\le 1$ and $f(z)
=z-\sum_{n=3}^\infty |a_n|z^n\in\es^*(1-\sin \frac{\pi\gamma}{6})$,
then $f\in \mathcal{R}_\gamma$. In particular, if $f''(0)=0$, then
$f\in\es^*(1/2)$ implies that ${\rm Re}\,f'(z)\geq 0$.
\end{corollary}

\begin{corollary}\label{psvth9-cor3}
If $0\leq \gamma \leq  1$ and $f(z)=z-\sum_{n=3}^\infty
|a_n|z^n\in\es^*(1-\sin
\frac{\pi\gamma}{4})$, then $f\in \es_\gamma$. In particular, if
$f''(0)=0$, then $f\in\es^*(1/2)$ implies that $|\arg \left (
zf'(z)/f(z) \right )| \leq \pi/3$.
\end{corollary}

\section{Conclusion}\label{chap5-sec6}

To present a meromorphic analog of the class $\mathcal{U}(\lambda )$, we recall,
for example, the following result.

\begin{lemma}\cite[Theorem~1.2]{PoV-05} \label{4exth4}
If $f\in {\mathcal U}(\lambda)$ and $a=|f''(0)|/2\leq 1$,
then $f\in {\es}^*(\delta)$ whenever $0\leq \lambda\leq \lambda(\delta)$,
where
\begin{equation}\label{eq19}
\lambda(\delta )=\left\{ \begin{array}{ll}
\displaystyle\frac{\sqrt{(1-2\delta)(2-a^2-2\delta)}-a(1-2\delta)}{2(1-\delta)}
& \mbox{ if }~ 0\leq \delta< \displaystyle\frac{1+a}{3+a}, \\[5mm]
\displaystyle\frac{1-\delta(1+a)}{1+\delta} &  \mbox{ if }~
\displaystyle\frac{1+a}{3+a}
\leq \delta<\frac{1}{1+a}  .
\end{array} \right.
\end{equation}
In particular,
$$f\in {\mathcal U}(\lambda),~ f''(0)=0 \Longrightarrow
f\in {\es}^* ~\mbox{ whenever $0\leq \lambda\leq 1/\sqrt{2}$}.
$$
\end{lemma}

%After the paper was submitted to the journal,
Fournier and Ponnusamy \cite{FR-06}  settled the
question of sharpness of the bound for $\lambda$ for which
${\mathcal U}(\lambda )\subset {\es}^*$.
As a motivation for our next result, we consider the class,
denoted by $\Sigma $, of all functions of the form
$$F(\zeta )=\zeta +\sum_{n=0}^{\infty} c_n\zeta ^{-n}
$$
that are analytic and univalent for $|\zeta |>1$. Thus
$$F\in \Sigma \Llra f\in \mathcal{S}, \quad f(z) =\frac{1}{F(1/z)}
=\frac{z}{1+\sum_{n=1}^\infty c_{n-1}z^n} .
$$
Also, we note that
$$
f'(z)\left(\frac{z}{f(z)} \right)^{2} =F'(1/z)~\mbox{ and }~
 \frac{zf'(z)}{f(z)} =\frac{(1/z)F'(1/z)}{F(1/z)}.
$$
Consequently, for $0\leq \lambda \leq 1$, $f\in {\mathcal U}(\lambda )$ if and
only if $|F'(\zeta )-1|\leq \lambda$ for $|\zeta|>1$. Similarly, for $0\leq \alpha
\leq 1$, $f\in {\es}^*(\alpha )$ if and only if
$${\rm Re}\,\left(\frac{\zeta F'(\zeta )}{F(\zeta )}\right)\geq \alpha
~\mbox{ for $|\zeta|>1$} .
$$
The class of all such functions satisfying the later condition is denoted by
${\Sigma}^*(\alpha )$. Thus, Lemma \ref{4exth4} takes the following form:

\begin{theorem}
Let $F(\zeta )=\zeta +\sum_{n=0}^{\infty} c_n\zeta ^{-n}$ be analytic and
univalent for $|\zeta |>1$. If $F$ satisfies the condition
$$|F'(\zeta )-1| \leq \lambda ~\mbox{ for $|\zeta|>1$}
$$
and $a=|-c_0|\leq 1$, then $F\in {\Sigma }^*(\delta)$ whenever $0<\lambda\leq
\lambda(\delta)$, where $\lambda(\delta )$ is given by $(\ref{eq19})$. In
particular, for $c_0=0$, $F\in {\Sigma }^*(\delta)$ whenever $0<\lambda\leq
1/\sqrt{2}$.
\end{theorem}

%In view of the above observations, we can state a number of results
%for various subclasses of the class of meromorphic univalent functions.
This result may be used to generate a number of results
for various subclasses of the class of meromorphic univalent functions.

\chapter{NORM ESTIMATES OF CERTAIN ANALYTIC FUNCTIONS}\label{chap6}
This chapter is devoted to the study of pre-Schwarzian norm estimates
of certain subclasses of analytic functions.
Section \ref{chap6-sec1} consists of definitions and preliminary results.
In Section \ref{chap6-sec2}, we collect some results
to prove our main theorems.
In Section \ref{chap6-sec3}, we state and prove our main results and some of
their consequences. Finally, Section \ref{chap6-sec4}
concludes with a number of open problems.

\vskip 1cm
Most of the results in this chapter are from the articles:
{\bf S. Ponnusamy and S.K. Sahoo} (2008)
Norm estimates for convolution transforms of certain classes
of analytic functions. {\em J. Math. Anal. Appl.} {\bf 342}, 171--180\\[2mm]
and\\[2mm]
{\bf R. Parvatham, S. Ponnusamy and S.K. Sahoo} (2008)
Norm estimate for the Bernardi integral transforms of functions
defined by subordination.
{\em Hiroshima Math. J.} {\bf 38}, 19--29.

\section{Introduction}\label{chap6-sec1}

We refer to Chapter \ref{chap1} for related definitions and notations
used in this chapter.
First we recall the subclass ${\mathcal F}_\beta$ of
$\A$ defined by
$$\F_\beta=\left\{f\in\A:\,
\real\left(1+\frac{zf''(z)}{f'(z)}\right)<\frac{3}{2}\beta,
\quad z\in\D\right\}
$$
for some  $\beta>\frac{2}{3}$.
%where the range for $\beta$ is $\frac{2}{3}<\beta \le 1$.

We also consider the subclasses $\es^*(A,B)$ and $\K(A,B)$ of
$\A$ defined by (see Janowski \cite{Jan73})
$$\es^*(A,B)=\left\{f\in\A:\,\frac{zf'(z)}{f(z)}\prec\frac{1+Az}{1+Bz}\right\}
$$
and
$$\K(A,B)=\left\{f\in\A:\,1+\frac{zf''(z)}{f'(z)}\prec
\frac{1+Az}{1+Bz}\right\}.
$$
Here we assume that $-1\leq B<A\leq 1$, but a relaxed
restriction on $A,B$ will be used in the last section.
%for $\es^*(A,B)$ and $\K(A,B)$ in Theorem \ref{thm:main}
%as well as in the last section.
These classes
are widely used in the literature. %by (see Janowski \cite{Jan73})
% $$\es^*(A,B)=\left\{f\in\es^*:\,\frac{zf'(z)}{f(z)}\prec\frac{1+Az}{1+Bz},
% \quad z\in \D \right\}
% $$
% and
% $$\K(A,B)=\left\{f\in\es^* :\,1+\frac{zf''(z)}{f'(z)}\prec
% \frac{1+Az}{1+Bz},
% \quad z\in \D\right\}.
% $$
% Here we assume that $-1\leq B<A\leq 1$. These classes
% are widely used in the literature.
For $0\leq\alpha<1$, we observe that
$$\es^*(1-2\alpha,-1)=\es^*(\alpha)\quad\mbox{ and }\quad \K
(1-2\alpha,-1)=\K(\alpha).
$$
%are the classes of starlike functions of order $\alpha$ and convex functions
%of order $\alpha$, respectively.
% Thus, $\es^*(0)=\es^*$, and we set
% $\K (0)=\K$.
We note that $f\in \es^*(A,B)$ if and only if
$J[f]\in\K(A,B)$, where $J[f]$ is defined by (\ref{Alex-trans}).

In addition, we estimate the pre-Schwarzian norm
of functions from the subclass
$\es^*(\alpha,\beta)$ of
$\A$ defined by
$$\es^*(\alpha,\beta)=\left\{f\in\A:\, \frac{zf'(z)}{f(z)}\prec
h_{\alpha,\beta}(z) \equiv \left(\frac{1+(1-2\beta)z}{1-z}\right)^\alpha\right\},
$$
for $0<\alpha\le 1$ and $0\le \beta< 1$.
%Geometrically, this means that the set
%$\{w:\,w=zf'(z)/f(z)\}$ is contained in the set
%$\{w:\,{\rm Re}\,w>\beta \mbox{ and }|\arg w|<\pi\alpha/2\}$.
Since functions in $\es^*(\alpha,\beta)$ belong to  $\es^*(1,0)\equiv \es^*$,
$\es^*(\alpha,\beta)\subsetneq \es$ for $0<\alpha\le 1$ and $0\le \beta< 1$.

The class $\es^*(\alpha,\beta)$ has been studied by Weso{\l}owski
in \cite{Wes71}. With $0<\alpha\le 1$ and $0<\beta< 1$, we  have
$$ h_{\alpha,\beta}(e^{i\theta})=(\beta +i(1-\beta)\cot (\theta /2))^{\alpha}
$$
from which we easily see that the univalent function $h_{\alpha,\beta}(z)$
maps $\D$ onto a convex domain bounded by the curve given by
$$w=\left(\frac{\beta}{\cos \phi}\right)^\alpha e^{i\alpha\phi},
\quad -\pi/2<\phi<\pi/2,
$$
where $\phi$ and $\theta$ satisfy the relation $(1-\beta)\cot (\theta /2)=\beta\tan \phi$.
%(see \cite{PP01}).
In particular, functions in the class $\es\es^*(\alpha)\equiv \es^*(\alpha,0)$
are called the strongly starlike functions of order $\alpha$;
equivalently, $f\in\es\es^*(\alpha)$ if and only if $|\arg(zf'(z))/f(z)|<\pi \alpha/2$,
for $z\in\D$.
Every strongly starlike function
$f$ of order $\alpha <1$ is bounded (see \cite{BK}).
Further, this class of functions has been studied by many authors,
for example by Sugawa (see \cite{SugawaST}).

%where the constants $b$ and $c$ are related by
%$1\le b\le c$ or $0< b\le 1\le c.$
%\begin{equation}\label{beta-eq1}
%\frac{2}{3}<\beta \le 1 ~~\mbox{ and }~~ 1\le b\le c
%\end{equation}
%or
%\begin{equation}\label{beta-eq2}
%\frac{2}{3}<\beta \le 1 ~~\mbox{ and }~~ 0< b\le 1\le c.
%\end{equation}

%\begin{cor}
%For $f\in\F$, the Libera transform $L[f]$ of $f$ satisfies
%the sharp inequality $\|L[f]\|\leq \frac{1}{2}(\sqrt 5-1)^2$.
%\end{cor}

%In order to state our next result,

\section{Preparatory Results}\label{chap6-sec2}

In this section, we collect some known results on starlikeness of
hypergeometric functions and as a consequence we also obtain
a useful result that deals with the starlikeness of
the derivative of hypergeometric functions. We also
need an invariance property of subordination in terms of
convolution of convex functions.

% For convenience, we will use the terminology ``starlike" and ``convex" in a
% broader sense in what follows.  A function $f\in\hol$ is called starlike
% (respectively convex) if $f$ is univalent and if the image $f(\D)$ is
% starlike with respect to $f(0)$ (respectively convex). As is well known, $f$
% is starlike (respectively convex) if and only if $zf'(z)/(f(z)-f(0))$
% (respectively $1+zf''(z)/f'(z)$) has positive real part.  In particular,
% $f\in\hol$ is convex if and only if $zf'(z)$ is starlike (with respect to the
% origin).

%The following result can be obtained using Theorem 8.9 in
%\cite[p.~254]{Duren:univ}.

The following result is a reformulated version of
Ma and Minda \cite[Theorem 1]{MM92B} (see also \cite{KS2}).

\begin{lemma}\label{lem:mm}
Let $\psi\in\hol_1$ be starlike and suppose that $g\in\A$ satisfies the
equation
$$1+\frac{zg''(z)}{g'(z)}=\psi(z), \ \ z\in\D .
$$
Then for $f\in\A$, the condition $1+zf''(z)/f'(z)\prec\psi(z)$
implies $f'(z)\prec g'(z).$
\end{lemma}

Recall that in Lemma \ref{lem:mm}, the notation $\hol_1$
is used for the class of analytic functions which take
origin into $1$.

From the theory of prestarlike functions (see \cite[p. 61]{Rus}
and \cite[Theorem B]{RS86}), one obtains the following
starlikeness criterion for hypergeometric functions.

%. The following
%result is implied by the proof of \cite[Theorem]{MS61} (see \cite[Theorem B]{RS86}).

\begin{lemma}\label{lem:lewis}
Let $a,b,c$ be real numbers with $0\le a\le b\le c.$ Then the function $z
F(a,b;c;z)$ is starlike of order $1-a/2.$
\end{lemma}

Starlikeness of functions in the form $z F(a,b;c;z)$ has also been studied by
many other authors (see, for example, \cite{Kus02,PV01} and the references
therein).
%Taking into account the relation $F(a,b;c;z)=F(b,a;c;z)$, we
%see that a consequence of Lemma \ref{lem:lewis} gives the following useful
%result that establishes the convexity of the hypergeometric function
%$F(1,b;c;z)$.

\begin{corollary}\label{cvx-of-hyp}
Suppose that the real numbers $b$ and $c$ are related by $1\le b\le c$
or $-1< b\le 1\le c$.
Then  $zF'(1,b;c;z)$ is starlike and hence $F(1,b;c;z)$ is convex.
\end{corollary}
\begin{proof} We have
$$zF'(1,b;c;z)=\frac{b}{c} zF(2,b+1;c+1;z)=\frac{b}{c} zF(b+1,2;c+1;z).
$$
The desired conclusion follows if we apply Lemma \ref{lem:lewis}
to the two expressions on the right of the last equality.
\end{proof}

%We also need
The following result is due to Ruscheweyh \cite[Theorem 2.36, p. 86]{Rus}
(see also \cite[Theorem 8.9, p. 254]{Duren:univ}):

\begin{lemma}\label{Rus}
Let $f\in \hol$ and $g$ be a convex function such that $f\prec g$.
Then for all convex functions $h$, we have $h*f\prec h*g$.
\end{lemma}

We also need the following integral representation of quotient of
two hypergeometric functions which is due to K\"{u}stner \cite[Theorem 1.5]{Kus02}
(see also \cite[Lemma 7]{CKPS04}).

\begin{lemma}\label{lem:pos}
Suppose that $a,b,c\in\R$ satisfy $-1\le a\le c$ and $0<b\le c.$ Then there
exists a Borel probability measure $\mu$ on the interval $[0,1]$ such that
$$\frac{F(a+1,b+1;c+1;z)}{F(a,b;c;z)}=\int_0^1\frac{d\mu(t)}{1-tz},
\quad z\in\D.
$$
\end{lemma}

%\vspace{10pt}
\section{Pre-Schwarzian Norm Estimates}\label{chap6-sec3} %\ref{thm1} and \ref{thm:main1}}

In this section we mainly concentrate in estimating the
pre-Schwarzian norm of functions and that of
the transforms $B_{b,c}[f]$ of functions  $f$ from the subclasses
defined in Section \ref{chap6-sec1}. We also present
some consequences in terms of quasidisks.

In order to discuss norm estimates for the class $\F_\beta$,
for $2/3<\beta \le 1$, $b>0$ and $c>0$, we define
$$L(\beta, b,c)=\frac{b}{c}(3\beta -2)\sup_{0\le x<1}(1-x^2)
\frac{F(3-3\beta,b+1;c+1;x)}{F(2-3\beta,b;c;x)}.
$$

Here we present

\begin{theorem}\label{thm1}
Let $2/3<\beta \le 1$ and $f\in\F_\beta$. Then
$\|f\|\le 2(3\beta -2)$. If moreover
$1\le b\le c$ or  $0< b\le 1\le c$,
then $\|B_{b,c} [f]\|\leq L(\beta,b,c)$.
The bounds in both cases are sharp and the quantity $L(\beta,b,c)$
is bounded above by $2(3\beta -2)b/c$.
\end{theorem}

% \noindent
% {\bf Proof of Theorem \ref{thm1}.}
\begin{proof}
Let $f\in\F_\beta$. Then we have
$$1+\frac{zf''(z)}{f'(z)}\prec
\frac{1+(1-3\beta)z}{1-z}=\phi(z),\quad z\in\D,
$$
where $\phi$ is clearly a convex function and therefore starlike.
Let $g\in\A$ be such that
$$1+\frac{zg''(z)}{g'(z)}=\frac{1+(1-3\beta)z}{1-z},\quad z\in\D.
$$
A simple computation shows that
$$g'(z)=(1-z)^{3\beta -2}=F(1,2-3\beta;1;z)
$$
so that
\begin{equation}\label{eq4}
g(z)=\frac{1-(1-z)^{3\beta -1}}{3\beta -1}.
\end{equation}
By Lemma \ref{lem:mm}, we conclude that
\begin{equation}\label{ABbc-eq3}
f'(z)\prec g'(z)=(1-z)^{3\beta -2}, \quad z\in \D,
\end{equation}
which, by the definition of subordination, implies that
$$f'(z)=(1-w(z))^{3\beta -2}
$$
for some Schwarz function $w(z)$, i.e. $w:\D \to \D$ is analytic
with $w(0)=0$. By Schwarz-Pick lemma we get
$$|w'(z)|\leq \frac{1-|w(z)|^2}{1-|z|^2}, \quad z\in \D ,
$$
and hence,
\begin{eqnarray*}
\left|\frac{f''(z)}{f'(z)}\right| & = &(3\beta -2)\left |\frac{w'(z)}{1-w(z)}\right |\\
& \leq & (3\beta -2)\frac{1-|w(z)|^2}{1-|z|^2}\frac{1}{1-|w(z)|}\\
& = & (3\beta -2) \frac{1+|w(z)|}{1-|z|^2}
\end{eqnarray*}
which gives that $\|f\|\le 2(3\beta -2)$ and the
equality holds for the function $g\in \F_\beta$ defined in (\ref{eq4}). Indeed,
we compute that
$$\|g\|=(3\beta -2)\sup_{|z|<1}\frac{1-|z|^2}{|1-z|}=2(3\beta -2).
$$

We now proceed to prove the second part.
%We have already seen that
%\begin{equation}\label{ABbc-eq3}
%f'(z)\prec g'(z), \quad  z\in\D.
%\end{equation}
By Corollary \ref{cvx-of-hyp}, we observe that $g'(z)$ is convex in $\D$,
since $2/3<\beta \le 1$.
Furthermore, by Corollary \ref{cvx-of-hyp}, it follows that if $b$ and $c$ are
related by $1\le b\le c$ or $-1< b\le 1\le c$ (which holds by the hypothesis
of the theorem), then the hypergeometric function $F(1,b;c;z)$ is convex.
In view of (\ref{ABbc-eq3}) and Lemma \ref{Rus}, we also have
$$
F(1,b;c;z)*f'(z) \prec  F(1,b;c;z)*g'(z), \quad \mbox{i.e. }~
(B_{b,c}[f])'(z) \prec (B_{b,c}[g])'(z).
$$
We see that (see the proof of Proposition \ref{prop})
$\|B_{b,c}[f]\|\le \|B_{b,c}[g]\|$ holds.
So it remains to compute the norm $\|B_{b,c}[g]\|$.

By the definition of Hadamard product we have
$$(J_{b,c}[g])'(z) =  F(1,b;c;z) * F(1,2-3\beta;1;z) =  F(2-3\beta,b;c;z).
$$
In view of the representation $(B_{b,c}[g])'(z)=F(2-3\beta,b;c;z)$ and
Lemma \ref{lem:pos}, we deduce that
%$$\left|\frac{(B_{b,c}[g])''(z)}{(B_{b,c}[g])'(z)}\right|
%\le \frac{(B_{b,c}[g])''(|z|)}{(B_{b,c}[g])'(|z|)},\quad z\in\D,
%$$
there exists a Borel probability measure $\mu$ on the interval $[0,1]$ such that
\begin{eqnarray*}
\frac{(B_{b,c}[g])''(z)}{(B_{b,c}[g])'(z)}
& = & \frac{b}{c}(2-3\beta )\frac{F(3-3\beta,b+1;c+1;z)}{F(2-3\beta,b;c;z)}\\
& = & \frac{b}{c}(2-3\beta )\int_0^1\frac{d\mu(t)}{1-tz}, \quad z\in\D,
\end{eqnarray*}
whenever $0<b\le c$ and $\frac{2-c}{3}\le \beta\le 1$ (and so is by the
hypothesis). The above formulation clearly shows that
\begin{eqnarray*}
\|B_{b,c}[g]\| & = & \sup_{|z|<1}(1-|z|^2)
\left|\frac{(B_{b,c}[g])''(z)}{(B_{b,c}[g])'(z)}\right|\\
& = & \sup_{0\le x<1} (1-x^2)\frac{(B_{b,c}[g])''(x)}{(B_{b,c}[g])'(x)}\\
& = & \frac{b(3\beta -2)}{c}\sup_{0\le x<1} (1-x^2)
\frac{F(3-3\beta,b+1;c+1;x)}{F(2-3\beta,b;c;x)}\\
& = & L(\beta, b,c).
\end{eqnarray*}
%Since the equality $\|B_{b,c}[g]\|=L(\beta, b,c)$ holds for the
%function $g\in \F_\beta$ defined by (\ref{eq4}), we have the sharp
Thus, we have the sharp inequality $\|B_{b,c}[f]\|\le L(\beta, b,c)$.
%After
%looking all the conditions on $\beta, b,c$ discussed in the proof, we
%arrive that it is sufficient to assume that the constants
%$\beta, b,c$ either satisfy (\ref{beta-eq1}) or (\ref{beta-eq2}).
To obtain an upper bound for the quantity $L(\beta,b,c)$, it suffices to
observe that
$$(1-x^2)\frac{F(3-3\beta,b+1;c+1;x)}{F(2-3\beta,b;c;x)}
=\int_0^1\frac{1-x^2}{1-tx} \,d\mu(t) \leq \int_0^1(1+x) \,d\mu(t) \leq 2
$$
%and, since $(1-x^2)(1-tx)\le 1+x$ for $0\le t\le 1$, get the inequality
%$$(1-x^2)\frac{F(3-3\beta,b+1;c+1;x)}{F(2-3\beta,b;c;x)}\le 1+x
%$$
%for $0\le x<1$. This gives that
which shows that
$$L(\beta, b,c)\le \frac{2b(3\beta-2)}{c}.
$$
We thus completed our proof.
\end{proof}

Recall that a quasidisk\index{quasidisk} is the image of a disk under a
quasiconformal self map of $\Cbar$.
In 1984 the following theorem was proved by Becker and Pommerenke \cite{BP84}
(see also \cite[Theorem 4.2]{Hag01}).
\begin{theorem}\label{Becker-Pom}
If $\|f\|< 1$, then $f(\D)$ is a quasidisk.
\end{theorem}

As a consequence of Theorem \ref{thm1}, by using Theorem \ref{Becker-Pom},
we obtain the following.
\begin{corollary}
For $f\in \F_\beta$, we obtain that
$f(\D)$ is a quasidisk if  $2/3<\beta <5/6$.
\end{corollary}

In the case $\beta =1$, Theorem \ref{thm1} takes the following simple
form.

\begin{corollary}\label{cor1}
Suppose that $1\le b\le c$ or $0<b\le 1\le c$ holds.
If $f\in\F$, then we have $\|f\|\le 2$ and
%the convolution operator $B_{b,c}[f]$ of $f\in \F$ satisfies the
%inequality
$$\|B_{b,c} [f]\|\leq L(1,b,c)=\frac{2(c-\sqrt{c^2-b^2})}{b}.
$$
The bounds are sharp.
\end{corollary}
\begin{proof}
From Theorem \ref{thm1}, we see that
$$L(1, b,c)=\frac{b}{c}\sup_{0\le x<1}(1-x^2)
\frac{F(0,b+1;c+1;x)}{F(-1,b;c;x)} =
\frac{b}{c}\sup_{0\le x<1}\frac{1-x^2}{1-(b/c)x}.
$$
For $b=c>0$ the conclusion is obvious. For $b/c<1$ $(b>0,c>0)$, it is a simple
exercise to see that the function $h(x)=(1-x^2)/(1-(b/c)x)$ defined on $[0,1)$
attains its maximum at
$$x_0=\frac{c-\sqrt{c^2-b^2}}{b}
$$
so that
$$h(x)\leq h(x_0)=\frac{2c(c-\sqrt{c^2-b^2})}{b^2}
$$
and the conclusion follows.
\end{proof}

Here we see that
$$\frac{2(c-\sqrt{c^2-b^2})}{b}\le 1 \iff b\le \frac{4}{5}c.
$$
Thus, as a consequence of Corollary \ref{cor1} we obtain the following
by using a result of Becker \cite{Be72} and Theorem \ref{Becker-Pom}.
\begin{corollary}
Let $f$ be in $\F$. Then $J_{b,c}[f]$ is univalent if $1\le b\le \frac{4}{5}c$
and $J_{b,c}[f](\D)$ is a quasidisk if $1\le b< \frac{4}{5}c$.
\end{corollary}

As a consequence of Corollary \ref{cor1}, we easily have
%obtain the norm estimates of the Bernardi
%integral operator $B_{\gamma}[f]$ of $f\in\F$ in the following form:

\begin{corollary}\label{cor2}
Let $\gamma>-1$, and $B_\gamma [f]$ be the Bernardi transform  of $f\in\F$.
Then, we have
$$\|B_\gamma [f]\|\leq  \frac{2(\gamma +2-\sqrt{3+2\gamma })}{\gamma +1}
$$
and the bound is sharp.
\end{corollary}

Here we compute that
$$\frac{2(\gamma +2-\sqrt{3+2\gamma })}{\gamma +1}\le 1 \iff \gamma\le 3,
$$
because $\gamma+1>0$. Thus, as a consequence of Corollary \ref{cor2} we
have the following result.
\begin{corollary}
Let $f$ be in $\F$. Then $B_{\gamma}[f]$ is univalent for $-1<\gamma\le 3$
and $B_{\gamma}[f](\D)$ is a quasidisk for $-1<\gamma<3$.
\end{corollary}

%Here we remark that Corollary \ref{cor2} reduces to the following
%if one chooses $\gamma=0$ and $\gamma=1$ respectively:
%For  $\gamma=0$ and $\gamma=1$, Corollary \ref{cor2} gives the following
%special results which need a special mention.

Setting $\gamma=0$ and $\gamma=1$ respectively. Thus, we have

\begin{corollary}\label{cor3}
%Suppose that $J[f]$ and $L[f]$ are respectively the Alexander transform
%and the Libera transform of $f\in\F$. Then we have
%the sharp inequalities $\|J[f]\|\leq 4-2\sqrt{3}$
%and $\|L[f]\|\leq  3-\sqrt{5}$.
Let $f\in\F$. Then we have
$\|J[f]\|\leq 4-2\sqrt{3}$ and $\|L[f]\|\leq  3-\sqrt{5}$.
The bounds are sharp.
\end{corollary}

Combining Theorem \ref{Becker-Pom} and Corollary \ref{cor3},
we obtain the following.
\begin{corollary}
If $f\in\F$, then the images $J[f](\D)$ and $L[f](\D)$ are quasidisks.
\end{corollary}

The class $\F$ is particularly interesting because of the inclusion
$\F\subset \es^*\subset \es$. On the other hand,
if $f\in \es^*$, then $\|f\|\le 6$ and $\|J[f]\|\le 4$.
Both the bounds here are sharp and was proved by S. Yamashita \cite{Yam97}
(see also \cite[Theorem A]{CKPS04}). Later from Corollary \ref{theoA}, we
see that if $f\in \K $, then $\|f\|\le 4$, $\|J[f]\|\le 2$
and $\|L[f]\|\leq  8/3$. All these bounds are sharp.

Corresponding to the class $\K(A,B)$,
$-1\leq B<A\le 1$, we introduce $N(A,B)$
\begin{equation}\label{NAB}
N(A,B):=\left\{
\begin{array}{ll}
2(A-B)\left[\ds\frac{1-\sqrt{1-B^2}}
{B^2}\right] & \mbox{ for $B\neq 0$,}\\
A & \mbox{ for $B=0$}.
\end{array}\right.
\end{equation}
To state our next theorem, we also need to define another quantity $M(A,B,b,c)$ by
\begin{equation}\label{ABbc-eq4}
M(A,B,b,c):=\frac{b(A-B)}{c}\sup_{0\le x<1} (1-x^2)
\frac{F(2-A/B,b+1;c+1;|B|x)}{F(1-A/B,b;c;|B|x)}
\end{equation}
where $A,B,b,c$ are related by
\begin{equation}\label{ABbc-eq1}
-1\le B<A\le \min\{1,B+1\},~ B\neq 0, ~1\le b\le c, ~\mbox{ and }
-2 \le -A/B\le c-1
\end{equation}
or
\begin{equation}\label{ABbc-eq2}
-1\le B<A\le \min\{1,B+1\},~ B\neq 0, ~0<b\le 1\le c,  ~\mbox{ and }
-2 \le -A/B\le c-1.
\end{equation}

\begin{theorem}\label{thm:main1}
Let $-1\le B<A\le 1$ and $f\in\K(A,B)$. Then
$\|f\|\le N(A,B)$. If moreover the real constants $A,B,b,c$ are related by $(\ref{ABbc-eq1})$ or $(\ref{ABbc-eq2})$, then $\|B_{b,c}[f]\|\le M(A,B,b,c)$.
The bounds are sharp and the quantity $M(A,B,b,c)$
is bounded from above by $\frac{b}{c}(1+|B|)(A-B)$.
\end{theorem}

\begin{proof}
The proof is similar to that of Theorem \ref{thm1}. Suppose that $f\in\K(A,B)$.
In terms of subordination, $f$ can be characterized by
$$1+\frac{zf''(z)}{f'(z)}\prec \frac{1+Az}{1+Bz}=\phi_{A,B}(z),\quad z\in\D,
$$
where $\phi_{A,B}$ is known to be a convex function and therefore starlike.
Define $g\in\A$ by the relation
\begin{equation}\label{d-relation}
1+\frac{zg''(z)}{g'(z)}=\frac{1+Az}{1+Bz},\quad z\in\D.
\end{equation}
By Lemma \ref{lem:mm}, we have
\begin{equation}\label{gprime-eq}
f'(z)\prec g'(z)=\begin{cases}
(1+Bz)^{(A/B)-1}&\quad\text{if}~B\neq 0,\\
e^{Az}&\quad\text{if}~B=0.
\end{cases}
\end{equation}
If $B=0$, then we see that $f'(z)\prec e^{Az}$ for $0<|A|\leq 1$ and so, by the definition of
subordination, we have $f'(z)=e^{Aw(z)}$ for some Schwarz function
$w(z)$. By Schwarz-Pick lemma we obtain
$$(1-|z|^2)\left|\frac{f''(z)}{f'(z)}\right|\le |A|(1-|w(z)|^2), \quad z\in\D ,
$$
and hence, for $B=0$ and $0<|A|\leq 1$, we finally get $\|f\|\le |A|.$
The estimate is sharp for the function $f(z)=(e^{Az}-1)/A$.

On the other hand, if $0\neq B$ and $-1\le B<A\le 1$, then by the same process we see that
$$(1-|z|^2)\left|\frac{f''(z)}{f'(z)}\right|\le
\frac{(A-B)(1-|w(z)|^2)}{1-|B|\,|w(z)|}
$$
for some Schwarz function $w(z)$ and hence we obtain
$$\|f\|\le (A-B)\sup_{0\le x<1}\frac{1-x^2}{1-|B|x}
=2(A-B)\left[\ds\frac{1-\sqrt{1-B^2}}{B^2}\right].
%=(A-B)\left[\ds\frac{B^2-(1-\sqrt{1-B^2})^2}{B^2\sqrt{1-B^2}}\right].
$$
Thus, for $-1\le B<A\le 1$, we formulate the pre-Schwarzian norm estimates of
the functions $f\in \K(A,B)$ by $\|f\| \le  N(A,B)$, where $N(A,B)$
is defined by (\ref{NAB}).
% \begin{eqnarray*}
% \|f\| & \le &  N(A,B):=
% \left\{\begin{array}{ll}
% 2(A-B)\left[\ds\frac{1-\sqrt{1-B^2}}{B^2}\right] & \mbox{ for $B\neq 0$,}\\
% A & \mbox{ for $B=0$.}
% \end{array}\right.
% \end{eqnarray*}

Our next task is to show that
$$\|B_{b,c}[f]\|\le \|B_{b,c}[g]\|.
$$
To do this, we first observe the fact that $f'(z)\prec g'(z)$ in $\D$. The convexity
of $g'(z)$ is easy when $A\le B+1\neq 1$. Indeed, set $h=g'$.
By the
defining relation (\ref{d-relation}) we then have
$$\frac{h'(z)}{h(z)}=\frac{A-B}{1+Bz}.
$$
Taking the logarithmic derivative of both sides and multiplying
with $z$, we obtain
$$\frac{zh''(z)}{h'(z)}-\frac{zh'(z)}{h(z)}=-\frac{Bz}{1+Bz}.
$$
Therefore,
$$1+\frac{zh''(z)}{h'(z)}=1+\frac{zh'(z)}{h(z)}-\frac{Bz}{1+Bz}
=\frac{1+Az}{1+Bz}-\frac{Bz}{1+Bz}=\frac{1+(A-B)z}{1+Bz}.
$$
We write
$$S(z)=\frac{1+(A-B)z}{1+Bz},\quad z\in\D.
$$
Since the M\"obius transformation $S(z)$ has no pole in the unit disk $\D$,
the image $S(\D)$ is the disk centered at $\frac{1-B(A-B)}{1-B^2}$ and
radius $\frac{A-2B}{1-B^2}$. Clearly the points $S(-1)$ and $S(1)$
are diametrically opposite points to this disk. Therefore,
$h(z)$ is convex (equivalently, $S(z)=1+zh''(z)/h'(z)$ has
a positive real part) if and only if $S(-1)\ge 0$ and $S(1)\ge 0$.
The last condition is equivalent to $A\le B+1$.
This shows that $g'(z)$ is convex for $A\le B+1\neq 1$.

Also, Corollary \ref{cvx-of-hyp} says that if $b$ and $c$ are
related by $1\le b\le c$ or  $-1<b\le 1\le c$, then
$F(1,b;c;z)$ is convex. Consequently, as in the proof of Theorem \ref{thm1},
Lemma \ref{Rus} gives
$$(B_{b,c}[f])'(z)=F(1,b;c;z)*f'(z) \prec (B_{b,c}[g])'(z) =F(1,b;c;z)*g'(z)
$$
%\begin{eqnarray*}
%(B_{b,c}[f])'(z) & = & F(1,b;c;z)*f'(z)\\
%& \prec & F(1,b;c;z)*g'(z)\\
%& = & (B_{b,c}[g])'(z),
%\end{eqnarray*}
whenever $A\le B+1$ and $b,c$ satisfy by $1\le b\le c$ or
$-1<b\le 1\le c$. Thus,
$$\|B_{b,c}[f]\|\le \|B_{b,c}[g]\|
$$
holds.

%Thus, as before, we establish that $\|B_{b,c}[f]\|\le \|B_{b,c}[g]\|$.
%Then there exist
%a $w\in \hol_0$ with $|w(z)|<1$ such that $(B_{b,c}[f])'(z)=(B_{b,c}[g])'(w(z))$
%in $\D$. Since a logarithmic
%differentiation yields
%$$\frac{(B_{b,c}[f])''(z)}{(B_{b,c}[f])'(z)}=
%\frac{((B_{b,c}[g])''(z)}{(B_{b,c}[g])'(z))} \circ w.w',
%$$
%we compute
%\begin{eqnarray*}
%(1-|z|^2)\left|\frac{(B_{b,c}[f])''(z)}{(B_{b,c}[f])'(z)}\right|
%& = & (1-|z|^2)|w'(z)|\left|\frac{(B_{b,c}[g])''(w(z))}
%{(B_{b,c}[g])'(w(z))}\right|\\
%& \le & (1-|w(z)|^2)\left|\frac{(B_{b,c}[g])''(w(z))}
%{(B_{b,c}[g])'(w(z))}\right|\\
% & \le & \|B_{b,c}[g]\|.
%\end{eqnarray*}
%Therefore,  we have obtained the inequality $\|B_{b,c}[f]\|\le \|B_{b,c}[g]\|$.
Finally, it remains to compute the norm $\|B_{b,c}[g]\|$ for $B\neq 0$.
Since
$$(1+Bz)^{(A/B)-1}=F(1,1-A/B;1;-Bz) ~\mbox{ for $B\neq 0$,}
$$
%\begin{equation}\label{gprime-eq}
%g'(z)=(1+Bz)^{(A/B)-1}=F(1,1-A/B;1;-Bz)
%\end{equation}
%Now differentiation(s) of the convolution operator (\ref{conv-op}),
%but for the function $g$ instead of the function $f$, give
%and so
it follows from the definition of the hypergeometric function that
$$(B_{b,c}[g])'(z)  =  F(1,b;c;z) *F(1,1-A/B;1;-Bz)=  F(1-A/B,b;c;-Bz)
$$
%\begin{eqnarray*}
%(B_{b,c}[g])'(z) & = & F(1,b;c;z)*g'(z)\\
%& = & F(1,b;c;z) *F(1,1-A/B;1;-Bz)\\
%& = & F(1-A/B,b;c;-Bz)
%\end{eqnarray*}
and so we can write
%$$(B_{b,c}[g])''(z)=\frac{b(A-B)}{c}F(2-A/B,b+1;c+1;-Bz).
%$$
%Thus,
$$\frac{(B_{b,c}[g])''(z)}{(B_{b,c}[g])'(z)}=\frac{b(A-B)}{c}
\frac{F(2-A/B,b+1;c+1;-Bz)}{F(1-A/B,b;c;-Bz)}.
$$
%If $-1\le B<0$, then by Lemma \ref{lem:pos} the inequality
%$$\left|\frac{(B_{b,c}[g])''(z)}{(B_{b,c}[g])'(z)}\right|
%\le \frac{(B_{b,c}[g])''(|z|)}{(B_{b,c}[g])'(|z|)},\quad z\in\D,
%$$
%holds for $0<b\le c,~B<A$ and $-2\le -A/B\le c-1$. Now we have
%\begin{eqnarray*}
%\|B_{b,c}[g]\| & = & \sup_{|z|<1}(1-|z|^2)
%\left|\frac{(B_{b,c}[g])''(z)}{(B_{b,c}[g])'(z)}\right|\\
%& = & \sup_{0\le x<1} (1-x^2)\frac{(B_{b,c}[g])''(x)}{(B_{b,c}[g])'(x)}\\
%& = & \frac{b(A-B)}{c}\sup_{0\le x<1} (1-x^2)
%\frac{F(2-A/B,b+1;c+1;-Bx)}{F(1-A/B,b;c;-Bx)}.
%\end{eqnarray*}
%If $0<B\le 1$, we similarly have
%$$\|B_{b,c}[g]\|=\frac{b(A-B)}{c}\sup_{0\le x<1} (1-x^2)
%\frac{F(2-A/B,b+1;c+1;Bx)}{F(1-A/B,b;c;Bx)}.
%$$
If $0<|B|\le 1$, then by Lemma \ref{lem:pos} we can easily obtain
\begin{eqnarray*}
\|B_{b,c}[g]\|
%& = & \frac{b(A-B)}{c}\sup_{0\le x<1} (1-x^2)
%\frac{F(2-A/B,b+1;c+1;|B|x)}{F(1-A/B,b;c;|B|x)}\\
& = & M(A,B,b,c)
\end{eqnarray*}
whenever  $0<b\le c,~B<A$ and $-2\le -A/B\le c-1$, where $M(A,B,b,c)$ is defined
by (\ref{ABbc-eq4}).
% It is easy to see that $M(A,B,b,c)=\|B_{b,c}[g]\|$ for
%$g\in\K(A,B)$ defined by (\ref{gprime-eq}).
This proves the sharpness of the
norm estimate of $\|B_{b,c}[f]\|$ whenever (\ref{ABbc-eq1}) or (\ref{ABbc-eq2})
holds.

%After looking all the conditions on $A,B,b,c$ discussed in the proof, we
%arrive that it is sufficient to assume that they
%either satisfy (\ref{ABbc-eq1}) or (\ref{ABbc-eq2}).

Finally, we establish an upper bound for the quantity $M(A,B,b,c)$.
Again, using Lemma \ref{lem:pos}, we may express
$$(1-x^2)\frac{F(2-A/B,b+1;c+1;|B|x)}{F(1-A/B,b;c;|B|x)}
 =\int_0^1\frac{1-x^2}{1-t|B|x}\,d\mu(t)
$$
for some Borel probability measure $\mu$ on the interval $[0,1]$ and under the
hypotheses on the constants $A,B,b,c$.  Since
$$\frac{1-x^2}{1-t|B|x}\leq\frac{1-|B|^2x^2}{1-|B|x}=1+|B|x\leq 1+|B|
\quad\mbox{ for } 0\leq t\leq 1,
$$
the inequality
$$(1-x^2)\frac{F(2-A/B,b+1;c+1;|B|x)}{F(1-A/B,b;c;|B|x)}\le
1+|B|x\leq 1+|B|
$$
holds for $0\le x<1$. This gives that
$$M(A,B,b,c)\le \frac{b}{c}(1+|B|)(A-B)
$$
and we complete the proof.
\end{proof}

\begin{remark}
%We note that Theorem \ref{thm:main1} reduces to Theorem A
%if one substitutes $c=b+1$ and $b=\gamma +1$.
In the proof of Theorem \ref{thm:main1}, we have established
the pre-Schwarzian norm estimate of $f\in\K(A,0)$ although
this is not stated in the statement. However, we do not have an
answer in finding norm estimate for $B_{b,c}[f]$ when $f\in\K(A,0)$.
\end{remark}

If one chooses $c=b+1=\gamma +2$, then we obtain that
\begin{eqnarray*}%\label{Bernardi-bound}
D(A,B,\gamma )
&:=& (A-B)\left(\frac{\gamma +1}{\gamma +2}\right)\sup_{0\leq x<1}
(1-x^2)\frac{F(2-A/B,\gamma +2;\gamma +3;|B|x)}{F(1-A/B,\gamma +1;\gamma
+2;|B|x)}\\
&=& M(A,B,\gamma +1,\gamma+2),
\end{eqnarray*}
where $A, B, \gamma$ are related by
\begin{equation}\label{chap5-eq1}
-1\le B<A\le \min\{1,B+1\},~B\neq 0,~-1<\gamma \mbox{ and }
-2\le -A/B\le \gamma+1.
\end{equation}

Thus, Theorem \ref{thm:main1} leads to the following result.
%of Parvatham et. el. \cite[Theorem 1]{HMJ}).

\begin{theorem}\label{theoA}
Let $A,B,\gamma $ be real constants satisfying the condition $(\ref{chap5-eq1})$.
Then for every $f\in\K(A,B)$,
the Bernardi transform $B_\gamma [f]$ of $f$ satisfies the inequality
$\|B_\gamma [f]\|\leq D(A,B,\gamma)$.
The bound
$D(A,B,\gamma )$ is sharp and satisfies
$$D(A,B,\gamma )\le \frac{(1+|B|)(A-B)(\gamma +1)}{\gamma +2}.
$$
\end{theorem}

% \begin{proof}
% Choose $c=b+1=\gamma +2$ in Theorem \ref{thm:main1}, and observe that
% $B_\gamma [f]=B_{1+\gamma ,2+\gamma }[f].$
% \end{proof}

Theorem \ref{theoA} actually extends the recent work in \cite{CKPS04}.
We remark that
$$N(1,-1)=4, ~D(1,-1,0)=2,~\mbox{ and } ~D(1,-1,1)=8/3.
$$

For the special case $B=-A$, Theorem \ref{theoA} yields the following simple result:

\begin{corollary}
Let $0<A\leq 1$  and $\gamma \ge 0$. We have then
$$1+\frac{zf''(z)}{f'(z)}\prec\frac{1+Az}{1-Az} ~ \Longrightarrow ~
\|B_\gamma [f]\|\leq D(A,-A,\gamma ).
$$
The bound $D(A,-A,\gamma )$ is sharp and satisfies
$$D(A,-A,\gamma )\le \frac{2A(1+A)(\gamma +1)}{\gamma +2}.
$$
\end{corollary}

\begin{remark}
We have proved Theorem \ref{theoA} in the article
``{\bf R. Parvatham, S. Ponnusamy and S.K. Sahoo}.
Norm estimate for the Bernardi integral transforms of functions
defined by subordination.
{\em Hiroshima Math. J.} (to appear)''
separately,
but not using the exact method that we use in the proof of Theorem \ref{thm:main1}.
Indeed, we mention that Theorem \ref{theoA} has been obtained
by proving the following proposition. Because of independent interest,
we describe the proposition in detail.
\end{remark}

Note that we have used following two lemmas. One is Lemma \ref{lem:mm} and
the second one, due to Hallenbeck and Ruscheweyh \cite{HR}, is stated below.

% \begin{lemma}\label{lem:mm}
% Let $\psi\in\hol_1$ be starlike and suppose that $g\in\A$ satisfies the
% equation
% $$1+\frac{zg''(z)}{g'(z)}=\psi(z), \ \ z\in\D .
% $$
% For $f\in\A$, the condition $1+zf''(z)/f'(z)\prec\psi(z)$
% then implies $f'(z)\prec g'(z).$
% \end{lemma}

\begin{lemma}\cite{HR}\label{lem:dur}
Let $p(z)$ and $q(z)$ be analytic functions in the unit disk
$\D$ with $p(0)=1=q(0)$. For $\alpha >0$ suppose that the
function $h(z)=q(z)+\alpha zq'(z)$ is convex. Then the condition
$p(z)+\alpha zp'(z)\prec h(z)$ implies $p(z)\prec q(z)$.
\end{lemma}

\begin{proposition}\label{prop}
Let $\gamma > -1$ be given.
Suppose that the function $\psi(z)=1+zg''(z)/g'(z)$ is starlike and that
the function $g'(z)$ is convex for a given function $g\in\A$.
If a function $f\in\A$ satisfies
$$1+\frac{zf''(z)}{f'(z)}\prec\psi(z), \ \ z\in \D
$$
then the inequalities
$\|f\|\le\|g\|$ and $\|B_\gamma [f]\|\le\|B_\gamma [g]\|$ hold.
\end{proposition}

\begin{proof}
First, by Lemma \ref{lem:mm}, the hypothesis implies that $f'(z)\prec g'(z).$
Namely, $f'(z)=(g'\circ \omega)(z)$ for some Schwarz function $\omega$.
By the Schwarz-Pick lemma, we have the inequality
$$\frac{|\omega'(z)|}{1-|\omega|^2}\le \frac{1}{1-|z|^2},\quad z\in \D.
$$
Since a logarithmic differentiation yields $f''/f'=(g''/g')\circ \omega .\omega'$,
we compute
$$(1-|z|^2)\left|\frac{f''(z)}{f'(z)}\right|
=(1-|z|^2)|\omega'(z)|\left|\frac{g''(\omega(z))}{g'(\omega(z))}\right|
\le (1-|\omega(z)|^2)|\left|\frac{g''(\omega(z))}{g'(\omega(z))}\right|
\le \|g\|.
$$
Therefore, we obtain the inequality $\|f\|\le \|g\|$.
Now we proceed to prove the inequality $\|B_\gamma [f]\|\leq\|B_\gamma [g]\|$.
It is enough to prove that $(B_\gamma [f])'(z)\prec (B_\gamma [g])'(z)$. It is
easy to see that the Bernardi transform $B_\gamma [g]$ of $g$ defined by
(\ref{bit}) satisfies the equation
$$z(B_\gamma [g])'(z)+\gamma B_\gamma [g](z)=(\gamma +1)g(z)
$$
and so,
$$z(B_\gamma [g])''(z)+(\gamma +1)(B_\gamma [g])'(z)=(\gamma +1)g'(z).
$$
In a similar fashion, we have
$$z(B_\gamma [f])''(z)+(\gamma +1)(B_\gamma [f])'(z)=(\gamma +1)f'(z).
$$
Set $p(z)=(B_\gamma [f])'(z)$ and $q(z)=(B_\gamma [g])'(z).$ Then, the
condition $f'(z)\prec g'(z)$ is equivalent to
%\begin{align*}
$$zp'(z)+(\gamma +1)p(z) =(\gamma +1)f'(z)
\prec  (\gamma +1)g'(z) =zq'(z)+(\gamma +1)q(z).
$$
%\end{align*}
This shows that
$$\frac{zp'(z)}{\gamma +1}+p(z)\prec\frac{zq'(z)}{\gamma +1}+q(z),
\quad z\in\D.
$$
Since $g'(z)$ is convex, by Lemma \ref{lem:dur}, we get
$$(B_\gamma [f])'(z)=p(z)\prec q(z)=(B_\gamma [g])'(z)
$$
for $\gamma > -1$. We thus proved the required inequality.
\end{proof}

Next, we are going to discuss the norm estimates for the class
$\es^*(\alpha,\beta)$ defined in Section \ref{chap6-sec1}.
Recall that in \cite{SugawaST}, Sugawa has presented the sharp norm
estimates for functions $f\in \es\es^*(\alpha)$. The aim of the last
result of this chapter is to generalize the result of Sugawa
\cite[Theorem 1.1]{SugawaST}. But unfortunately, we do not have
an sharp norm estimate although we have an optimal estimate in the
following form:

\begin{theorem}
Let $0<\alpha<1 \text{ and }0\le\beta<1$. If
$f\in\es^*(\alpha,\beta)$, then
% Then we have the following norm estimate for $f$:
\begin{equation}\label{p3eq1} \|f\|\le L(\alpha,\beta)+2\alpha,
\end{equation}
where
\begin{equation}\label{p3eq2}
L(\alpha,\beta)=\frac{4(1-\beta)(k-\beta)(k^\alpha
-1)}{(k-1)(k+1-2\beta)}
\end{equation}
and $k$ is the unique solution of the following equation in $x\in (1,\infty)$:
\begin{eqnarray}\label{p3eq3}
&&(1-\alpha)x^{\alpha +2}+\beta(3\alpha-2)x^{\alpha+1}
+[(1-2\beta)(1+\alpha)+2\beta^2(1-\alpha)]x^\alpha \\
\nonumber ~&&  \hspace{3cm}- \alpha\beta(1-2\beta)x^{\alpha -1}- x^2+2\beta x
=(1-\beta)^2+\beta^2.
\end{eqnarray}
% The equality in $(\ref{p3eq1})$ holds precisely when
% $$\frac{zf'(z)}{f(z)}=\left(\frac{1+(1-2\beta)\mu z}{1-\mu
% z}\right)^\alpha,
% $$
% for a unimodular constant $\mu$.}
\end{theorem}

\begin{remark}
For $\alpha=1$, it is well known that $\|f\|\le 6-4\beta$ and equality
holds if and only if $f(z)=\overline\mu\Phi(\mu z)$, where
$\Phi(z)=z/(1-z)^{2(1-\beta)}$ and $\mu$ is a unimodular constant
(see \cite{Yam97}). Moreover, if $\alpha=1$ as well as $\beta=0$,
it is known that $\|f\|\le 6$; and equality holds for the Koebe function
$k(z)=z/(1-z)^2$.
Now we shall prove the main theorem by using the method
adopted by
%Chiang \cite{Chiang91} and
Sugawa \cite{SugawaST}. %We also use the following lemma: %in the proof of the main theorem.
\end{remark}

\begin{proof}
Let $p(z)=P_f(z)=zf'(z)/f(z)$ and $f$ belong to the class $\es^*(\alpha,\beta)$.
Then, by the definition, $p(z)$ is subordinate to the univalent function
$$q(z)=\left(\frac{1+(1-2\beta)z}{1-z}\right)^\alpha, \quad z\in\D ,
$$
and therefore, there exists an analytic function $\omega:\,\D\ra\D$ with
$\omega(0)=0$ such that
\begin{equation}\label{p3eq4}
p=q\circ \omega =
\left(\frac{1+(1-2\beta)\omega}{1-\omega}\right)^\alpha.
\end{equation}
Let $F\in\A$ be the function with $P_F=q$, i.e.
$$F(z)=z\exp\left(\int_0^z\frac{q(t)-1}{t}\,dt\right).
$$
We split the proof into two cases. Assume first that $0\le \beta\le 1/2$.
Logarithmic differentiation of (\ref{p3eq4}) yields that
$$1+ \frac{zf''}{f'}-\frac{zf'}{f}=\frac{2\alpha(1-\beta)z\omega'}
{(1-\omega)(1+(1-2\beta)\omega)}.
$$
We thus have
\begin{equation}\label{p3eq5}
T_f(z)=\frac{2\alpha(1-\beta)\omega' (z)}
{(1-\omega (z))(1+(1-2\beta)\omega (z))}
+\frac{p(z)-1}{z}.
\end{equation}
By triangle inequality and Schwarz-Pick lemma, we obtain
\begin{eqnarray*}
|T_f(z)|
& \le & \frac{2\alpha(1-\beta)|\omega'(z)|}{|1-2\beta\omega (z)-(1-2\beta)\omega^2(z)|}
+ \frac{|p(z)-1|}{|z|}\\
& \le & \frac{2\alpha(1-\beta)(1-|\omega (z)|^2)}{(1-|z|^2)(|1-2\beta\omega (z)|-
(1-2\beta)|\omega (z)|^2)}+\frac{|q(\omega (z))-1|}{|z|}\\
& \le & \frac{2\alpha(1-\beta)(1-|\omega (z)|^2)}{(1-|z|^2)(1-2\beta|\omega (z)|-
(1-2\beta)|\omega (z)|^2)}+\frac{|q(\omega (z))-1|}{|z|}\\
& \le & \frac{2\alpha(1-\beta)(1+|\omega(z)|)}{(1-|z|^2)(1+(1-2\beta)|\omega(z)|)}
+\frac{|q(\omega(z))-1|}{|z|}.
\end{eqnarray*}
Using a similar argument, namely the triangle inequality (as we did in the
denominator above), we see that
\begin{eqnarray*}
|q(z)-1| & = & \left|\int_0^z q'(t)\, dt \right |\\
& = & \left|\int_0^z\left(\frac{1+(1-2\beta)t}{1-t}\right)^\alpha
\frac{2\alpha(1-\beta)}{(1-t)(1+(1-2\beta)t)}\,dt\right|\\
& \le & \int_0^{|z|}\left(\frac{1+(1-2\beta)t}{1-t}\right)^\alpha
\frac{2\alpha(1-\beta)}{(1-t)(1+(1-2\beta)t)}\,dt\\
& = & q(|z|)-1.
\end{eqnarray*}
So, using this inequality and the fact $|\omega(z)|\le |z|$, we get
\begin{eqnarray*}
|T_f(z)| & \le &
\frac{2\alpha(1-\beta)(1+|\omega (z)|)}{(1-|z|^2)(1+(1-2\beta)|\omega (z)|)}
+\frac{q(|\omega (z)|)-1}{|z|}\\
& \le & \frac{2\alpha(1-\beta)(1+|z|)}{(1-|z|^2)(1+(1-2\beta)|z|)}
+\frac{q(|z|)-1}{|z|}\\
& = & T_F(|z|),
\end{eqnarray*}
where the second inequality is strict provided $\omega (z)/z$ is not a
unimodular constant. Therefore, we see that $\|f\|\le\|F\|$.

Since
$$(1-t^2)T_F(t)=\frac{2\alpha(1-\beta)(1+t)}{1+(1-2\beta)t}+
\frac{1-t^2}{t}(q(t)-1)\ra 2\alpha \mbox{ as $t\ra 1^-$},
$$
the equality $\|f\|=\|F\|$ holds only if $|T_f(z_0)|=T_F(|z_0|)$ for some
$z_0\in\D$. Hence we conclude that equality holds if $P_f(z)=q(\mu z)$
for some unimodular constant $\mu$.

We next consider the case $1/2\le\beta<1$. If we use triangle
inequality again without multiplying the factors in the denominator,
we obtain
$$|q(z)-1|\le q(|z|)-1.$$
%\beqq
%|q(z)-1|
%& = & \left|\int_0^z\frac{2\alpha(1-\beta)}{(1-t)^{1+\alpha}
%(1+(1-2\beta)t)^{1-\alpha}}\,dt\right|\\
%& \le & \int_0^{|z|}\frac{2\alpha(1-\beta)}{(1-t)^{1+\alpha}
%(1+(1-2\beta)t)^{1-\alpha}}\,dt\\
%& = & q(|z|)-1.
%\eeqq
Now using the same argument as in the
first case, we get
\begin{eqnarray*}
(1-|z|^2)|T_f(z)|
& \le & \frac{2\alpha(1-\beta)(1-|\omega^2 (z)|)}{|1-\omega (z)|\,|1+(1-2\beta)
\omega (z)|}
        +\frac{1-|z|^2}{|z|}(q(|\omega (z)|)-1)\\
& \le & \frac{2\alpha(1-\beta)(1+|\omega (z)|)}{1+(1-2\beta)|\omega (z)|}
        +\frac{1-|z|^2}{|z|}(q(|\omega (z)|)-1)\\
& \le & \frac{2\alpha(1-\beta)(1+|z|)}{1+(1-2\beta)|z|}
        +\frac{1-|z|^2}{|z|}(q(|z|)-1)\\
& = &  (1-|z|^2)T_F(|z|).
\end{eqnarray*}
This shows that $\|f\|\le \|F\|$ and the inequality is sharp
(as in the argument of the previous case). Thus, it is enough to compute $\|F\|$.
Now, we write
$$ L(\alpha,\beta)  =  \sup_{0<t<1}\frac{1-t^2}{t}(q(t)-1)
 = \sup_{x>1}g(x),
$$
where
$$ g(x)=\frac{4(1-\beta)(x-\beta)(x^\alpha-1)}{(x-1)(x+1-2\beta)}
$$
with the substitution $x=[1+(1-2\beta)t]/(1-t)$. Logarithmic derivative
of $g(x)$ yields
$$\frac{g'(x)}{g(x)}= -\frac{h(x)}{(x-\beta)(x^\alpha-1)(x-1)(x+1-2\beta)},
$$
where $h(x)$ is given by
\begin{eqnarray*}
h(x)& = & (1-\alpha)x^{\alpha+2}+\beta(3\alpha-2)x^{\alpha+1}+
[(1+\alpha)(1-2\beta)+2\beta^2(1-\alpha)]x^\alpha \\
&& -\alpha\beta(1-2\beta)x^{\alpha-1}
-x^2+2\beta x-(1-\beta)^2-\beta^2.
\end{eqnarray*}
Differentiations give easily the following:
\begin{eqnarray*} h'(x) \hspace{-.26cm}& = &\hspace{-.26cm}
(1-\alpha)(\alpha+2)x^{\alpha+1}+\beta(3\alpha-2)(\alpha+1)x^{\alpha}+
\alpha[(1+\alpha)(1-2\beta)+2\beta^2(1-\alpha)]
x^{\alpha-1}\\&& -\alpha\beta(\alpha-1)(1-2\beta)x^{\alpha-2}-2x+2\beta\\
h''(x) \hspace{-.26cm}& = &\hspace{-.26cm}(1-\alpha)(\alpha+2)
(\alpha+1)x^{\alpha}+\alpha\beta(3\alpha-2)
(\alpha+1)x^{\alpha-1}+\alpha(\alpha-1)[(1+\alpha)(1-2\beta)\\
&& +2\beta^2(1-\alpha)]x^{\alpha-2} -\alpha\beta(\alpha-1)(\alpha-2)(1-2\beta)x^{\alpha-3}-2\\
h'''(x) \hspace{-.26cm}& = &\hspace{-.26cm} (1-\alpha)(\alpha+1)(\alpha+2)\alpha x^{\alpha-1}+\alpha\beta(3\alpha-2)
(\alpha+1)(\alpha-1)x^{\alpha-2}+\alpha(\alpha-1)(\alpha-2)\\
&& [(1+\alpha)(1-2\beta)
+2\beta^2(1-\alpha)]x^{\alpha-3} -\alpha\beta(\alpha-1)(\alpha-2)(\alpha-3)
(1-2\beta)x^{\alpha-4}\\
& = &\hspace{-.26cm} \alpha(1-\alpha)x^{\alpha-4} \phi (x)
%[(\alpha+1)(\alpha+2)x^3-
%\beta(3\alpha-2)(\alpha+1)x^2-(\alpha-2)[(1+\alpha)(1-2\beta)\\ &&
%+2\beta^2(1-\alpha)]x
%+\beta(1-2\beta)(\alpha-2)(\alpha-3)].
\end{eqnarray*}
where
\begin{eqnarray*}
\phi(x)\hspace{-.16cm}& = &\hspace{-.16cm} (\alpha+1)(\alpha+2)x^3-
\beta(3\alpha-2)(\alpha+1)x^2-(\alpha-2)[(1+\alpha)(1-2\beta)+2\beta^2(1-\alpha)]x\\
&&+\beta(1-2\beta)(\alpha-2)(\alpha-3).
\end{eqnarray*}
It follows that
$$
\phi'(x) =
%\hspace{-.26cm}& = & \hspace{-.26cm}
3(\alpha+1)(\alpha+2)x^2+2\beta(2-3\alpha)(1+\alpha)x
+(2-\alpha)[(1+\alpha)(1-2\beta)+2\beta^2(1-\alpha)]
$$
and
$$\phi''(x)=
%\hspace{-.26cm}& = & \hspace{-.26cm}
6(\alpha+1)(\alpha+2)x+2\beta(2-3\alpha)(1+\alpha).
$$
Since
$\phi'''(x)=6(\alpha+1)(\alpha+2)>0$, $\phi''(x)$ is increasing for
all $x>1$. So we have
$$\phi''(x)\ge \phi''(1)=6\alpha^2(1-\beta)
+16\alpha+12+4\beta+2\alpha(1-\beta)>0.
$$
This implies that $\phi'(x)$ is
increasing  for $x>1$ and so
$$\phi'(x)\ge \phi'(1)=2(1+\alpha)
(\alpha+2+2(1-\alpha\beta))+2\beta^2(1-\alpha)(2-\alpha)>0.
$$
So $\phi(x)$ is also increasing for $x>1$ and hence,
$$\phi(x)\ge \phi(1) =4(1-\beta)(1+\alpha+\beta+\beta(1-\alpha))>0.
$$
Therefore,
$h'''(x)>0$ and so $h''(x)$ is increasing for $x>1$. Since
$h''(x)$ is increasing in $(1,\infty)$ and
$$h''(1)=-2\alpha(1-\beta)[\alpha(1-\beta)+\beta]<0,
$$
we see that $h''(x)$ has a unique zero in $(1,\infty)$, say $x=x_1$. Since
$h'(1)=0$ and $h'(x)$ is increasing on $(x_1,\infty)$ and
decreasing on $(1,x_1)$, we obtain that $h'(x)$ has a unique zero, say $x_2$
($x_2>x_1$) in $(1,\infty)$. Since $h(1)=0$, by the same argument
we conclude that $h(x)$ has a
unique zero, say $k=k(\alpha,\beta)>x_2$ in $(1,\infty)$. Thus
$h(x)<0$ in $(1,k)$ and $h(x)>0$ in $(k,\infty)$, equivalently,
$g'(x)$ is positive for $x\in(1,k)$ and negative for $x>k$. This
shows that $g(x)$ assumes its maximum at $x=k$ and hence we have
(\ref{p3eq2}). Since $k$ is the zero of $h(x)$, it is the unique
solution of the equation (\ref{p3eq3}). Thus we have established (\ref
{p3eq1}).
\end{proof}

\begin{remark}
Here we calculate some bounds for $L(\alpha,\beta)$ and
$k(\alpha,\beta)$ although these are not better estimates. Since
$g(x)$ attains its maximum at $k>1$, we note that
$$L(\alpha,\beta)=g(k)>\lim_{x\ra 1^{+}}g(x)=2\alpha(1-\beta).
$$
%This implies that $L(\alpha,\beta)-2\alpha(1-\beta)>0$ and hence
%$k<[2\beta(1-\alpha)+\alpha]/(2-\alpha)$, since $0<\alpha<1$ and
%$k>1$.
%Finally we observe that, for $\beta=0$, $g(x)$ satisfies the
%second order differential equation
%$$
%(x^2-1)x^2g''(x)+[1+\alpha+(3-\alpha)x^2]xg'(x)-[1+\alpha-(1-\alpha)x^2]g(x)=0.
%$$
Finally we observe that,  $g(x)$ satisfies the
second order differential equation
$$A(x)g''(x)+B(x)g'(x)+C(x)g(x)=0
$$
where
\begin{eqnarray*}
A(x)& =& x(x-1)(x+1-2\beta)(x-\beta)^2\\
B(x)& =&4x(x-\beta)^3+(1-\alpha)(x-1)(x+1-2\beta) (x-\beta)^2\\
& &~~~~~~~~~~~ -2x(x-1)(x+1-2\beta)(x-\beta)\\
C(x)& =& 2(1-\alpha )(x-\beta)^3-2x(x-\beta)^2
-(1-\alpha)(x-1)(x+1-2\beta)(x-\beta) \\
& &~~~~~~~~~~~ +2x(x-1)(x+1-2\beta).
\end{eqnarray*}
%Our feeling is that, in the general case also a similar differential
%equation satisfied by $g(x)$ can be obtained,
This observation is perhaps to justify its close connection between these
bounds and special functions.
\end{remark}

% \begin{remark}
% If we consider $$\sup_{0<t<1}\left[\frac{2\alpha(1-\beta)(1+t)}{1+(1-2\beta)t}+
% \frac{1-t^2}{t}(q(t)-1)\right]
% $$
% instead simply for the quantity $\sup_{0<t<1}\frac{1-t^2}{t}(q(t)-1)$, one can
% obtain the sharp inequality in (\ref{p3eq1}). This may be complicated
% due to lengthy expression.
% \end{remark}

%%%%%%%%%%%%%%%%%%%%%%%%%%%%%%%%%%%%%%%%%%%%%
%%%%%%%%%%%%%%%%%%%%%%%%%%%%%%%%%%%%%%%%%%%%%
%%%%%%%%%%%%%%%% SECTION 4 %%%%%%%%%%%%%%%%%%
%%%%%%%%%%%%%%%%%%%%%%%%%%%%%%%%%%%%%%%%%%%%%
%%%%%%%%%%%%%%%%%%%%%%%%%%%%%%%%%%%%%%%%%%%%%

\section{Concluding Remarks}\label{chap6-sec4}

Let $\beta,\gamma,A$ and $B$ be real numbers and suppose that $\beta>0$,
$\beta+\gamma>0$, $-1\le B<1$ and $B<A\le 1+\gamma(1-B)\beta^{-1}$.  For $f\in
\es^*(A,B)$, we consider $g=J_{\beta,\gamma}[f]$ defined by

\begin{eqnarray}\label{gop}
g(z)=J_{\beta,\gamma}[f](z)=\left[\frac{\beta+\gamma}{z^\gamma}\int_0^z
t^{\gamma-1}f^\beta(t)\,dt\right]^{1/\beta},\quad z\in\D.
\end{eqnarray}
Moreover, we define the order of (univalent) starlikeness of the class
$J_{\beta,\gamma}[\es^*(A,B)]$ by the largest number
$\delta=\delta(A,B;\beta,\gamma)$ such that
$$J_{\beta,\gamma}[\es^*(A,B)]\subset\es^*(\delta).
$$
Before we propose a general problem, we recall a special case of a result
from \cite{Pon88}.

\begin{lemma}\label{lem:pon}
Let $\beta>0,$ $\beta+\gamma>0$ and consider the integral operator defined by
{\rm (\ref{gop}).}
\begin{enumerate}
\item[(a)] If $-1\le B<1$ and $B<A\le 1+\gamma(1-B)\beta^{-1},$ then the order
of $($univalent$)$ starlikeness of $J_{\beta,\gamma}[\es^*(A,B)]$ is given by
$$\delta(A,B;\beta,\gamma)=\inf_{|z|<1}{\rm Re}\,q(z),
$$
where $q$ is given by
$$q(z)=\frac{1}{\beta Q(z)}-\frac{\gamma}{\beta}
$$
with
$$Q(z)=\left\{
\begin{array}{ll}\displaystyle  \int_0^1\left(\displaystyle\frac{1+Bzt}{1+Bz}\right)
^{\beta((A-B)/B)}t^{\beta+\gamma-1}\,dt &\mbox{ if $B\neq 0$, }\\
\displaystyle  \int_0^1t^{\beta+\gamma-1}\exp(\beta Az(t-1))\,dt &\mbox{ if $B=0$ }
\end{array}\right.
$$
and
$$q(z)=\frac{\beta-\gamma Bz}{\beta(1+Bz)}\quad \mbox{ when } A=-
\frac{(\gamma+1)B}{\beta},\quad B\neq 0.
$$
\item[(b)] Moreover, if $-1\le B<0,$ $B<A\le \min\{1+\gamma(1-B)\beta^{-1},
-(\gamma+1)B\beta^{-1}\},$ then
%for $f\in\es^*(A,B),$ we have
\begin{eqnarray}\label{delta1}
\delta(A,B;\beta,\gamma)=q(-1)=\frac{1}{\beta}\left[\frac{\beta+\gamma}
{F(1,\beta(\frac{B-A}{B});\beta+\gamma+1;\frac{-B}{1-B})}-\gamma\right].
\end{eqnarray}
\item[(c)] Furthermore, if $0<B<1$, $B<A\le \min\{1+\gamma(1-B)\beta^{-1},
(2\beta+\gamma+1)B\beta^{-1}\}$, then
%for $f\in\es^*(A,B),$ we have
\begin{equation}\label{delta2}
\delta(A,B;\beta,\gamma)=q(1)=\frac{1}{\beta}\left[\frac{\beta+\gamma}
{F(1,\beta(\frac{A-B}{B});\beta+\gamma+1;\frac{B}{1+B})}-\gamma\right].
\end{equation}
%The results are all sharp.
\end{enumerate}
\end{lemma}

Under the hypotheses of Lemma \ref{lem:pon}, when $f\in\es^*(A,B),$ we get by
\cite[Theorem 2]{Yam97}
$$\|J_{\beta,\gamma}[f]\|\le 6-4\delta,
$$
where $\delta$ is given either by (\ref{delta1}) or (\ref{delta2}) with the
corresponding conditions.

As a special case, we mention the following: if $f\in\es^*(\alpha)$ and
$\beta,\gamma$ are real numbers such that $\beta>0$, $\beta+\gamma>0$ and
$$\max \left \{0,-\frac{\gamma}{\beta}, \frac{\beta -\gamma -1}{2\beta}\right
\}\le\alpha<1,
$$
then $J_{\beta,\gamma}[f]$ defined by (\ref{gop})
is in $\es^*(\delta),$ where
\begin{eqnarray}\label{dabg}
\delta=\delta(\alpha,\beta,\gamma)=\frac{1}{\beta}\left[\frac{\beta+\gamma}
{F(1,2\beta(1-\alpha);\beta+\gamma+1;1/2)}-\gamma\right].
\end{eqnarray}
Consequently, by \cite[Theorem 2]{Yam97}, we have the estimate
$$\|J_{\beta,\gamma}[f]\|\le 6-4\delta,
$$
where $\delta$ is given by (\ref{dabg}).

In particular, for $f\in\es^*(\alpha)$ and $\max\{0,-\gamma \}\le\alpha<1,$ we have
$B_\gamma [f]\in\es^*(\delta(\alpha,\gamma )),$ where
\begin{eqnarray}\label{del}
\delta=\delta(\alpha,\gamma )=
\frac{\gamma +1}{F(1,2(1-\alpha);\gamma +2;1/2)}-\gamma .
\end{eqnarray}
Thus, we have
$$\|B_\gamma [f]\|\le6-4\delta,
$$
where $\delta$ is given by (\ref{del}).
Consequently, the following result gives a norm estimate for the Bernardi
integral transform of functions that are not necessarily univalent.

\begin{corollary}
Let $\gamma >-1$ and $f\in\es^*(-\gamma )$. Then
$$\|B_\gamma [f]\|\le 6-4 \left[
%\frac{1}{B(1/2,1+\gamma )}
\frac{\Gamma(\frac{3}{2}+\gamma )}{\sqrt{\pi}\,\Gamma (1+\gamma )
}
-\gamma \right].
$$
%where $B(1/2,1+\gamma )$ denotes the beta function.
\end{corollary}
\begin{proof}
Recall the well-known identity (see \cite[p. 69]{Rain60})
$$F(2a,2b;a+b+1/2;1/2)=\frac{\Gamma(a+b+\frac{1}{2})
\Gamma(\frac{1}{2})}{\Gamma(a+\frac{1}{2})\Gamma(b+\frac{1}{2})}.
$$
Choose $a=1/2,b=1-\alpha$ and $\alpha=-\gamma $. Then (\ref{del}) yields
$$\delta(\gamma )=\delta(-\gamma ,\gamma )
=-\gamma +\frac{\Gamma(\frac{3}{2}+\gamma )}
{\Gamma(1+\gamma )\Gamma(\frac{1}{2})}
$$
which may be written in terms of beta function
given by
$$\delta(\gamma )=-\gamma + \frac{1}{B(1/2,1+\gamma )}.
$$
%\begin{equation}\label{delta3}
%$$B(p,q)=\int_0^1t^{p-1}(1-t)^{q-1}\,dt.
%$$
Thus, for $f\in\es^*(-\gamma )$ we notice that $B_\gamma
[f]\in\es^*(\delta(\gamma ))$. Therefore, we have
$$\|B_\gamma [f]\|\le 6-4\delta (\gamma )
$$
and the conclusion follows.
\end{proof}

\begin{problem} Find the sharp norm estimate for $B_\gamma [f]$ when
$f\in\es^*(-\gamma )$. More generally, find a sharp norm estimate for
$J_{\beta, \gamma} [f]$ whenever $f\in\es^*(\alpha )$, $\alpha <1$.
\end{problem}

A number of problems of this type may be raised for various integral
transforms. For example, there exist conditions on $\lambda (t)$ and
subfamilies $\mathcal F$ of $\mathcal A$ such that the integral transform of
the form
$$ V_{\lambda} [f](z) = \int_{0}^{1} \lambda(t) \frac{f(t z)}{t}\,dt
\quad \quad (f\in {\mathcal F})
$$
is close-to-convex or starlike or convex, respectively (see \cite{FR,PR,KR}
for details). In view of this, one  can ask for the norm estimate for
$ V_{\lambda} [f]$ when $f$ runs over suitable subclasses ${\mathcal F}$ of $\A$.
We remark that for the choice $\lambda (t)=(1+\gamma)t^{\gamma}$ ($\gamma >-1$),
$V_{\lambda} [f](z)$ reduces to the  Bernardi transform of $f$.

% \vskip 1cm
% This chapter has been organized from the article\\
% {\bf Ponnusamy, S. and Sahoo, S.K.}
% Norm estimates for convolution transforms of certain classes of
% analytic functions. Preprint.

% I thank Professor Toshiyuki Sugawa for several
% helpful discussions on his work \cite{SugawaST} in this direction.
% Also, I thank Professor R. Parvatham for useful discussion on this
% topic and for bringing the last problem of this chapter to our attention.

\newpage
\centerline{\large{\bf LIST OF PAPERS BASED ON THE THESIS}}

\newpage
\begin{center}
{\large\bf CURRICULUM VITAE}
\end{center}

\vskip .5cm
\begin{tabular}{lll}
{\bf 1. NAME} &:& Swadesh Kumar Sahoo\\
{\bf 2. DATE OF BIRTH} &:& 05 February 1978\\
{\bf 3. EDUCATIONAL QUALIFICATIONS}\\
&&\\
\hspace{.6cm}
{\bf 1998 \quad Bachelor of Arts (B.A.)}\\
\hspace{3cm}
Institution &:& F.M. College, Balasore\\
\hspace{3cm}
Specialization &:& Mathematics\\
&&\\
\hspace{.6cm}
{\bf 2000 \quad Master of Arts (M.A.)}\\
\hspace{3cm}
Institution &:& Utkal University, Bhubaneswar\\
\hspace{3cm}
Specialization &:& Mathematics\\
&&\\
\hspace{.6cm}
{\bf Doctor of Philosophy (Ph.D.)}\\
\hspace{3cm}
Institution &:& IIT Madras, Chennai\\
\hspace{3cm}
Registration Date &:& 30-07-2002
\end{tabular}

\newpage
\begin{center}
{\large\bf DOCTORAL COMMITTEE}
\end{center}

\vskip .5cm
\begin{tabular}{lllll}
{\bf CHAIRPERSON:} &&&& Dr. S.A. Choudum\\
&&&& Professor and Head\\
&&&& Department of Mathematics
\end{tabular}

\begin{tabular}{llllllllll}
{\bf GUIDE(S):} &&&&&&&& Dr. S. Ponnusamy\\
&&&&&&&& Professor\\
&&&&&&&& Department of Mathematics
\end{tabular}

\begin{tabular}{lllllllll}
{\bf MEMBERS:} &&&&&&& Dr. Arindama Singh (Departmental member)\\
&&&&&&& Professor\\
&&&&&&& Department of Mathematics\\
&&&&&&&\\
&&&&&&& Dr. V. Vetrivel (Departmental member)\\
&&&&&&& Professor\\
&&&&&&& Department of Mathematics\\
&&&&&&&\\
&&&&&&& Dr. Anju Chadha (member from other Department)\\
&&&&&&& Professor\\
&&&&&&& Department of Biotechnology\\
&&&&&&&\\
&&&&&&& Dr. M. Kamaraj (member from other Department)\\
&&&&&&& Associate Professor\\
&&&&&&& Department of Metallurgical and Materials Engineering
\end{tabular}

\newpage
\input{swad-thesis.ind}
\end{document}